\documentclass[reqno]{amsart}

\usepackage{hyperref} 

\usepackage[mathscr]{euscript}
\usepackage{amsmath}
\usepackage{amssymb}
\usepackage{enumerate}

\usepackage{graphicx}
\usepackage{verbatim}
\usepackage{bm}
	
\usepackage{slashed}
\usepackage{mathrsfs}

\newtheorem{MainThm}{Theorem}

\theoremstyle{plain}

\newtheorem{thm}[equation]{Theorem}
\newtheorem{cor}[equation]{Corollary}
\newtheorem{lem}[equation]{Lemma}
\newtheorem{prop}[equation]{Proposition}

\theoremstyle{definition}
\newtheorem{definition}[equation]{Definition}

\newtheorem{definition-proposition}[equation]{Definition/Proposition}
\newtheorem{assumptions}[equation]{Assumptions}

\newtheorem{defn}[equation]{Definition}

\theoremstyle{remark}
\newtheorem{remark}[equation]{Remark}

\newtheorem{rem}[equation]{Remark}

\numberwithin{equation}{subsection}

\newcommand{\bC}{\mathbb{C}}

\newcommand{\bF}{\mathbb{F}}

\newcommand{\bI}{\mathbb{I}}

\newcommand{\bK}{\mathbb{K}}

\newcommand{\bN}{\mathbb{N}}

\newcommand{\bP}{\mathbb{P}}
\newcommand{\bQ}{\mathbb{Q}}
\newcommand{\bR}{\mathbb{R}}

\newcommand{\bZ}{\mathbb{Z}}

\newcommand{\gB}{\bold{B}}

\newcommand{\gD}{\bold{D}}

\newcommand{\gR}{\bold{R}}

\newcommand{\gT}{\bold{T}}

\newcommand{\gX}{\bold{X}}

\newcommand{\gp}{\mathbf{p}}
\newcommand{\gb}{\mathbf{b}}

\newcommand{\cC}{\mathcal{C}}

\newcommand{\cG}{\mathrm{G}}

\newcommand{\cR}{\mathcal{R}}

\newcommand{\fb}{\mathfrak{b}}

\newcommand{\fk}{\mathfrak{k}}

\newcommand{\abs}{\operatorname{abs}}
\newcommand{\ahat}{\hat{\mathscr{A}}}
\newcommand{\Ahat}{\hat{\mathscr{A}}}

\newcommand{\Aut}{\operatorname{Aut}}

\newcommand{\bott}{\operatorname{bott}}
\newcommand{\Bun}{\mathrm{Bun}}

\newcommand{\Cl}{\mathbf{Cl}}
\newcommand{\Cyl}{\mathrm{Cyl}}

\newcommand{\dtor}{\mathrm{dtor}}

\newcommand{\Dir}{\slashed{\mathfrak{D}}}

\newcommand{\Fred}{\operatorname{Fred}}

\newcommand{\GL}{\mathrm{GL}}

\newcommand{\hofib}{\operatorname{hofib}}

\newcommand{\im}{\operatorname{Im}}
\newcommand{\indx}[1]{\operatorname{Dir}(#1)}
\newcommand{\ind}[1]{\operatorname{ind}(#1)}
\newcommand{\id}{\operatorname{id}}

\newcommand{\inddiff}{\operatorname{inddiff}} 
\newcommand{\inddiffaps}{\operatorname{inddiff}^{\mathrm{GL}}} 

\newcommand{\bas}{\mathrm{bas}}
\newcommand{\KO}{\mathrm{KO}}
\newcommand{\KU}{\mathrm{KU}}
\newcommand{\ko}{\mathrm{ko}}
\newcommand{\ku}{\mathrm{ku}}
\newcommand{\Kom}{\mathbf{Kom}}
\newcommand{\Lin}{\mathbf{Lin}}

\newcommand{\loopinf}[1]{\Omega^{\infty #1}}

\newcommand{\MT}{\mathrm{MT}}
\newcommand{\MTSpin}{\mathrm{MTSpin}}
\newcommand{\MSpin}{\mathrm{MSpin}}
\newcommand{\MTSO}{\mathrm{MTSO}}
\newcommand{\map}{\operatorname{map}}

\newcommand{\normalize}[1]{\frac{#1}{(1+{#1}^{2})^{1/2}}}

\newcommand{\smb}{\mathrm{smb}}

\newcommand{\Path}{\mathcal{P}}

\newcommand{\ph}{\mathrm{ph}}

\newcommand{\Riem}{\cR}
\newcommand{\res}{\mathrm{res}}

\newcommand{\Spin}{\mathrm{Spin}}
\newcommand{\supp}{\operatorname{supp}}
\newcommand{\Sym}{\mathrm{Sym}}

\newcommand{\spinor}{\slashed{\mathfrak{S}}}
\newcommand{\scal}{\mathrm{scal}}

\newcommand{\trg}{\operatorname{trg}}
\newcommand{\Th}{\mathrm{Th}}
\newcommand{\tor}{\mathrm{tor}}

\newcommand\lra{\longrightarrow}

\newcommand\Diff{\mathrm{Diff}}
\newcommand\Emb{\mathrm{Emb}}

\newcommand\hocolim{\operatorname{hocolim}}
\newcommand\colim{\operatorname{colim}}

\newcommand{\round}{\circ}
\newcommand{\ch}{\mathrm{ch}}

\newcommand\hcoker{/\!\!/}

\newcommand{\hemi}{\mathrm{hemi}}
\newcommand{\inter}{\mathrm{int}}

\newcommand{\Sing}{\mathrm{Sing}}

\newcommand{\Gr}{\mathrm{Gr}}

\newcommand{\CircNum}[1]{\ooalign{\hfil\raise .00ex\hbox{\scriptsize #1}\hfil\crcr\mathhexbox20D}}

\newcommand{\hq}{/\!\!/}

\newcommand{\SE}{\mathrm{SE}}

\newcommand{\kothom}[1]{\lambda_{-#1}}

\newcommand\mnote[1]{\marginpar{\tiny #1}}

\usepackage[all]{xy}
\SelectTips{cm}{10}
\CompileMatrices

\title[Infinite loop spaces and positive scalar curvature]{Infinite loop spaces\\ and positive scalar curvature}

\author{Boris Botvinnik}
\email{botvinn@math.uoregon.edu}
\address{
Department of Mathematics\\
University of Oregon\\
Eugene OR\\
97403, U.S.A.
}

\author{Johannes Ebert}
\thanks{J. Ebert was partially supported by the SFB 878}
\email{johannes.ebert@uni-muenster.de}
\address{
Mathematisches Institut der Westf{\"a}lischen Wilhelms-Universit{\"a}t M{\"u}nster\\
Einsteinstr. 62\\
DE-48149 M{\"u}nster\\
Germany
}

\author{Oscar Randal-Williams}
\thanks{O.\ Randal-Williams acknowledges Herchel Smith Fellowship support.}
\email{o.randal-williams@dpmms.cam.ac.uk}
\address{
Centre for Mathematical Sciences\\
Wilberforce Road\\
Cambridge CB3 0WB\\
UK}

\date{\today}

\begin{document}

\begin{abstract}
We study the homotopy type of the space of metrics of positive
scalar curvature on high-dimensional compact spin manifolds. Hitchin used the fact that there are no harmonic spinors on a manifold with positive scalar curvature to construct a secondary index map from the space of positive scalar metrics to a suitable space from the real $K$-theory spectrum. Our main results concern the nontriviality of this map.
We prove that for $2n \geq 6$, the natural $KO$-orientation from the infinite loop space of the Madsen--Tillmann--Weiss spectrum factors (up to homotopy) through the space of metrics of positive scalar curvature on any $2n$-dimensional spin manifold. For manifolds of odd dimension $2n+1 \geq 7$, we prove the existence of a similar factorisation.

When combined with computational methods from homotopy theory, these results have strong implications. For example, the secondary index map is surjective on all rational homotopy groups. We also present more refined calculations concerning integral homotopy groups.

To prove our results we use three major sets of technical tools and
results. The first set of tools comes from Riemannian geometry: we use a
parameterised version of the Gromov--Lawson surgery technique which allows
us to apply homotopy-theoretic techniques to spaces of metrics of positive scalar curvature. Secondly, we relate Hitchin's secondary index to several other index-theoretical results, such as the Atiyah--Singer family index theorem, the additivity theorem for indices on noncompact manifolds and the spectral flow index theorem. 
Finally, we use the results and tools developed recently in the study of moduli spaces of manifolds and cobordism categories. The key new ingredient we use in this paper is the high-dimensional analogue of the Madsen--Weiss theorem, proven by Galatius and the third named author.
\end{abstract}

\maketitle

\clearpage

\tableofcontents

\clearpage
\section{Introduction}

\subsection{Statement of results}

Among the several curvature conditions one can put on a Riemannian metric, the condition of \emph{positive scalar curvature} (hereafter, \emph{psc}) has the richest connection to topology, and in particular to cobordism theory. This strong link is provided by two fundamental facts.

The first comes from index theory. Let $\Dir_g$ be the Atiyah--Singer Dirac
operator on a Riemannian spin manifold $(W,g)$ of dimension $d$. The scalar curvature appears in the remainder term in the Weitzenb{\"o}ck formula (or, more appropriately, Schr{\"o}dinger--Lichnerowicz formula, cf.\ \cite{Schroed} and \cite{Lich}). 
This forces the Dirac operator $\Dir_g$ to be invertible provided that $g$ has positive scalar
curvature, and hence forces the index of $\Dir_g$ to vanish. The index of $\Dir_g$ is an element of the real $K$-theory group $KO^{-d}(*)=KO_d (*)$ of a point, due to the Clifford symmetries of $\Dir_g$. 
The Atiyah--Singer index theorem in turn equates $\ind{\Dir_g}$ with the
$\hat{\mathscr{A}}$-invariant of $W$, an element in $KO^{-d}(*)$ which can be defined in homotopy-theoretic terms through the Pontrjagin--Thom construction. 
Thus if $W$ has a psc metric, then $\hat{\mathscr{A}}(W)=0$.

The second fundamental fact, due to Gromov and Lawson \cite{GL2}, is that if a
manifold $W$ with a psc metric is altered by a suitable surgery to a
manifold $W'$, then $W'$ again carries a psc metric. These results
interact extremely well provided the manifolds are spin, simply-connected and of dimension at least five. Under these circumstances, the question of whether $W$ admits a psc metric depends only on the spin cobordism class of $W$, which reduces it to a problem in stable homotopy theory. Stolz \cite{Stolz} managed to solve this problem, and thereby showed that such manifolds admit a psc metric precisely if their $\hat{\mathscr{A}}$-invariant vanishes. Much work has been undertaken to relax these three hypotheses; see \cite{Rosenberg} and \cite{SchickICM} for surveys.

Rather than the existence question, we are interested in understanding the
\emph{homotopy type} of the space $\Riem^+ (W)$ of all psc metrics on a manifold
$W$. Our method requires to consider manifolds $W$ with boundary $\partial W$. In that case, we choose a collar $[-\epsilon,0] \times \partial W \subset W$ and consider the space $\Riem^+ (W)_h$ of all psc metrics on $W$ which are of the form $dt^2 + h$ on this collar, for a fixed psc metric $h$ on $\partial W$.

To state our results, we recall Hitchin's definition of a secondary
index-theoretic invariant for psc metrics \cite{HitchinSpin}, which we shall call the \emph{index difference}.
Ignoring some technical details for now, the definition is as follows.
For a closed spin $d$-manifold $W$ choose a basepoint psc metric $g_0 \in \cR^+(W)$, so for another psc metric $g$ we can form the path of metrics $g_t = (1-t) \cdot g + t \cdot g_0$ for $t \in [0,1]$. There is an associated path of Dirac operators in the space $\Fred^d$ of $\Cl^d$-linear self-adjoint odd Fredholm operators on a Hilbert space, and it starts and ends in the subspace of invertible operators, which is contractible. As the space $\Fred^d$ represents $KO^{-d}(-)$, we obtain an element $\inddiff_{g_0} (g) \in KO^{-d}([0,1], \{0,1\}) = KO^{-d-1}(*) = KO_{d+1}(*)$. This construction generalises to manifolds with boundary and to families, and gives a well-defined homotopy class of maps
$$\inddiff_{g_0} : \Riem^+ (W)_h \lra \loopinf{+d+1}\KO$$
to the infinite loop space which represents real $K$-theory. The main theorems of this paper are stated as Theorems \ref{factorization-theorem:even-case} and \ref{factorization-theorem:odd-case}. They involve the construction of maps $\rho: X \to \Riem^+(W)_h$ from certain infinite loop spaces $X$, and the identification of the composition $\inddiff_{g_0} \circ \rho$ with a well-known infinite loop map. Using standard methods from homotopy theory, one can then derive consequences concerning the induced map on homotopy groups,
\begin{equation*}
A_{k}(W,g_0): \pi_k (\Riem^+ (W)_h, g_0) \lra KO_{k+d+1}(*) = 
\begin{cases}
\bZ & k+d+1 \equiv 0 \pmod 4\\
\bZ/2 & k+d+1 \equiv 1 ,2 \pmod 8\\
0 & \text{else},
\end{cases}
\end{equation*}
and we state these consequences first. 

\begin{MainThm}\label{thm:intro:htpyGps}
Let $W$ be a spin manifold of dimension $d \geq 6$, and fix $h \in \Riem^+ (\partial W)$ and $g_0 \in \Riem^+ (W)_h$. If $k=4s-d-1\geq 0$ then the map
$$A_{k}(W,g_0) \otimes \bQ: \pi_{k} (\Riem^+ (W)_h, g_0) \otimes \bQ \lra KO_{4s}(*) \otimes \bQ = \bQ$$
is surjective. If $e=1,2$ and $k = 8s+e-d-1$, then the map
$$A_{k}(W,g_0): \pi_{k} (\Riem^+ (W)_h, g_0) \lra KO_{8s+e}(*) = \bZ/2$$
is surjective. In other words, the map $A_{k}(W, g_0)$ is nontrivial if $k\geq 0$, $d \geq 6$ and  the target is nontrivial.
\end{MainThm}

Theorem \ref{thm:intro:htpyGps} supersedes, to our knowledge, all previous results in the literature concerning the nontriviality of the maps $A_k(W, g_0)$, namely those of Hitchin \cite{HitchinSpin}, Gromov--Lawson \cite{GL}, Hanke--Schick--Steimle \cite{HSS}, and Crowley--Schick \cite{Crowley2013}.

Let us turn to the description of our main result, which implies Theorem \ref{thm:intro:htpyGps} by a fairly straightforward computation. To formulate it, we first recall the definition of a specific Madsen--Tillmann--Weiss spectrum. 
Let $\Gr^{\Spin}_{d,n}$ denote the spin Grassmannian (see Definition \ref{defn:spingrassmannian}) of $d$-dimensional subspaces of $\bR^n$ equipped with a spin structure. It carries a vector bundle $V_{d,n}\subset \Gr^\Spin_{d,n} \times \bR^n$ of rank $d$, which has an orthogonal complement $V^\perp_{d,n}$ of rank $(n-d)$.
There are structure maps $\Sigma\mathrm{Th}(V_{d,n}^{\perp}) \to \mathrm{Th}(V_{d,n+1}^{\perp})$ between the Thom spaces of these vector bundles, forming a spectrum (in the sense of stable homotopy theory) which we denote by $\MTSpin(d)$. This spectrum has associated infinite loop spaces
$$\loopinf{+l} \MTSpin (d) := \underset{n \to \infty}\colim \, \Omega^{n+l} \mathrm{Th} (V_{d,n}^{\perp})$$
for any $l \in \bZ$, where the colimit is formed using the adjoints of the structure maps.

The parametrised Pontrjagin--Thom construction associates to any smooth bundle $\pi:E \to B$ of compact $d$-dimensional spin manifolds a natural map 
\begin{equation}\label{intro:pt-map}
\alpha_E: B \lra \loopinf{} \MTSpin (d),
\end{equation}
which encodes many invariants of smooth fibre bundles. For example, there is a map of infinite loop spaces $\hat{\mathscr{A}}_d: \loopinf{} \MTSpin (d) \to \loopinf{+d} \KO$ such that the composition $\hat{\mathscr{A}}_d \circ \alpha_E: B \to \loopinf{+d} \KO$ represents the family index of the Dirac operators on the fibres of $\pi$ (a consequence of the Atiyah--Singer family index theorem). The collision of notation with the $\hat{\mathscr{A}}$-invariant mentioned earlier is intended: there is a map of spectra $\MTSpin (d) \to \Sigma^{-d} \MSpin$ into the desuspension of the classical spin Thom spectrum, and the classical $\hat{\mathscr{A}}$-invariant is induced by a spectrum map $\hat{\mathscr{A}}: \MSpin \to \KO$ constructed by Atiyah--Bott--Shapiro. Our map $\hat{\mathscr{A}}_d$ is the composition of these maps (or rather the infinite loop map induced by the composition). 

\begin{defn}\label{defn:intro:weaklyhomotopic}
Two continuous maps $f_0,f_1: X \to Y$ are called \emph{weakly homotopic} \cite{AdamsPhantom} if for each finite CW complex $K$ and each $g:K \to X$, the maps $f_0 \circ g$ and $f_1 \circ g$ are homotopic. Similarly for maps of pairs.
\end{defn}

With these definitions understood, we can now state the main results of this paper.

\begin{MainThm}\label{factorization-theorem:even-case}
Let $W$ be a spin manifold of dimension $2n \geq 6$. Fix $h \in \Riem^+ (\partial W)$ and $g_0 \in \Riem^+ (W)_{h}$. Then there is a map $\rho: \loopinf{+1} \MTSpin (2n) \to \Riem^+ (W)_{h}$ such that the composition 
$$\loopinf{+1} \MTSpin (2n) \stackrel{\rho}{\lra} \Riem^+ (W)_h \stackrel{\inddiff_{g_0}}{\lra} \loopinf{+2n+1} \KO$$
is weakly homotopic to $\Omega \hat{\mathscr{A}}_{2n}$, the loop map of $\hat{\mathscr{A}}_{2n}$.
\end{MainThm}

For manifolds of odd dimension, we have a result which looks very similar; we state it separately as it is deduced from Theorem \ref{factorization-theorem:even-case}, and its proof is quite different.

\begin{MainThm}\label{factorization-theorem:odd-case}
Let $W$ be a spin manifold of dimension $2n+1 \geq 7$. Fix $h \in \Riem^+ (\partial W)$ and $g_0 \in \Riem^+ (W)_{h}$. Then there is a map $\rho: \loopinf{+2} \MTSpin (2n) \to \Riem^+ (W)_h$ such that the composition 
$$\loopinf{+2} \MTSpin (2n) \stackrel{\rho}{\lra} \Riem^+ (W)_{h} \stackrel{\inddiff_{g_0}}{\lra} \loopinf{+2n+2} \KO$$
is weakly homotopic to $\Omega^2 \hat{\mathscr{A}}_{2n}$, the double loop map of $\hat{\mathscr{A}}_{2n}$.
\end{MainThm}

It should be remarked that in both cases the homotopy class of the map $\rho$ is in no sense canonical and depends on many different choices (among them the metric $g_0$). 
Theorem \ref{thm:intro:htpyGps} is a consequence of Theorems \ref{factorization-theorem:even-case} and \ref{factorization-theorem:odd-case} and relatively easy computations in stable homotopy theory. 
The geometric form of Theorems \ref{factorization-theorem:even-case} and \ref{factorization-theorem:odd-case} gives an interpretation in terms of spaces of psc metrics to more difficult and interesting stable homotopy theory computations as well. In Section \ref{sec:computation} we make such computations, and the results obtained are stated below in Section \ref{subsec:intro-furthercomputations}.

\subsection{Outline of the proofs of the main results}

We now turn to a brief outline of the proofs of Theorems \ref{factorization-theorem:even-case} and \ref{factorization-theorem:odd-case}, starting with Theorem \ref{factorization-theorem:even-case}. All manifolds in the sequel are assumed to be compact spin manifolds. 

The first ingredient a refinement of the Gromov--Lawson surgery theorem. This is the cobordism invariance theorem of Chernysh \cite{Chernysh} and Walsh \cite{WalshC}, which says that the homotopy type of $\Riem^+ (W)_h$ is unchanged when $W$ is modified by appropriate surgeries in its interior (the precise formulation is stated as Theorem \ref{chernysh-walsh-theorem} below). Together with the cut-and-paste invariance of the index difference (which we discuss in detail in Section \ref{sec:additivity}), this has two important consequences. Firstly, it is enough to prove Theorem \ref{factorization-theorem:even-case} when $W=D^{2n}$ and $h=h_{\round}$ is the round metric on $S^{2n-1}$; secondly, we are free to replace $D^{2n}$ by any other simply-connected manifold $W$ within its spin cobordism class (relative to $S^{2n-1}$). Here the hypothesis $2n \geq 6$ is required for the first time. From now on, let us assume that $W^{2n}$ is a compact connected spin manifold of dimension $2n \geq 6$ with $\partial W = S^{
2n-1}$.

As in \cite{HitchinSpin} and \cite{Crowley2013}, our method relies on the action of the diffeomorphism group on the space of psc metrics. For a smooth manifold $W$ equipped with a collar $c : [-\epsilon,0] \times \partial W \hookrightarrow W$, we write $\Diff_{\partial}(W)$ for the topological group of those diffeomorphisms of $W$ which are the identity on the image of $c$. This group acts by pullback on $\Riem^+ (W)_{h_{\round}}$, and we may form the Borel construction $E \Diff_{\partial} (W) \times_{\Diff_{\partial} (W)} \Riem^+ (W)_{h_{\round}}$, which fits into a fibration sequence
\begin{equation}\label{intro-fibre-sequence}
\Riem^+ (W)_{h_{\round}} \lra E \Diff_{\partial} (W) \times_{\Diff_{\partial} (W)} \Riem^+ (W)_{h_{\round}} \lra B \Diff_{\partial} (W).
\end{equation}
The Borel construction $E \Diff_{\partial} (W) \times_{\Diff_{\partial} (W)} \Riem^+ (W)_{g_{\round}} $ is the homotopy theoretic quotient $\Riem^+ (W)_{g_{\round}} \hq \Diff_{\partial} (W)$, and we use this notation from now on. 

The action of $\Diff_{\partial}(W)$ on $\Riem^+ (W)_{h_{\round}}$ is free, so the projection map
$$\Riem^+ (W)_{h_{\round}} \hq \Diff_{\partial} (W) \lra \Riem^+ (W)_{h_{\round}} / \Diff_{\partial} (W)$$
is a weak homotopy equivalence. This quotient space is known as the \emph{moduli space of psc metrics on $W$} in the literature, e.g. \cite[\S 1]{TW}. 

We, however, find the following point-set model for the spaces in \eqref{intro-fibre-sequence} more enlightening. Choose an embedding $\partial W \subset \bR^{\infty-1}$ and then take as a model for $E \Diff_{\partial}(W)$ the space $\Emb_\partial(W, (-\infty,0] \times \bR^{\infty-1})$ of all embeddings $e : W \to (-\infty,0] \times \bR^{\infty-1}$ such that $e \circ c (t, x) = (t, x)$ for all $(t,x) \in [-\epsilon,0] \times \partial W$. With this model, $B \Diff_{\partial} (W)$ may be identified with the set of all compact submanifolds $X \subset (-\infty,0] \times \bR^{\infty-1}$ such that $X \cap ([-\epsilon, 0] \times \bR^{\infty-1}) = [-\epsilon, 0] \times \partial W$ and which are diffeomorphic (relative to $\partial W$) with $W$. One may therefore view $B \Diff_{\partial} (W)$ as the moduli space of manifolds diffeomorphic to $W$.
Using this model, the Borel construction $E \Diff_{\partial} (W) \times_{\Diff_{\partial} (W)} \Riem^+ (W)_{h_{\round}}$ is the space of pairs $(X,g)$ with $X \in B \Diff_{\partial} (W)$ and $g \in \Riem^+ (X)_{h_{\round}}$. 

One important feature of the fibre sequence \eqref{intro-fibre-sequence} is that it interacts well with index-theoretic constructions, by virtue of a homotopy commutative diagram
\begin{equation}\label{intro-index-fibre-diagram}
\begin{gathered}
\xymatrix{
\Riem^+ (W)_{h_{\round}} \ar[d] \ar[r]^{\inddiff} & \loopinf{+2n+1} \KO \ar[d]\\
\Riem^+ (W)_{h_{\round}} \hq \Diff_{\partial} (W) \ar[d]\ar[r] & \ast \ar[d] \\
 B \Diff_{\partial} (W) \ar[r]^{\mathrm{ind}} & \loopinf{+2n} \KO,
}
\end{gathered}
\end{equation}
which we establish in Section \ref{subsec:spin-fibre-bundles-and-inddiff}. The right-hand column is the path-loop-fibration, the top map is the index difference (with respect to a basepoint) and the bottom map is the ordinary family index of the Dirac operator on the universal $W$-bundle over $B \Diff_{\partial} (W)$, which can be computed in topological terms by the index theorem. 

If we choose $W$ so that the topology of $B \Diff_{\partial} (W)$ is well-understood, we might hope to extract information about the homotopy type of $\Riem^+ (W)_{h_{\round}}$ and hence about $\pi_* (\inddiff)$ from \eqref{intro-fibre-sequence}. 
For example, if one can show that the map $\mathrm{ind}_*:\pi_k (B \Diff_{\partial} (W)) \to KO^{-2n-k} (*)$ is nontrivial, then it follows that the map $\inddiff_* : \pi_{k-1} (\Riem^+(W)_{g_\circ}) \to KO^{-2n-k}(*)$ is also nontrivial. This is essentially the technique which was introduced by Hitchin \cite{HitchinSpin}. But precise knowledge of $\pi_* (B\Diff_{\partial} (W))$ is scarce, especially when additional information such as the map $\mathrm{ind}_*:\pi_{k} (B \Diff_{\partial } (W))    \to KO^{-2n-k} (*)$ is required as well. 

The driving force behind Theorem \ref{factorization-theorem:even-case} is the high-dimensional analogue of the Madsen--Weiss theorem, proved by Galatius and the third named author \cite{GRW, GRWHomStab2}, which makes it possible to understand the \emph{homology} of $B \Diff_{\partial} (W)$, rather than its homotopy. To describe this result, let $K:=  ([0,1] \times S^{2n-1}) \# (S^n \times S^n)$, and when $W_0$ is a $2n$-dimensional compact spin manifold with boundary $\partial W_0=S^{2n-1}$ write $W_k:= W \cup kK$ for the composition of $W_0$ with $k$ copies of $K$. We choose $W_0$ such that
\begin{enumerate}[(i)]
\item $W_0$ is spin cobordant to $D^{2n}$ relative to $S^{2n-1}$ and
\item the map $W_0 \to B \Spin (2n)$ classifying the tangent bundle of $W_0$ is $n$-connected.
\end{enumerate}
Such a manifold $W_0$ exists, and in fact by Kreck's stable diffeomorphism classification theorem \cite[Theorem C]{Kreck} it is unique in the sense that if two such manifolds $W_0$ and $W'_0$ are given, then $W_k \cong W'_l$ for some $k, l \in \bN$. 

Gluing in the cobordism $K$ induces stabilisation maps 
\begin{equation*}
B \Diff_{\partial} (W_0) \lra B \Diff_{\partial} (W_1) \lra B \Diff_{\partial} (W_2) \lra B \Diff_{\partial} (W_3) \lra  \cdots.
\end{equation*}
The parametrised Pontrjagin--Thom maps \eqref{intro:pt-map} induce a map
\begin{equation}\label{grw-map:intro}
\alpha_{W_{\infty}}:B \Diff_\partial (W_\infty) := \underset{k \to \infty}\hocolim \, B \Diff_\partial (W_k) \lra \loopinf{}_{0}\MTSpin(2n),
\end{equation}
and it follows from \cite{GRWHomStab2} that the map $\alpha_{W_{\infty}}$ is acyclic. Recall that a map $f : X \to Y$ of spaces is called \emph{acyclic} if for each $y \in Y$ the homotopy fibre $\hofib_y (f)$ has the singular homology of a point. In particular, the map \eqref{grw-map:intro} induces an isomorphism in homology and can in fact be identified with the Quillen plus construction of $B \Diff_\partial (W_\infty)$ 

Once a psc metric $m \in \Riem^+ (K)_{h_{\round}, h_{\round }}$ is chosen, gluing in the Riemannian cobordism $(K,m)$ induces stabilisation maps
$$\Riem^+ (W_0)_{h_{\round}} \lra \Riem^+ (W_1)_{h_{\round}} \lra \Riem^+ (W_2)_{h_{\round}} \lra \Riem^+ (W_3)_{h_{\round}} \lra \cdots.$$
The cobordism invariance theorem of Chernysh \cite{Chernysh} and Walsh \cite{WalshC} allows us to choose the psc metric $m$ so that all these stabilisation maps are homotopy equivalences, so in particular the map
\begin{equation}\label{eq:RStab}
\Riem^+ (W_0)_{h_{\round}} \lra \underset{k \to \infty}\hocolim \, \Riem^+ (W_k)_{h_{\round}} 
\end{equation}
is a homotopy equivalence. Acting on a fixed collection of compatible psc metrics gives a map
$$\underset{k \to \infty}\hocolim \, \Diff_\partial (W_k) \lra \underset{k \to \infty}\hocolim \, \Riem^+ (W_k)_{h_{\round}} \simeq \Riem^+ (W_0)_{h_{\round}} \simeq \Riem^+ (D^{2n})_{h_{\round}},$$
where we have used cobordism invariance twice, once in the form of \eqref{eq:RStab} and once to replace $W_0$ by the spin cobordant manifold $D^{2n}$. The map we wish to construct is an extension of this map along $\Omega \alpha_{W_\infty} : \hocolim_k \, \Diff_\partial (W_k) \to \loopinf{+1}\MTSpin(2n)$, but, as the map $\alpha_{W_\infty}$ is not a homotopy equivalence but merely acyclic, an argument is necessary to form such an extension.

The two stabilisation maps (on the moduli space of manifolds and on the space of psc metrics respectively) together yield stabilisation maps on the Borel construction. After passing to the (homotopy) colimit, we obtain a fibre sequence 
\[
 \underset{k \to \infty}\hocolim \, \Riem^+ (W_k)_{h_{\round}} \to \underset{k \to \infty}\hocolim \, ( \Riem^+ (W_k)_{h_{\round}} \hq \Diff_{\partial} (W_k) ) \stackrel{p_\infty}{\to} \underset{k \to \infty}\hocolim \, B\Diff_{\partial} (W_k). 
\]
The key step is now the construction of a fibration $T_\infty^+ \stackrel{p_{\infty}^+}{\to} \loopinf{}_0 \MT \Spin (2n)$ with fibre $\Riem^+ (W_0)_{h_\round}$ and a homotopy cartesian diagram 
\begin{equation}\label{intro-stable-fibre-sequence}
\begin{gathered}
\xymatrix{
 {\underset{k \to \infty}\hocolim \, ( \Riem^+ (W_k)_{h_{\round}} \hq \Diff_{\partial} (W_k) )} \ar[d]^{p_\infty} \ar[r] & T_\infty^+ \ar[d]^{p_{\infty}^+}\\
 {\underset{k \to \infty}\hocolim \, B\Diff_{\partial} (W_k)} \ar[r]^-{\alpha_{W_{\infty}}} & \loopinf{}_0 \MTSpin (2n). 
}
\end{gathered}
\end{equation}




The main input for the construction of the diagram \eqref{intro-stable-fibre-sequence} is the result that for each pair of diffeomorphisms $f_0,f_1 \in \Diff_\partial (W_k)$ the automorphisms $f_{0}^{*}, f_{1}^{*}: \Riem^+ (W_k)_{h_\round} \to \Riem^+ (W_k)_{h_\round}$ \emph{commute up to homotopy}. This is again derived from the cobordism invariance theorem, along with an argument of Eckmann--Hilton type. This commutativity property allows us to carry out the obstruction-theoretic argument to produce \eqref{intro-stable-fibre-sequence}. At this point it is crucial that $\alpha_{W_{\infty}}$ is acyclic and not only a homology equivalence. 

The map $\rho$ is defined to be the fibre transport map
$$\loopinf{+1}\MTSpin(2n) \lra \underset{k \to \infty}\hocolim \, \Riem^+ (W_k)_{h_{\round}}$$
of the fibration $p_\infty^+$, followed by the homotopy inverse of \eqref{eq:RStab} and the identification $\Riem^+(W_0)_{g_\circ} \simeq \Riem^+(D^{2n})_{g_\circ}$ coming from the cobordism invariance theorem. 
To show that $\inddiff \circ \rho$ is weakly homotopic to $\Omega\hat{\mathscr{A}}_{2n}$ we use the diagram \eqref{intro-index-fibre-diagram}, Bunke's additivity theorem for the index \cite{Bunke1995}, and the Atiyah--Singer family index theorem.

The deduction of Theorem \ref{factorization-theorem:odd-case} from Theorem \ref{factorization-theorem:even-case} is by a short index-theoretic argument. In the course of that argument, we need to compare indices of operators on manifolds of different dimension, and some index theorem is needed for that purpose. We phrase the argument by using an alternative description of the index difference map for closed manifolds, due to Gromov and Lawson \cite{GL2}, in terms of a boundary value problem of Atiyah--Patodi--Singer type on the cylinder $W \times [0,1]$. The proof of Theorem \ref{factorization-theorem:odd-case} relates the index difference for $(2n+1)$-dimensional manifolds (Hitchin's definition) with the index difference for $2n$-dimensional manifolds (Gromov--Lawson's definition). In order to make the argument conclusive we need to know that these definitions agree, but this follows from a family version of the spectral flow index theorem which was proved by the second named author \cite{Eb13}.

\begin{remark}
The argument for the deduction of Theorem \ref{factorization-theorem:odd-case} from Theorem \ref{factorization-theorem:even-case} also yields that $\im (A_{k+1} (\partial W,h)) \subset \im (A_{k}(W,g))$ when $g \in \cR^+(W)_h$. This allows one to prove Theorem \ref{thm:intro:htpyGps} by induction on the dimension, starting with $d=6$. Along the same lines, if Theorem \ref{factorization-theorem:even-case} is established for $2n=6$, we get for all $d \geq 6$ a factorisation
$$\Omega^{d-5}\hat{\mathscr{A}}_6 : \loopinf{+d-5} \MTSpin (6) \stackrel{\rho}{\lra} \Riem^+ (W^d)_h \stackrel{\inddiff_{g}}{\lra} \loopinf{+d+1} \KO,$$
which suffices for some of the computational applications. 

In addition, the proof of Theorem \ref{factorization-theorem:even-case} in the special case $2n=6$ enjoys several simplifications, the principal one being that the results of \cite{GRWHomStab2} may be replaced by those of \cite{GRW}. Consequently, we have given (in Section \ref{finishingproof:dim6}) a separate proof of this special case.
\end{remark}


\subsection{Further computations}\label{subsec:intro-furthercomputations}

We state the following more detailed computations for even-dimensional manifolds: they have odd-dimensional analogues too, which we leave to the reader to deduce from the results of Section \ref{sec:computation}.
The first concerns the surjectivity of the map on homotopy groups induced by the index difference, without any localisation.

\begin{MainThm}\label{thm:intro:integral}
Let $W$ be a spin manifold of dimension $2n \geq 6$. Fix $h \in \Riem^+ (\partial W)$ and $g_0 \in \Riem^+ (W)_{h}$. Then the map
$$A_{k}(W,g_0): \pi_{k} (\Riem^+ (W)_h, g_0) \lra KO_{k+2n+1}(*)$$
is surjective for $0 \leq k \leq 2n-1$.
\end{MainThm}

One application of this theorem is as follows. Let $B^8$ be a spin manifold such that $\hat{\mathscr{A}}(B)\in KO_8(*)$ is the Bott class (such a manifold is sometimes called a ``Bott manifold''). By the work of Joyce \cite[\S 6]{Joyce} there is a Bott manifold which admits a metric $g_B$ with holonomy group $\Spin (7)$. Then $g_B$ must be Ricci-flat and hence scalar-flat. For any closed spin $d$-manifold $W$, cartesian product with $(B, g_B)$ thus defines a direct system
$$\cR^+(W) \lra \cR^+(W \times B) \lra \cR^+(W \times B \times B) \lra \cdots$$
and as $\hat{\mathscr{A}}(B)$ is the Bott class there is an induced map from the direct limit
$$\inddiff_{h}[B^{-1}] : \cR^+(W)[B^{-1}] := \underset{k \to \infty} \hocolim \, \cR^+(W \times B^k) \lra \Omega^{\infty+d+1} \KO.$$
It then follows from Theorem \ref{thm:intro:integral} (or its odd-dimensional analogue) that this map is surjective on all homotopy groups.

Secondly, working away from the prime 2 we are able to use work of Madsen--Schlichtkrull \cite{MadsenSchlichtkrull} to obtain an upper bound on the index of the image of the index difference map on homotopy groups.

\begin{MainThm}\label{thm:intro:awayfrom2}
Let $W$ be a spin manifold of dimension $2n \geq 6$. Fix $h \in \Riem^+ (\partial W)$ and $g_0 \in \Riem^+ (W)_{h}$. Then the image of the map
$$A_{4m-2n-1}(W, g_0)[\tfrac{1}{2}] : \pi_{4m-2n-1}(\cR^+(W)_h, g_0)[\tfrac{1}{2}] \lra KO_{4m}(*)[\tfrac{1}{2}]$$
has finite index, dividing
$$A(m, n):= \gcd\left\{ \prod_{i=1}^n(2^{2m_i-1}-1) \cdot \mathrm{Num}\left(\frac{B_{m_i}}{2m_i}\right)  \,\Bigg\vert\, m_i \geq 0,  \sum_{i=1}^n m_i = m\right\}$$
(where we adopt the convention that $(2^{2m-1}-1)\cdot \mathrm{Num}(\frac{B_{m}}{2m})=1$ when $m=0$).
\end{MainThm}

While the numbers $A(m,1) = (2^{2m-1}-1) \cdot \mathrm{Num}(\tfrac{B_{m}}{2m})$ are complicated, computer calculations (for which we thank Benjamin Young) show that $A(m,2)=1$ for $m \leq 45401$. Hence $A(m,2\ell)=1$ for $m \leq 45401\cdot \ell$, and so for $W$ a spin manifold of dimension $4\ell \geq 6$ the map
\begin{equation*}
A_{k}(W, g_0)[\tfrac{1}{2}] : \pi_{k}(\cR^+(W)_h, g_0)[\tfrac{1}{2}] \lra KO_{k+4\ell+1}(*)[\tfrac{1}{2}]
\end{equation*}
is surjective for $k < 45400 \cdot 4\ell$. (One can deduce similar ranges for manifolds whose dimensions have other residues modulo 4, cf.\ Section \ref{sec:awayfrom2}.)


\begin{rem}
In Section \ref{sec:sharp} we show that the estimate in Theorem \ref{thm:intro:awayfrom2} is approximately sharp: the quotient of $A(m,n)$ by the index of the image of the map
\begin{equation*}
\pi_{4m-2n}(\MTSpin(2n))[\tfrac{1}{2}] \lra \pi_{4m-2n-1}(\cR^+(W)_h, g_0)[\tfrac{1}{2}] \lra  KO_{4m}(*)[\tfrac{1}{2}]
\end{equation*}
has all prime factors $p$ relatively small, in the sense that $p \leq 2m+2$. Thus any prime number $q$ dividing $A(m,n)$ with $q > 2m+2$ must in fact divide the index of the image of this composition.
\end{rem}

Thirdly, we study the $p$-local homotopy type of the spaces $\cR^+(S^d)$ of positive scalar curvature metrics on spheres. In \cite{WalshH}, Walsh has shown that $\cR^+(S^d)$ admits the structure of an $H$-space, so for any prime $p$ we may form the localisation $\cR^+_\round(S^d)_{(p)}$ of the identity component of this $H$-space, that is, the component of the round metric $g_\round$. This may be constructed, for example, as the mapping telescope of the $r$th power maps on this $H$-space, over all $r$ coprime to $p$.
\begin{MainThm}\label{thm:intro:split}
Let $d \geq 6$ and $p$ be an odd prime. Then there is a map
$$f : (\loopinf{+d+1}\KO)_{(p)} \lra \cR^+_\round(S^d)_{(p)}$$
such that $(\inddiff_{g_\round})_{(p)} \circ f$ induces multiplication by $(2^{2m-1}-1)\cdot \mathrm{Num}(\frac{B_m}{2m})$ times a $p$-local unit on $\pi_{4m-d-1}$.
\end{MainThm}

Thus if $p$ is an odd regular prime (i.e.\ is not a factor of $\mathrm{Num}(B_m)$ for any $m$) and in addition does not divide any number of the form $(2^{2m-1}-1)$, then there is a splitting
$$\cR^+_\round(S^d)_{(p)} \simeq (\loopinf{+d+1}_0\KO)_{(p)} \times F_{(p)},$$
where $F$ is the homotopy fibre of $\inddiff_{g_\round}$. In particular, for such primes the map induced by $\inddiff_{g_\round}$ on $\bF_p$-cohomology is injective.

At the prime $2$ we are not able to obtain such a strong splitting result, but we can still establish enough information to obtain the cohomological implication.

\begin{MainThm}\label{thm:intro:mod2}
For $d \geq 6$, the map $\inddiff_{g_\round} :\Riem^+(S^{d}) \to \loopinf{+d+1}\KO$ is injective on $\bF_2$-cohomology.
\end{MainThm}

\subsection*{Acknowledgements}
The authors would like to thank Mark Walsh for 
helpful discussions concerning surgery results for psc metrics.
An early draft of the paper had intended to contain a detailed
appendix written by Walsh on this subject; later we found a way to prove the main results without using these delicate geometric arguments. The appendix grew instead into a separate paper, \cite{Walsh4}, which is of independent interest.

The authors would also like to thank Don Zagier and Benjamin Young for their interest and advice regarding the numbers $A(m,2)$ arising in Theorem \ref{thm:intro:awayfrom2}.

\section{Spaces of metrics of positive scalar curvature}\label{chap:pscspaces}

We begin this chapter by precisely defining the spaces which we shall study, in Section \ref{sec:defnpsc}. In the other sections, we survey results on spaces of positive scalar curvature metrics which we shall need later. Section \ref{sec:basic-constructions-psc} provides technical but mostly elementary results, for later reference. 
The most important result, to be discussed in Section \ref{sec:bordismtheorem}, is the cobordism invariance theorem of Chernysh and Walsh and the hasty reader can jump directly to that section.

\subsection{Definitions}\label{sec:defnpsc}
Let $W: M_0 \rightsquigarrow M_1$ be a cobordism between closed manifolds. For us, a cobordism will always be a morphism in the cobordism category in the sense of \cite[\S 2.1]{GMTW}. In particular, the boundary of $W$ is collared: for some $a_0 < c_0 < c_1 <a_1 \in \bR$, there is an embedding $b=b_0 \sqcup b_1: [a_0,c_0] \times M_0 \sqcup [c_1, a_1] \times M_1 \to W$ which induces the canonical identification of $\{a_i\}\times M_i$ with $M_i$.
Let $\Gamma (W;\Sym^2 (TW))$ be the space of smooth symmetric $(0,2)$-tensor fields on $W$.
This is a Frech\'et topological vector space; the topology is generated by the maximum
norms $\| \nabla^k u \|_{C^0}$, where $u \in \Gamma (W;\Sym^2 (TW))$,
and the gradients are taken with respect to any fixed reference metric on $W$; the topology so defined does not depend on the specific choice of this reference metric.

For $\epsilon_i>0$ small enough, we write $\Riem (W)^{\epsilon_0,\epsilon_1}\subset \Gamma (W; \Sym^2(TW))$ for the subspace of all Riemannian metrics $g$ on $W$ for which there exist Riemannian metrics $h_i$ on $M_i$ such that $b_i^*(g) = h_i+ dt^2$ on the collars $[a_0,a_0+\epsilon_0] \times M_0$ and $[a_1-\epsilon_1,a_1] \times M_1$. This is convex and hence contractible.
The subspace $\Riem^+(W)^{\epsilon_0,\epsilon_1} \subset \Riem (W)^{\epsilon_0,\epsilon_1}$ of metrics with positive scalar curvature is open. 
Any metric $g$ on $W$ induces a metric $h$ on $\partial W$ by restriction, and if $g$ is of product form $g=h+ dt^2$ on the collar and has positive scalar curvature then $h$ also has positive scalar curvature. This yields a
continuous map
$$\res: \Riem^+ (W)^{\epsilon_0,\epsilon_1} \lra \Riem^+ (M_0)\times \Riem^+ (M_1).$$
For a pair $(h_0,h_1)\in \Riem^+ (M_0)\times \Riem^+ (M_1)$ we write
$$\Riem^+(W)^{\epsilon_0,\epsilon_1}_{h_0,h_1}:= \res^{-1}(h_0,h_1) \subset \Riem^+(W)^{\epsilon_0,\epsilon_1}.$$
In plain language: this is the space of all positive scalar curvature metrics whose restriction to the collar around $M_i$ coincides with $h_i + dt^2 $.

Let us write
$$r: \Gamma (W; \Sym^2(TW)) \lra \Gamma ([a_0,a_0+\epsilon_0] \times M_0 \sqcup [a_1-\epsilon_1,a_1] \times M_1; \Sym^2 (TW))$$
for the restriction map. Note that $\Riem^+(W)^{\epsilon_0,\epsilon_1}_{h_0,h_1}$ is an open subspace of the space $r^{-1} (h_0+dt^2 \sqcup h_1+dt^2)$, and that the latter is homeomorphic to a Frech\'et space. Therefore $\Riem^+ (W)_{h_0,h_1}^{\epsilon_0,\epsilon_1}$ is a metric space, and hence paracompact. From \cite[Theorem 13]{PalaisInfMan} and \cite[Proposition A.11]{Hatcher}, it follows that $\Riem^+ (W)_{h_0,h_1}^{\epsilon_0,\epsilon_1}$ has the homotopy type of a CW complex. 

From now on, we abbreviate 
$$\Riem^+ (W):= \Riem^+ (W)^{\epsilon_0,\epsilon_1} \quad \text{and}\quad \Riem^+ (W)_{h_0,h_1} :=\Riem^+ (W)_{h_0,h_1}^{\epsilon_0,\epsilon_1}$$
for implicitly fixed values of $\epsilon_i$. In Lemma \ref{collar-stretching} below we will show that the homotopy type of this space does not depend on $\epsilon_i$, which justifies this short notation. 

If one of the manifolds $M_i$ is empty then we write $\Riem^+ (W)_h$, where $h$ is the boundary metric. 

Let $W: M_0 \rightsquigarrow M_1$ and $W': M_1 \rightsquigarrow M_2$
be two cobordisms and $W\cup W'=W\cup_{M_1} W'$ be their composition. Let $h_i \in \Riem^+ (M_i)$, $i=0,1,2$, be
given. Then there is a gluing map
$$\mu: \Riem^+ (W)_{h_0,h_1} \times \Riem^+ (W')_{h_1,h_2} \lra \Riem^+ (W \cup W')_{h_0,h_2},$$
where the metric $\mu(g,g')$ is defined to agree with $g$ on $W$ and with $g'$ on $W'$. If we fix $g' \in \Riem^+ (W')_{h_1,h_2}$, then we obtain a map
\begin{align*}
\mu_{g'}: \Riem^+ (W)_{h_0,h_1} & \lra \Riem^+ (W \cup W')_{h_0,h_2}\\
g & \longmapsto \mu (g,g')
\end{align*}
by gluing in the metric $g'$. Sometimes we abbreviate $g \cup g':= \mu(g,g')$.

\subsection{Some basic constructions with psc metrics}\label{sec:basic-constructions-psc}

\subsubsection{Collar stretching}

\begin{lem}\label{collar-stretching}
For $0<\delta_i\leq \epsilon_i<|a_i-c_i|$, the inclusion $\Riem^+ (W)_{h_0,h_1}^{\epsilon_0,\epsilon_1}\hookrightarrow \Riem^+ (W)_{h_0,h_1}^{\delta_0,\delta_1}$ is a homotopy equivalence.
\end{lem}

\begin{proof}
For typographical simplicity, we assume that $a_0=0$, $M_1 = \emptyset$ and write $(\epsilon,\delta,c,h):=(\epsilon_0,\delta_0,c_0,h_0)$. Let $H_s : \bR \to \bR$, $s \in [0,1]$, be an isotopy such that 

\begin{enumerate}[(i)]
\item $H_0=\id$,  
\item\label{it:collar-stretching:2} $H_s = \id$ near $[c,\infty)$ and near $(-\infty,0]$,
\item $(H_s)' =1$ near $\delta$,
\item $H_1 (\delta)=\epsilon$ and 
\item $H_s \leq H_u$ for $s \leq u$. 
\end{enumerate}
This induces an isotopy of embeddings, also denoted $H_s$, of the collar to itself, and by condition (\ref{it:collar-stretching:2}) also of $W$ into itself. Define a homotopy $F_s : \Riem^+ (W)^{\delta}_{h} \to \Riem^+ (W)^{\delta}_{h}$ by the formula
$$F_s (g):=
\begin{cases}
dt^2 + h & \text{ on $[0,H_s (\delta)]$},\\
(H_s)^* g & \text{ elsewhere. }
\end{cases}$$

By construction, $F_s$ maps the subspace $\Riem^+ (W)_{h}^{\epsilon}$ to itself, $F_0$ is the identity and $F_1$ maps $\Riem^+ (W)_{h}^{\delta}$ into $\Riem^+ (W)_{h}^{\epsilon}$. This proves that $F_1$ is a two-sided homotopy inverse of the inclusion, as claimed.
\end{proof}

Lemma \ref{collar-stretching} has the following immediate consequence.

\begin{cor}\label{cor:collarstretching}\mbox{}
\begin{enumerate}[(i)]
\item The map $\Riem^+ (W)_{h_0,h_1}^{\epsilon_0,\epsilon_1}\to \underset{\epsilon_0,\epsilon_1 \to 0}\colim \Riem^+ (W)_{h_0,h_1}^{\epsilon_0,\epsilon_1}$ is a weak homotopy equivalence.
\item\label{it:CollarStretching} For $a_2 >a_1$ and $h_1 \in \Riem^+ (M_1)$, let $g=dt^2 + h_1 \in \Riem^+ ([a_1,a_2] \times M_1)$. The gluing map $\mu_{g}: \Riem^+ (W)_{h_0,h_1} \to \Riem^+ (W \cup ([a_1,a_2] \times M_1))_{h_0,h_1} $ is a homotopy equivalence.
\end{enumerate}
\end{cor}

\subsubsection{The quasifibration theorem}

Let $W$ be a manifold with collared boundary $M$ and $\res: \Riem^+ (W) \to \Riem^+ (M)$ be the restriction map. For $h \in \Riem^+ (M)$, the geometric fibre $\res^{-1} (h)$ is the space $\Riem^+ (W)_{h}$ while the homotopy fibre $\hofib_h (\res)$ is the space of pairs $(g,p)$, with $g \in \Riem^+ (W)$ and $p$ a continuous path in $\Riem^+ (M)$ from $\res(g)$ to $h$. Inside the homotopy fibre, we have the space $(\hofib_h (\res))_{C^{\infty}}$, which is defined by the condition that $p$ has to be a smooth path. The inclusion $i:(\hofib_h (\res))_{C^{\infty}} \to \hofib_h (\res)$ is a homotopy equivalence \cite[Lemma 2.3]{Chernysh2}. Chernysh constructs a map 
$$S': (\hofib_h (\res))_{C^{\infty}} \lra \Riem^+ (W)_{h}^{\epsilon}$$
roughly as follows: pick an embedding $j: W \to W$ onto the complement of a collar $[0,1] \times \partial W \subset W$, then the metric $S'(g,p)$ is defined to be $(j^{-1})^* g$ on the image of $j$, and a suitably tempered form of the metric $dt^2 + p_t$ on the collar (the metric $dt^2 + p_t$ has in general neither positive scalar curvature nor a product form near the boundary $\{0,1\}\times \partial W$, but Chernysh shows how to carefully modify it to have these properties). Chernysh proves that $S'$ is a two-sided homotopy inverse to the fibre inclusion $\Riem^+ (W)_{h} \to(\hofib_h (\res))_{C^{\infty}}$ \cite[Lemma 2.2]{Chernysh2}. By inverting $i$, we obtain a homotopy class of map $S:\hofib_h (\res)\to \Riem^+ (W)_{h} $. 

\begin{thm}{\rm (Chernysh \cite{Chernysh2})}\label{chernysh-thm:qfibr}
\begin{enumerate}[(i)]
\item The map $S$ is a two-sided homotopy inverse to the fibre inclusion $\Riem^+ (W)_{h} \to \hofib_h (\res)$.
\item In particular, the restriction map $\res: \Riem^+ (W) \to \Riem^+(M)$ is a quasifibration.
\end{enumerate}
\end{thm}

\subsection{The cobordism theorem}

\subsubsection{Standard metrics}
On the disc $D^d$, fix a collar of its boundary $S^{d-1} \subset D^d$  by the formula
\begin{align*}
b : S^{d-1} \times (-1,0] &\lra D^d\\
(v, t) &\longmapsto (1+t)\cdot v.
\end{align*}
We assume that disks are always equipped with this collar.

On the sphere $S^d$, let $g_{\round}^d = g_{\round} \in \Riem^{+} (S^d)$ and $h_{\round}^{d-1} = h_{\round} \in \Riem^{+} (S^{d-1})$ be the
ordinary metrics of Euclidean spheres of radius $1$ (of positive scalar curvature as long as the sphere has dimension at least 2)\footnote{The notation $g$ versus $h$ carries no mathematical meaning, but we typically use $g$'s for metrics on a cobordism and $h$'s for metrics on the boundary of a cobordism.}. Let $g_{\hemi}^d$ be the
metric on $D^d$ which comes from identifying $D^d \subset \bR^d$ with the lower
hemisphere of $S^d \subset \bR^{d+1}$ via 
\begin{align*}
D^d &\lra S^d\\
x &\longmapsto (x, - \sqrt{1-|x|^2})
\end{align*}
and taking $g^d_\round$ under this identification. (Note that $g_{\hemi}^d$ does not have a product form near the boundary of $D^d$.)

We say that a rotation-invariant psc-metric $g$ on $D^d$ is a \emph{torpedo metric} if 
\begin{enumerate}[(i)]
\item $b^*(g)$ agrees with the product metric $h_{\round}^{d-1} + dt^2$ near $S^{d-1} \times \{0\}$,
\item $g$ agrees with $g_{\hemi}^d$ near the origin.
\end{enumerate}
We fix a torpedo metric $g_{\tor}^{d}$ on $D^d$ once and for all (for
each $d \geq 3$). In \cite[\S 2.3]{Walsh01}, it is proved that $g_{\tor}^d$ can be chosen to have the following extra property: the metric on
$S^d$ obtained by gluing together two copies of $g_{\tor}^d$ on the
upper and lower hemispheres is isotopic to $g_{\round}^d$. 
Such a metric on $S^d$ will be called a \emph{double torpedo metric} and denoted by
$g^d_{\dtor}$.

\subsubsection{Spaces of metrics which are standard near a submanifold}

Let $W$ be a compact manifold of dimension $d$ with boundary $M$, equipped with a collar $b: M \times (-1,0] \to W$. Let $X$ be a closed $(k-1)$-dimensional manifold and $\phi :  X^{k-1} \times D^{d-k+1} \to W^d$ be an embedding, and suppose that $\phi$ and $b$ are disjoint. Let $g_X \in \Riem (X)$ be a Riemannian metric, not necessarily of positive scalar curvature. However, we assume that the metric $g_X + g_{\tor}^{d-k+1}$ on $X \times D^{d-k+1}$ has positive scalar curvature (this is the case for example if $g_X$ has non-negative scalar curvature). Fix $h \in \Riem^+ (M)$ and let 
$$\Riem^+ (W ; \phi, g_X)_h \subset \Riem^+ (W)_h$$
be the subspace of those metrics $g$ such that $\phi^* g = g_X + g_{\tor}^{d-k+1}$. We call this the \emph{space of psc metrics on $W$ which are standard near $X$}. One of the main ingredients of the proof of Theorem \ref{factorization-theorem:even-case} is the following result, due to Chernysh \cite{Chernysh}. A different proof was later given by Walsh \cite{WalshC}.

\begin{thm}\label{chernysh-walsh-theorem}{\rm (Chernysh, Walsh)}
If $d -k+1 \geq 3$, then the inclusion map 
$$
\Riem^+(W ; \phi, g_X)_{h} \lra \Riem^+(W)_{h}
$$
is a homotopy equivalence.
\end{thm}

Both authors state the result when the manifold $W$ is closed. However, the deformations of the metrics appearing in the proof take place in a given tubular neighbourhood of $X$, and therefore the global structure of $W$ does not play a role. The precursor of Theorem \ref{chernysh-walsh-theorem} is the famous surgery theorem of Gromov and Lawson \cite{GL}: if $\Riem^+(W)_{h}$ is nonempty, then $\Riem^+(W; \phi,g_X )_{h}$ is nonempty. One might state this by saying that the inclusion map is $(-1)$-connected. Gajer \cite{Gajer} showed that the inclusion map is $0$-connected, i.e.\ that each psc metric on $W$ is isotopic to one which is standard near $X$.

\subsubsection{Cobordism invariance of the space of psc metrics}\label{sec:bordismtheorem}
The original application of Theorem \ref{chernysh-walsh-theorem} was
to show that for a closed, simply-connected, spin manifold $W$ of dimension at least 5,
the homotopy type of $\Riem^+ (W)$ only depends on the spin cobordism
class of $W$. We recall the precise statement and its proof.

\begin{thm}\label{surgery-invariance}{\rm (Chernysh, Walsh)}
Let  $W: M_0 \rightsquigarrow M_1$  be a compact $d$-dimensional cobordism, $\phi: S^{k-1} \times D^{d-k+1} \to \inter \ W$ be an embedding, and $W'$ be the result of surgery
along $\phi$. Fix $h_i \in \Riem^+ (M_i)$. 
If $3 \leq k \leq d-2$ then there is a homotopy
equivalence 
$$\SE_{\phi}:\Riem^+ (W)_{h_0,h_1} \simeq \Riem^+ (W')_{h_0,h_1}.$$
Furthermore, the surgery datum $\phi$ determines a preferred homotopy
class of $\SE_{\phi}$.
\end{thm}

The map $\SE_{\phi}$ is called the \emph{surgery equivalence} induced by the surgery datum $\phi$.

\begin{proof}
Since the surgery is in the interior of $W$, the boundary of $W$ and the metric on $\partial W$ are not affected. So we may, for typographical simplicity, assume that $W$ is closed. Let $W^{\circ}:= W \setminus
\phi(S^{k-1} \times \inter(D^{d-k+1}))$, a manifold with boundary $S^{k-1} \times
S^{d-k}$, and let
$$W'=W^{\circ} \cup_{S^{k-1} \times D^{d-k+1}} (D^{k} \times S^{d-k})$$
be the result of doing a surgery on $\phi$ to $W$. There is a
canonical embedding $\phi': D^{k} \times S^{d-k} \to W'$, and if
we do surgery on $\phi'$, we recover $W$. Note that the restriction of the psc metric
$g_{\round}^{k-1} + g_{\tor}^{d-k+1}$ on $S^{k-1} \times D^{d-k+1}$ to the boundary $S^{k-1} \times S^{d-k}$ is $g_{\round}^{k-1} + g_{\round}^{d-k}$, by the definition of a
torpedo metric. Similarly, the restriction of the psc metric $g_{\tor}^{k} + g_{\round}^{d-k}$ on $D^{k} \times S^{d-k}$ to the boundary is $g_{\round}^{k-1} + g_{\round}^{d-k}$. Therefore we get
maps
\begin{equation}\label{sec2:eq2}
\begin{gathered}
\xymatrix{   
\Riem^+ (W; \phi, g_\round^{k-1})\ar[d]^{\iota_0} \!\!\!\!\!\!\!\!\!\!\!\!\!\!\!\!&\cong &\!\!\!\!\!\!\!\!\!\!\!\!\!\!\!\!
  \Riem^+ (W' ; \phi', g_\round^{d-k}) \ar[d]^{\iota_1}
\\
\Riem^+ (W) & & \Riem^+ (W')}
\end{gathered}
\end{equation}
By Theorem \ref{chernysh-walsh-theorem}, the map $\iota_0$
($\iota_1$, respectively) is a homotopy equivalence if $d-k+1
\geq 3$ (if $k\geq 3$, respectively).
\end{proof}

The cobordism invariance of the space $\Riem^+ (W)$ for closed, simply-connected, spin manifolds of dimension at least five follows by the
same use of Smale's handle cancellation technique as in
\cite{GL}.

\subsubsection{Existence of stabilising metrics}

We use Theorem \ref{chernysh-walsh-theorem} to deduce the existence of psc metrics $g$ on certain cobordisms $K$ such that the gluing map $\mu_{g}$ is a homotopy equivalence.

\begin{thm}\label{lemma:GluingNullCob}
Let $d \geq 5$ and $M^{d-1} $ be a closed simply-connected spin manifold. Let $K : M \leadsto M$ be a cobordism which is simply-connected and spin, and which is in turn spin cobordant to $[0,1] \times M$ relative to its boundary. Then for any boundary condition $h \in \cR^+(M)$ there is a $g \in \cR^+(K)_{h,h}$ with the following property: if $W: N_0 \leadsto M$ and $V: M \leadsto N_1$ are cobordisms, and $h_i \in \Riem^+ (N_i)$ are boundary conditions, then the two gluing maps 
\begin{align*}
\mu(-, g) : \cR^+(W)_{h_0,h} &\lra \cR^+(W \cup_M K)_{h_0,h}\\
\mu(g,-) : \cR^+(V)_{h,h_0} &\lra \cR^+(K \cup_M V )_{h,h_0}
\end{align*}
are homotopy equivalences.
\end{thm}
\begin{proof}
By assumption there is a relative spin cobordism $X^{d+1}$ from $K$ to $[0,1] \times M$. Since $\dim (X) \geq 6$, by doing surgery in the interior of $X$ we can achieve that $X$ is $2$-connected, so the inclusions $[0,1] \times M \to X$ and $K \to X$ are both $2$-connected maps. By an application of Smale's handle cancellation technique to $X$, we can assume that the cobordism $X$ is obtained by attaching handles of index $3 \leq k \leq d-2$ to the interior of either of its boundaries. (A reference which discusses handle cancellation for cobordisms between manifolds with boundary is \cite{WallGC}.)

Let $g\in \cR^+(K)_{h,h}$ be a psc metric, and $\phi : S^{k-1} \times D^{d-k+1} \hookrightarrow K$ be a piece of surgery data in the interior of $K$ such that surgery along it yields a manifold $K'$ (this corresponds to a surgery of index $k$) and suppose that $3 \leq k \leq d-2$. Let $g' \in \cR^+(K')_{h,h}$ be in the path component corresponding to that of $g$ under the surgery equivalence $\SE_{\phi}: \cR^+(K)_{h,h} \simeq \cR^+(K')_{h,h}$ of Theorem \ref{surgery-invariance}.

As in the proof of Theorem \ref{surgery-invariance} there is a commutative diagram
\begin{equation*}
\xymatrix{
\cR^+(W)_{h_0, h} \times \cR^+(K)_{h,h} \ar[r] &  \cR^+(W \cup K)_{h_0, h}\\
\cR^+(W)_{h_0, h} \times \cR^+(K; \phi, g_\round^{k-1})_{h,h} \ar[r] \ar[u] \ar[d] &  \cR^+(W \cup K; \phi, g_\round^{k-1})_{h_0, h} \ar[u] \ar[d]\\
\cR^+(W)_{h_0, h} \times \cR^+(K')_{h,h} \ar[r] &  \cR^+(W \cup K')_{h_0, h}
}
\end{equation*}
where all the vertical maps are homotopy equivalences (since $2 \leq k-1 \leq d-3$). Thus gluing on the metric $g \in \cR^+(K)_{h,h}$ from the right induces a homotopy equivalence if and only if gluing on the corresponding metric $g' \in \cR^+(K')_{h,h}$ does. The same is true for gluing in metrics from the left. 

Gluing $([0,1] \times M, dt^2 + h)$ on to either side induces a homotopy equivalence, by Corollary \ref{cor:collarstretching}. 
By Theorem \ref{surgery-invariance} the cobordism $X$ induces a surgery equivalence $\cR^+(K)_{h,h} \simeq \cR^+([0,1] \times M)_{h,h}$, so if we let $g \in \cR^+(K)_{h,h}$ be in a path component corresponding to that of $h + dt^2$ under the surgery equivalence then gluing on $(K, g)$ from either side also induces a homotopy equivalence, as required.
\end{proof}

\section{The secondary index invariant}\label{chap:indextheory} 

This chapter contains the index theoretic arguments that go into the proof of our main results. We begin by stating our framework for $K$-theory in Section \ref{generalities-k-theory}. Then we recall the basic properties of the Dirac operator on a spin manifold and on bundles of spin manifolds, including those with noncompact fibres, in Section \ref{sec:generaldirac}. These analytical results allow the definition of the secondary index invariant, $\inddiff$, to be presented in Section \ref{sec:inddiff}. 

Conceptually simple as the definition of $\inddiff$ is, it seems to be impossible to compute directly. The purpose of the rest of this chapter is to provide computational tools. One computational strategy is to use the additivity property of the index, which results in a cut-and-paste property for the index difference. This is done in Section \ref{sec:additivity}.
The other computational strategy is to relate the secondary index to a primary index. In Section \ref{sec:abstract-nonsense}, we describe the abstract setting necessary to carry out such a comparison. This is then applied in two different situations. The first is the passage from even to odd dimensions, in other words the derivation of Theorem \ref{factorization-theorem:odd-case} from Theorem \ref{factorization-theorem:even-case}, which is carried out in Section \ref{sec:increasing-dimension}.
The second situation in which we apply the general comparison pattern is when we compute the index difference by a family index in the classical sense. This will be essential for the proof of Theorem \ref{factorization-theorem:even-case}, and is done in Section \ref{sec:FibBunIndDiff}. The classical index can be computed using the Atiyah--Singer index theorem for families of Clifford-linear differential operators, which we first discuss in Section \ref{sec:atiyah-singer}.
Also in Section \ref{sec:FibBunIndDiff}, the index theorem is interpreted in homotopy-theoretic terms, and there, another key player of this paper enters the stage: the Madsen--Tillmann--Weiss spectra.

\subsection{Real K-theory}\label{generalities-k-theory}

The homotopy theorists' definition of real $K$-theory is in terms of the periodic $K$-theory spectrum $\KO$. By definition, the $KO$-groups of a CW-pair $(X,Y)$ are given by 
$$KO^k (X,Y) := [(X,Y), (\loopinf{-k}\KO,*)].$$
In general, as is usual in homotopy theory, one first replaces a space pair by a weakly equivalent CW-pair to which one applies the above definition, but in Sections \ref{chap:indextheory} and \ref{sec:PfMain} we shall take care to only apply it to pairs of the homotopy type of CW-pairs.

The specific choice of a model for $\KO$ is irrelevant, as long as one considers only spaces having the homotopy type of CW complexes. For index theoretic arguments, we use the Fredholm model, which we now briefly describe. Our model is a variant of a classical result by Atiyah--Singer \cite{AS69} and Karoubi \cite{karoubi-espaces}. More details and further references can be found in \cite[\S 2]{Eb13}. We begin by recalling some subtleties concerning Hilbert bundles and their maps.

\subsubsection{Hilbert bundles}\label{hilbert-bundles}
Let $X$ be a space (usually paracompact and Hausdorff) and let $H \to X$ be a real or complex Hilbert bundle. For us, a Hilbert bundle will always have separable fibres and the unitary group with the compact-open topology as a structure group. An \emph{operator family} $F: H_0 \to H_1$ is a fibre-preserving, fibrewise linear continuous map. It is determined by a collection $(F_x)_{x \in X}$ of bounded operators $F_x: (H_0)_x \to (H_1)_x$. Some care is necessary when defining properties of the operator family $F$ by properties of the individual operators $F_x$. We will recall the basic facts and refer the reader to \cite[\S 2.3]{Eb13} for a more detailed discussion.
An operator family $F$ is \emph{adjointable} if the collection of adjoints $(F_x^*)$ also is an operator family. The algebra of adjointable operator families on $H$ is denoted by $\Lin_X (H)$. There is a notion of a \emph{compact} operator family which is due to Dixmier--Douady \cite[\S 22]{DixDou}; the set of compact operator families is denoted $\Kom_X (H)$ and is a $*$-ideal in $\Lin_X(H)$. A \emph{Fredholm family} is an element in $\Lin_X (H)$ which is invertible modulo $\Kom_X (H)$. 

The reader is warned that being compact (or Fredholm) is a stronger condition on an operator family $F$ than just saying that all $F_x$ are compact (or Fredholm), and it can be difficult to check in concrete cases. To prove that a given operator family is compact (or invertible, or Fredholm), one can use the following sufficient criterion \cite[Lemma 2.16]{Eb13}. To state the criterion, let us say that $F:H_0 \to H_1$ is \emph{locally norm-continuous} if each point $x \in X$ \emph{admits} a neighborhood $U$ and trivialisations of $H_i|_U$ such that in this trivialisation, $F$ is given by a continuous map $U \to \Lin ((H_0)_x, (H_1)_x)$ (with the norm topology in the target). If $F$ is locally norm-continuous and each $F_x$ is compact (or invertible, or Fredholm), then $F$ is compact (or invertible, or Fredholm), at least when the base space $X$ is paracompact. One has to keep in mind that local norm-continuity always refers to a specific local trivialization. Therefore, the composition $H_0 \stackrel{F_0}{\to} H_1  \stackrel{F_1}{\to} H_2$ of two locally norm-continuous operator families is again locally norm-continuous only if $F_0$ and $F_1$ are locally norm continuous with respect to the same local trivialization of $H_1$. 

\subsubsection{Clifford bundles}

\begin{definition}
Let $V\to X$ be a Riemannian vector bundle and let $\tau: V \to V$ be a self-adjoint involution on $V$. A \emph{$\Cl(V^\tau)$-Hilbert bundle} over $\bK=\bR$ or $\bC$ is a triple $(H, \iota,c)$, where $H \to X$ is a $\bK$-Hilbert bundle (always assumed to have separable fibres), $\iota$ is a $\bZ/2$-grading (i.e.\ a self-adjoint involution) on $H$ and $c =(c_x)_{x \in X}$ is a collection of linear maps $c_x: V_x \to \Lin (H_x)$ such that
\begin{enumerate}[(i)]
 \item For all $v,v' \in V_x$, the following identities hold:
\begin{align*}
c_x(v) \iota + \iota c_x(v) &=0\\
c_x(v)^* &= -c_x(\tau v)\\
c_x (v) c_x(v')+ c_x(v') c_x (v) &= -2 \langle v, \tau v' \rangle .
\end{align*}
\item If $s \in \Gamma(X;V)$ is a continuous section, then the collection $(c_x (s(x))_{x \in X}$ of bounded operators is an element of $\Lin_X (H)$.
\end{enumerate}
\end{definition}


To ease notation, we typically write $c(v):=c_x (v)$. If the grading and Clifford multiplication is understood, we denote a $\Cl(V^\tau)$-Hilbert bundle simply by the letter $H$.
The \emph{opposite} $\Cl(V^\tau)$-Hilbert bundle has the same underlying Hilbert bundle, but the Clifford multiplication and grading are replaced by $-c$ and $-\iota$. 
We write ``$\Cl (V^+ \oplus W^-)$-Hilbert bundle'' when the involution $\tau(v,w)=(v,-w)$ on $V \oplus W$ is considered. We denote by $\bR^{p,q}$ the space $\bR^{p+q}$, with the standard scalar product and the involution $\tau(v,w)=(v,-w)$, $v \in \bR^p$, $w \in \bR^q$ (note that this convention differs from that in \cite{AtKR}, we think it is easier to memorise). We abbreviate the term ``$\Cl ((X \times \bR^p)^+\oplus (X \times \bR^q)^- )$-Hibert bundle'' to ``$\Cl^{p,q}$-Hilbert bundle''. Instead of ``finite-dimensional $\Cl(V^+ \oplus W^-)$-Hilbert bundle'', we will rather say \emph{$\Cl(V^+ \oplus W^-)$-module}.
A $\Cl^{p,q}$-\emph{Fredholm family} is a Fredholm family such that $F \iota=-\iota F$ and $Fc(v) = c(v)F$ for all $v \in \bR^{p,q}$.

\subsubsection{$K$-theory}
We denote the product of pairs of spaces by $(X,Y) \times (A,B) :=(X \times A, X \times B \cup Y \times A)$. A \emph{$(p,q)$-cycle} on $X$ is a tuple $(H,\iota,c,F)$, consisting of a $\Cl^{p,q}$-Hilbert bundle and a $\Cl^{p,q}$-Fredholm family. 
If $Y \subset X$ is a subspace, then a \emph{relative} $(p,q)$-cycle is a $(p,q)$-cycle $(H,\iota,c,F)$ with the additional property that the family $F|_Y$ is invertible. Clearly $(p,q)$-cycles can be pulled back along continuous maps, and there is an obvious notion of isomorphism of $(p,q)$-cycles and of direct sum of finitely many $(p,q)$-cycles.
A \emph{concordance} of relative $(p,q)$-cycles $(H_i,\iota_i,c_i,F_i)$, $i=0,1$ on $(X,Y)$ consists of a relative $(p,q)$-cycle $(H,\iota,c,F)$ on $(X ,Y) \times [0,1]$ and isomorphisms $(H,\iota,c,F)|_{X \times \{i\}}\cong (H_i,\iota_i,c_i, F_i)$. A $(p,q)$-cycle $(H,\iota,c,F)$ is \emph{acyclic} if $F$ is invertible.
Often, we abbreviate $(H,\iota,c,F)$ to $(H,F)$ if there is no risk of confusion. Occasionally, we write $x \mapsto (H_x,F_x)$ to describe a $(p,q)$-cycle on $X$.

\begin{definition}\label{defn-fpq}
Let $X$ be a paracompact Hausdorff space and $Y \subset X$ be a closed subspace. The group $F^{p,q} (X,Y)$ is the quotient of the abelian monoid of concordance classes of relative $(p,q)$-cycles, divided by the submonoid of those concordance classes which contain acyclic $(p,q)$-cycles.
\end{definition}

The monoid so obtained is in fact a group, and the additive inverse is given by 
\begin{equation}\label{eq:KOInv}
-[H,\iota,c,F] = [H,-\iota,-c,F],
\end{equation}
see \cite[Lemma 2.19]{Eb13}. 

We let $(\bI,\partial \bI):=([-1,1], \{-1,1\})$ and use the following notation:
\begin{equation*}
\Omega F^{p,q} (X,Y):= F^{p,q} ((X,Y)  \times (\bI,  \partial \bI)).
\end{equation*}

Bott periodicity \cite[\S 2.4]{Eb13} in this setting states that the map
\begin{align}\label{bott-map}
\bott:F^{p,q} (X,Y) &\lra \Omega F^{p-1,q} (X,Y)\\
(H,F) &\longmapsto \left((x,s) \mapsto (H_x,F_x + s \iota_x c(e_1)_x )\right) 
\end{align}
is an isomorphism of abelian groups (here $e_i$ denotes the $i$th basis vector of $\bR^p$). 
By iteration, we get an isomorphism 
\begin{equation}\label{eqn:bott-isomorphism}
 F^{p,0} (X,Y) \lra \Omega^p F^{0,0} (X,Y).
\end{equation}

\subsubsection{Classifying spaces and relation to the homotopical definition of $K$-theory}
In this section we shall explain the relation between $F^{p,q}$ and the $KO$-groups, and in particular produce a comparison map $F^{p,q} (X,Y) \to KO^{q-p}(X,Y)$.

First, we recall a classical result by Atiyah--Singer \cite{AS69} and Karoubi \cite{karoubi-espaces}. A $\Cl^{p,q}$-Hilbert space is \emph{ample} if it contains any finite-dimensional irreducible $\Cl^{p,q}$-Hilbert space with infinite multiplicity, and we fix such an ample $\Cl^{p,q}$-Hilbert space $U$. Let $\Fred^{p,q}$ be the space of $\Cl^{p,q}$-Fredholm operators on $U$, with the norm topology and let $\cG^{p,q} \subset \Fred^{p,q}$ be the (contractible) space of invertible operators. These spaces are open subsets of Banach spaces and hence are paracompact and have the homotopy type of CW complexes, by \cite[Theorem 13]{PalaisInfMan} and \cite[Proposition A.11]{Hatcher}.
There is a map
$$\bott: \Fred^{p,q}\lra \map ((\bI,\partial \bI),(\Fred^{p-1,q},\cG^{p-1,q})) \simeq \Omega \Fred^{p-1,q},$$
defined by a formula analogous to \ref{bott-map} and the main result of \cite{AS69} asserts that it is a homotopy equivalence. Moreover, it is proven in \cite{AS69} that there are homotopy equivalences 
\[
(\Fred^{p,q},\cG^{p,q}) \simeq (\Omega^{\infty+p-q} \KO,*)  
\]
of space pairs. In \cite[Definition A.3]{Eb13}, a coarser topology on $\Fred^{p,q}$ is defined, and with the new topology, the space pair is denoted $(K^{p,q},D^{p,q})$. We have the following comparison result.

\begin{thm}\cite[Theorem 2.21 and Theorem 2.22]{Eb13}\label{atsingkar}\mbox{}
\begin{enumerate}[(i)]
\item There is a natural map $[(X,Y);(K^{p,q},D^{p,q})] \to F^{p,q}(X,Y)$ which is bijective if $X$ is paracompact and compactly generated and if $Y \subset X$ is closed.
\item The identity map $(\Fred^{p,q},\cG^{p,q})\to (K^{p,q},D^{p,q})$ is a weak equivalence of pairs.
\end{enumerate}
\end{thm}
Hence there are maps
\begin{equation*}
KO^{q-p}(X,Y) \longleftarrow [(X,Y);(\Fred^{p,q},\cG^{p,q})] \lra [(X,Y);(K^{p,q},D^{p,q})] \lra F^{p,q}(X,Y)
\end{equation*}
which are isomorphisms when $(X,Y)$ has the homotopy type of a CW pair. Therefore for such a pair $(X,Y)$ and any class $\gb \in F^{p,q}(X,Y)$, we obtain a homotopy class of maps $(X,Y) \to (\loopinf{+p-q}\KO,*)$, the \emph{homotopy-theoretic realisation of $\gb$}. We shall use these isomorphisms to abuse notation slightly, and for pairs with the homotopy type of CW-pairs we shall from now on write $KO^{-p}(X,Y) = F^{p,0}(X,Y)$. 

For a general paracompact space pair $(X,Y)$, we at least have a comparison map $F^{p,q} (X,Y) \to KO^{q-p}(X,Y)$: a CW approximation $(X',Y') \to (X,Y)$ induces a map $F^{p,q}(X,Y) \to F^{p,q}(X',Y') \cong KO^{q-p}(X',Y') $, and the latter group is isomorphic to $KO^{q-p}(X,Y)$, by the homotopy-theoretic definition of the $KO$-groups. We will use the comparison map to identify elements in $F^{p,q}(X,Y)$ with their image in $KO^{q-p}(X,Y)$. That the comparison map is not an isomorphism in full generality does not matter to us, since we only use it to construct elements in $KO^{q-p}(X,Y)$ out of analytical data.


\subsection{Generalities on Dirac operators}\label{sec:generaldirac}

\subsubsection{The spin package}\label{subsection:spinpackage}

A basic reference for spin vector bundles and associated constructions is \cite[\S II.7]{SpinGeometry}.
Let $V \to X$ be a real vector bundle of rank $d$. A \emph{topological spin structure} is a reduction of the structure group of $V$ to the group $\widetilde{\GL}^{+}_{d} (\bR)$, the connected $2$-fold covering group of the group $\GL_d^+ (\bR)$ of matrices of positive determinant. In the presence of a Riemannian metric on $V$, a topological spin structure induces a reduction of the structure group to $\Spin (d)$, which is the familiar notion of a spin structure. Explicitly, a spin structure is given by a $\Spin (d)$-principal bundle $P \to X$ and an isometry $\eta:P \times_{\Spin (d)} \bR^d \cong V$. From a spin structure on $V$, one can construct a fibrewise irreducible real $\Cl (V^+ \oplus \bR^{0,d})$-module $\spinor_V$, the \emph{spinor bundle}. One can reconstruct $P$ and $\eta$ from $\spinor_V$, and it is often more useful to view the spinor bundle $\spinor_V$ as the more fundamental object. The \emph{opposite} spin structure $\spinor_V^{op}$ has the same underlying vector bundle as $\spinor_V$ and the 
same Clifford multiplication, but the grading is inverted (note that this is \emph{not} the opposite bundle in the sense of the previous section). 
If $V=TM$ is the tangent bundle of a Riemannian manifold $M^d$ with metric $g$, we denote spin structures typically by $\spinor_M$. There is a canonical connection $\nabla$ on $\spinor_M$ derived from the Levi-Civita connection on $M$. The spin Dirac operator $\Dir=\Dir_g$ acts on sections of $\spinor_M$ and is defined as the composition
$$\Gamma (M; \spinor_M) \stackrel{\nabla}{\lra} \Gamma (M; TM \otimes \spinor_M) \stackrel{c}{\lra} \Gamma(M; \spinor_M).$$
It is a linear formally self-adjoint elliptic differential operator of order $1$, which anticommutes with the grading and Clifford multiplication by $\bR^{0,d}$. We can change the $\Cl^{0,d}$-multiplication on $\spinor_M$ to a $\Cl^{d,0}$-multiplication by replacing $c(v) $ by $\iota c(v)$. With this new structure, $\Dir$ becomes $\Cl^{d,0}$-linear. Passing to the opposite spin structure leaves the operator $\Dir$ unchanged, but changes the sign of the grading and of the $\Cl^{d,0}$-multiplication. The relevance of the Dirac operator to scalar curvature stems from the well-known \emph{Schr\"odinger--Lichnerowicz formula} (also known as \emph{Lichnerowicz--Weitzenb\"ock formula}) \cite{Schroed}, \cite{Lich}, or \cite[Theorem II.8.8]{SpinGeometry}:
\begin{equation}\label{weitzenboeck}
\Dir^2 = \nabla^* \nabla + \frac{1}{4} \scal(g).
\end{equation}

\subsubsection{The family case}
We need to study the Dirac operator for families of manifolds and also for nonclosed manifolds. Let $X$ be a paracompact Hausdorff space. We study bundles $\pi:E \to X$ of possibly \emph{noncompact} manifolds with $d$-dimensional fibres. The fibres of $\pi$ are denoted $E_x:=\pi^{-1}(x)$. The vertical tangent bundle is $T_v E \to E$ and we always assume implicitly that a (topological) spin structure on $T_v E$ is fixed (then of course the fibres are spin manifolds).
A fibrewise Riemannian metric $(g_x)_{x \in X}$ on $E$ then gives rise to the spinor bundle $\spinor_E \to E$, a $\Cl ((T_v E)^+ \oplus \bR^{0,d})$-module. The restriction of the spinor bundle to the fibre over $x$ is denoted $\spinor_x \to E_x$, and the Dirac operator on $E_x$ is denoted $\Dir_x$. 
When we consider bundles of noncompact manifolds, they are always required to have a simple structure outside a compact set. 

\begin{definition}\label{definition:cylindricalends}
Let $\pi:E \to X$ be a bundle of noncompact $d$-dimensional spin manifolds. Let $t:E \to \bR$ be a fibrewise smooth function such that $(\pi,t):E \to X \times \bR$ is proper. Let $a_0< a_1: X \to [-\infty, \infty]$ be continuous functions. Abusing notation, we denote $X \times (a_0,a_1)=\{ (x,s)\in X \times \bR \,|\, a_0 (x) < s < a_1(x)\}$ and $E_{(a_0,a_1)}:= (\pi,t)^{-1} (X \times (a_0,a_1))$. 
For a closed $(d-1)$-dimensional spin manifold $M$, we consider the trivial bundle $X \times \bR \times M \to X$, with the obvious projection map to $\bR$. We say that $E$ \emph{is cylindrical over $(a_0,a_1)$} if the projection map $E_{(a_0,a_1)} \to X  \times (a_0,a_1)$ is a smooth fibre bundle and if there is a $(d-1)$-dimensional spin manifold $M$ such that there is an isomorphism $E_{(a_0,a_1)} \cong ( X \times \bR \times M)_{(a_0,a_1)}$ of spin manifold bundles over $X \times (a_0,a_1)$. The bundle is said to have \emph{cylindrical ends} if there are functions $a_-,a_+: X\to \bR$ such that $E$ is cylindrical over $(-\infty,a_-)$ and $(a_+,\infty)$.
A fibrewise Riemannian metric $g=(g_x)_{x \in X}$ on $E$ is called cylindrical over $(a_0,a_1)$ if $g_x=dt^2+h_x $ for some metric $h_x$ on $M$, over $(a_0,a_1)$. We always consider Riemannian metrics which are cylindrical over the ends.
We say that a bundle with cylindrical ends and metrics has \emph{positive scalar curvature at infinity} if there is a function $\epsilon: X \to (0, \infty)$, such that the metrics on the ends of $E_x$ have scalar curvature $\geq \epsilon(x)$.
\end{definition}

Any bundle of closed manifolds, equipped with an arbitrary function $t$, tautologically has cylindrical ends and positive scalar curvature at infinity.
Note that the fibres of a bundle with cylindrical ends are automatically complete in the sense of Riemannian geometry.
Fibre bundles $\pi:E \to X$ of spin manifolds with boundary can be fit into the above framework by a construction called \emph{elongation}. 

\begin{defn}
Let $\pi:E \to X$ be a bundle of compact manifolds with collared boundary, and assume that the boundary bundle is trivialised: $\partial E = X \times N$ as a spin bundle. Let $g=(g_x)_{x \in X}$ be a fibrewise metric on $E$ so that $g_x$ is of the form $dt^2 +h_x $ near the boundary, for psc metrics $h_x \in \Riem^+ (N)$. The \emph{elongation} of $(E,g)$ is the bundle $\hat{E}= E \cup_{\partial E} (X  \times [0,\infty)\times N)$, with the metric $(dt^2 + h_x )$ on the added cylinders. The elongation has cylindrical ends and positive scalar curvature at infinity. 
\end{defn}

\subsubsection{Analysis of Dirac operators}
Let $E \to X$ be a spin manifold bundle, equipped with a Riemannian metric $g$, and with cylindrical ends.
Let $L^2 (E;\spinor)_x$ be the Hilbert space of $L^2$-sections of the spinor bundle $\spinor_x \to E_x$. These Hilbert spaces assemble to a $\Cl^{d,0}$-Hilbert bundle on $X$. The Dirac operator $\Dir_x$ is a densely defined symmetric unbounded operator on the Hilbert space $L^2 (E_x; \spinor_x)$, with initial domain $\Gamma_c (E_x, \spinor_x)$, the space of compactly supported sections. The domain of its closure is the Sobolev space $W^1 (E_x; \spinor_x)$, and the closure of $\Dir_x$ is self-adjoint (this is true because $E_x$ is a complete manifold). 
Thus we can apply the functional calculus for unbounded operators to form $f(\Dir_x)$, for (say) continuous bounded functions $f: \bR \to \bC$. In particular, we can take $f(x)=\normalize{x}$ and obtain the \emph{bounded transform}
$$F_x := \normalize{\Dir_x}.$$
\emph{From now on, we make the crucial assumption that the scalar curvature of $g$ is positive at infinity}.
This has the effect that $\normalize{\Dir_x}$ is a bounded $\Cl^{d,0}$-\emph{Fredholm} operator. The collection of operators $(F_x)_{x \in X}$ is a $\Cl^{d,0}$-Fredholm family over $X$, which is moreover locally norm-continuous in the sense of Section \ref{hilbert-bundles}. These are standard facts, we refer to \cite[\S 3.1]{Eb13} for detailed proofs geared to fit into the present framework. Furthermore, if for each $y \in Y \subset X$ the metric $g_y$ has positive scalar curvature (on all of $E_y$, not only at the ends), then the operator $\Dir_y$ has trivial kernel, by the Bochner method using the Schr\"odinger--Lichnerowicz formula \eqref{weitzenboeck}, cf.\ \cite[II Corollary 8.9]{SpinGeometry}. Consequently, the operator $F_y$ is invertible for all $y \in Y$.

\begin{definition}
Let $X$ be a paracompact Hausdorff space and $Y \subset X$ be a closed subspace. Let $\pi:E \to X$ be a bundle of Riemannian spin manifolds with cylindrical ends and positive scalar curvature at infinity, with metric $g$. We denote by $\indx{E,g}$ the $(d,0)$-cycle $x \mapsto (L^2 (E_x, \spinor_{E_x}), F_x)$ 
and by $\ind{E,g}$ its class in $KO^{-d}(X)$. If it is understood that the metric $g_y$ has positive scalar curvature for all $y \in Y$, we use the same symbol to denote the class in the relative $K$-group $KO^{-d}(X, Y)$.
\end{definition}

When we meet a bundle with boundary, by the symbols $\indx{E,g}$ and $\ind{E,g}$ we always mean the index of the elongated manifolds.

\begin{lem}\label{index:independent-of-metric}
Let $\pi: E \to X$ be a spin manifold bundle with cylindrical ends and let $g_0$, $g_1$ be fibrewise metrics which both have positive scalar curvature at infinity. Assume that $g_0$ and $g_1$ agree on the ends. Let $Y \subset X$ and assume that $g_0$ and $g_1$ agree over $Y$ and have positive scalar curvature there.
Then $\ind{E,g_0}=\ind{E,g_1}\in KO^{-d}(X,Y)$.
\end{lem}
\begin{proof}
This follows from the homotopy invariance of the Fredholm index or by considering the metric $(1-t)g_0+t g_1$ on the product bundle $E \times I \to X \times I$, which gives a concordance of cycles. 
\end{proof}
In particular, if the fibres of $\pi$ are closed and if $Y=\emptyset$, then $\ind{E,g}\in KO^{-d}(X)$ does not depend on $g$ at all. This observation justifies the notation $\ind{E}\in KO^{-d}(X)$ for closed bundles and $\ind{E,g}\in KO^{-d}(X)$ when $E$ is a bundle with cylindrical ends, and the psc metric $g$ is only defined on the ends. In that case, $g=dt^2 +h$, and we could also write $\ind{E,h}$, emphasising the role of $h$ as a boundary condition. Finally, we remark that if we pass to the opposite bundle $E^{op} \to X$ (with the opposite spin structure), then 
\begin{equation}\label{formula:index-opposite-bundle}
\ind{E^{op},g} =-\ind{E,g}.
\end{equation}
This follows from the definition of the opposite spin structure given in subsection \ref{subsection:spinpackage} and from the formula \eqref{eq:KOInv} for the additive inverse of a $K$-theory class.

\subsection{The secondary index invariants}\label{sec:inddiff}

We now define the secondary index invariant, the \emph{index difference}. There are two definitions of this invariant which we will consider. The first one is due to Hitchin \cite{HitchinSpin} and we take it as the main definition. 

\subsubsection{Hitchin's definition}

Let $W$ be a manifold with collared boundary $M$, and let $h \in \Riem^+(M)$. On the space $\bI \times \Riem^+ (W)_h \times \Riem^+ (W)_h$, we consider the elongation of the trivial bundle with fibre $W$, and introduce the following fibrewise Riemannian metric $g$: on the fibre over $(t,g_0,g_1)$ it is $\frac{1-t}{2}g_0 + \frac{1+t}{2}g_1$ (since both metrics agree on $M$, this has positive scalar curvature at infinity). For $t=\pm 1 $, this metric has positive scalar curvature and thus applying the results from the previous section, we get an element
\begin{equation}\label{inddiff-path-bounded-case}
\inddiff :=\ind{\bI\times \Riem^+ (W)_h \times \Riem^+ (W)_h \times W,g} \in \Omega KO^{-d} (\Riem^+ (W)_h \times \Riem^+ (W)_h,\Delta)
\end{equation}
(where $\Delta$ is the diagonal) which is the path space version of the \emph{index difference}. 

\begin{rem}
One can phrase this construction in a slightly imprecise but conceptually enlightening way. Each pair $(g_0,g_1)$ of psc metrics defines a path $t \mapsto \Dir_{\frac{1-t}{2}g_0 + \frac{1+t}{2}g_1}$ in $F^{d,0} $, which for $t=\pm 1$ is invertible, and hence in the contractible subspace $D^{d,0}$. The space of all paths $\gamma:(\bI,\{\pm 1\}) \to (F^{d,0}, D^{d,0})$ is homotopy equivalent to the loop space $\Omega F^{d,0}$. Since the spinor bundle depends on the underlying metric and thus the operators $\Dir_{\frac{1-t}{2}g_0 + \frac{1+t}{2}g_1}$ do not act on the same Hilbert space, this does not strictly make sense.
\end{rem}

If we keep a basepoint $g \in \Riem^+ (W)_h$ fixed, we get an element 
$$\inddiff_g \in \Omega KO^{-d} (\Riem^+ (W)_h,g)$$
by fixing the first variable. Using Theorem \ref{atsingkar}, we can represent this element in a unique way by a homotopy class of pointed maps 
$$\inddiff_g:(\Riem^+ (W)_h,g) \lra (\loopinf{+d+1} \KO,*).$$

\begin{rem}\label{rem:NoExtension}
Note that it is important that the metrics are fixed on $M$, since otherwise the metrics $\frac{1-t}{2}g_0 + \frac{1+t}{2}g_1$ might not have positive scalar curvature at infinity. In fact, if $\partial W \neq \emptyset$, then there does not exist an extension of the map $\inddiff_g$ to $\Riem^+ (W) \supset \Riem^+ (W)_h$, see Corollary \ref{cor:inddiff-doesnt-extend} below.
\end{rem}

\subsubsection{Gromov--Lawson's definition}
The second version of the index difference is due to Gromov and Lawson \cite{GL2}, and we will use it as a computational tool. We define it (and use it) only for \emph{closed} manifolds. 
Let $W$ be a $d$-dimensional closed manifold and consider the trivial bundle over $\Riem^+ (W) \times \Riem^+ (W)$ with fibre $\bR \times W$. Choose a smooth function $\varphi: \bR \to [0,1]$ that is equal to $0$ on $(-\infty,0]$ and equal to $1$ on $[1,\infty)$. Equip the fibre over $(g_0,g_1)$ with the metric $h_{(g_0,g_1)}:=dt^2+(1-\varphi(t))g_0 + \varphi(t)g_1 $, which has positive scalar curvature at infinity and so gives an element
\begin{equation}\label{inddiff-aps}
 \inddiffaps := \ind{\Riem^+ (W) \times \Riem^+ (W) \times \bR \times W,h} \in KO^{-d-1} (\Riem^+ (W) \times \Riem^+ (W),\Delta)
\end{equation}
(the degree shift appears since $ \bR \times W$ has dimension $d+1$). Again, we obtain $\inddiffaps_g \in KO^{-d-1} (\Riem^+ (W), g)$ by fixing a basepoint.
The homotopy-theoretic realisations of both index differences are maps of pairs
\begin{equation}\label{indexdifferences}
\inddiff, \inddiffaps: (\Riem^+ (W) \times \Riem^+ (W),\Delta)  \lra (\loopinf{+d+1} \KO,*).
\end{equation}
The following theorem answers the obvious question whether both definitions of the index difference yield the same answer. It will be the main ingredient for the derivation of Theorem \ref{factorization-theorem:odd-case} from Theorem \ref{factorization-theorem:even-case}. It follows from a generalisation of the classical spectral flow index theorem \cite{RobSal} to the Clifford-linear family case.

\begin{thm}[Spectral flow index theorem \cite{Eb13}]\label{spectralflowindex}
For each closed spin manifold $W$, the two definitions of the secondary index invariant agree, i.e.\ the maps \eqref{indexdifferences} are weakly homotopic (cf.\ Definition \ref{defn:intro:weaklyhomotopic}).
\end{thm}

\subsection{The additivity theorem}\label{sec:additivity}

An efficient tool to compute the secondary index invariant is the additivity theorem for the index of operators on noncompact manifolds. There are several versions of this result in the literature, but the version that is most useful for our purposes is due to Bunke \cite{Bunke1995}.
\subsubsection{Statement of the additivity theorem}
\begin{assumptions}\label{assumpt-additivity}
Let $X$ be a paracompact Hausdorff space. Let $E \to X$ and $E' \to X$ be two Riemannian spin manifold bundles of fibre dimension $d$ with metrics $g$ and $g'$ and with cylindrical ends, such that $E$ and $E'$ have positive scalar curvature at infinity. Assume that there exist functions $a_0 < a_1:X \to \bR$ such that $E$ and $E'$ are cylindrical over $X \times (a_0,a_1)$ and agree there: $E_{(a_0,a_1)}=E'_{(a_0,a_1)}$ (more precisely, we mean that there exists a spin-preserving isometry of bundles of closed manifolds over $X \times (a_0,a_1)$). Assume that the scalar curvature on $E_{(a_0,a_1)}$ is positive. Let
\begin{equation*}
 E_0=E_{(-\infty,a_1)}; \; E_1=E_{(a_0,\infty)}; \; E_2=E'_{(-\infty,a_1)}; \; E_3=E'_{(a_0,\infty)}
\end{equation*}
and define $E_{ij}:=E_i \cup E_j$ for $(i,j) \in \{(0,1),(2,3),(0,3),(2,1)\}$. Note that $E= E_{01}$ and $E'= E_{23}$. These are bundles of spin manifolds with cylindrical ends, having positive scalar curvature at infinity. 
\end{assumptions}
\begin{thm}[Additivity theorem]\label{thm:additivity}
Under Assumptions \ref{assumpt-additivity}, we have
\begin{equation*}
 \ind{E_{01}} + \ind{E_{23}} = \ind{E_{03}}+\ind{E_{21}} \in KO^{-d}(X).
\end{equation*}
Furthermore, if both bundles $E_{01}$ and $E_{23}$ have positive scalar curvature over the closed subspace $Y \subset X$, then the above equation holds in $KO^{-d}(X,Y)$.
\end{thm}

If $X=*$, this is due to Bunke \cite{Bunke1995}, and the case of arbitrary $X$ and $Y =\emptyset$ is straightforward from his argument. For the case $Y \neq \emptyset$, we need to give an additional argument, and this forces us to go into some details of Bunke's proof. Furthermore, the setup used by Bunke is slightly different from ours (cf.\ \S 1.1 loc.cit.), and so we decided to sketch the full proof here.

\subsubsection{The proof of the additivity theorem}
Using \eqref{formula:index-opposite-bundle}, we have to prove that
\begin{equation}\label{eqn:additivityproof-1}
 \ind{E_{01}} + \ind{E_{23}} + \ind{E_{03}^{op}}+\ind{E_{21}^{op}}=0 \in KO^{-d}(X).
\end{equation}
Let $H_{ij}:= L^2 (E_{ij};\spinor)$ be the Hilbert bundle on $X$ associated with the bundle $E_{ij}$. The sum of the indices showing up in \eqref{eqn:additivityproof-1} is represented by the tuple $(H, \iota,c, F)$; the graded $\Cl^{d,0}$-bundle is $H:=H_{01}\oplus H_{23} \oplus H_{03}\oplus H_{21}$, with Clifford action, involution, and operator given by
\begin{align*}
c &:=\begin{pmatrix}
c_{01} & & & \\
&  c_{23} & & \\
 & &-c_{03} & \\
 & & & -c_{21}
\end{pmatrix} &
\iota &:=\begin{pmatrix}
\iota_{01} & & & \\
&  \iota_{23} & & \\
 & &-\iota_{03} & \\
 & & & -\iota_{21}
\end{pmatrix}\\
D &:=\begin{pmatrix}
\Dir_{01} & & & \\
&  \Dir_{23} & & \\
 & &\Dir_{03} & \\
 & & & \Dir_{21}
\end{pmatrix} &
F &:= \normalize{D}.
\end{align*}

Pick smooth functions $\lambda_0, \mu_0: \bR \to [0,1]$ with $\supp (\lambda_0)\subset [0, \infty)$, $\supp (\mu_0) \subset (-\infty,1]$ and $\mu^{2}_{0} + \lambda^{2}_{0} =1$; we can choose them to have $|\lambda_{0}'|, |\mu'_{0}| \leq 2$. We obtain functions $\mu, \lambda: X \times \bR \to [0,1]$ by $\mu (x,t)=\mu_0 (\frac{t-a_0(x)}{a_1 (x)-a_0(x)})$ and $\lambda (x,t)=\lambda_0 (\frac{t-a_0(x)}{a_1 (x)-a_0(x)})$. The formula
\begin{equation*}
J_0  :=\begin{pmatrix}
 & & -\mu & -\lambda\\
 & &-\lambda  &\mu \\
\mu &\lambda & & \\
\lambda & -\mu& &
\end{pmatrix} 
\end{equation*}
defines an operator on $H$ (the interpretation should be clear: for example, multiplying a spinor over $E_{01}$ by $\mu$ gives a spinor with support in $E_0$, and we can transplant it to $E_{03}$). Then $J:= J_0 \iota $ is an odd, $\Cl^{d,0}$-linear involution.

Bunke proves that the anticommutator $FJ+JF$ is compact (stated and proved as Lemma \ref{anticommutator-compact} below), whence 
\begin{equation}\label{additivity-proof-homotopy}
s \longmapsto \cos (s) F + \sin (s) J
\end{equation}
defines a homotopy from $F$ to $J$; since $J$ is invertible, the element $[H, \iota,c,F] \in KO^{-d}(X)$ is zero, which proves Theorem \ref{thm:additivity} for $Y= \emptyset$. 
Assume that the scalar curvature is positive over $Y \subset X$, so that $F$ is invertible over $Y$. If the homotopy \eqref{additivity-proof-homotopy} would go through invertible operators over $Y$, the proof of Theorem \ref{thm:additivity} would be complete. However, $s \mapsto \cos (s) F + \sin (s) J$ is not in general invertible if $F$ is, and so we have to adjust this homotopy.

\begin{lem}\label{bound-commutator}
There exists $C>0$ with the following property. Let $x \in X$ and $\ell_x := a_1 (x)-a_0(x)$ be the length of the straight cylinder in the middle. Then $\|F_x J_x+J_x F_x\| \leq \frac{C}{\ell_x}$.
\end{lem}

\begin{proof}
We ease notation by dropping the subscript $x$.
First, compute the anticommutator
\begin{equation*}
P:=DJ+JD =
\begin{pmatrix}
 & &- \mu' &- \lambda' \\
 & & -\lambda' & \mu'\\
 \mu' &  \lambda'& & \\
 \lambda' & -\mu' & &
\end{pmatrix} e \iota,
\end{equation*}
where $e$ denotes Clifford multiplication by the unit vector field $\partial_t$ on $E$.
The functions $\mu$ and $\lambda$ have been chosen so that $|\mu'|,|\lambda'|  \leq 2/\ell$. Thus $DJ+JD$ is bounded, and $\|DJ + JD\|\leq \frac{2\sqrt{2}}{\ell}$. As in \cite{Bunke1995}, we write $FJ+JF$ as an integral, starting with the absolutely convergent integral
\begin{equation*}
\frac{1}{(1+D^2)^{1/2}}=\frac{2}{\pi} \int_{0}^{\infty} (1 + D^2 + t^2)^{-1} dt \in \Lin (L^2 (E_x, \spinor_x)).
\end{equation*}
Let $Z(t):=(1 + D^2 + t^2)^{-1}$. For each $u \in W^1 (E_x; \spinor_x)$, we get
\begin{equation}\label{integral0}
( FJ+JF)u = \frac{2}{\pi} \int_{0}^{\infty}( Z(t)DJ+ J Z(t) D ) u dt \in L^2(E_x, \spinor_x),
\end{equation}
and we claim that the formula
\begin{equation}\label{integral}
FJ+JF = \frac{2}{\pi} \int_{0}^{\infty} Z(t)DJ+ J Z(t) D  dt \in \Lin (L^2 (E_x, \spinor_x))
\end{equation}
is true and the integral is absolutely convergent. To prove this, it is enough to show that the integral \eqref{integral} converges absolutely, the equality then follows from 
\eqref{integral0}. 
The operator $Z(t)DJ+ J Z(t) D$ is bounded, and we estimate its norm as follows. 
Note that $Z(t)$ and $D$ commute. Thus $Z(t) DJ + JZ(t)D= Z(t)P + [J,Z(t)]D$. Moreover, $[J,Z(t)] = Z(t)[Z(t)^{-1},J] Z(t)$ and $[Z(t)^{-1},J] =DP-PD$. Altogether, this shows that the integrand can be written as
\begin{equation}\label{integrand}
Z(t)DJ+ J Z(t) D  = Z(t) P - Z(t) P D^2 Z(t) + Z(t)D P Z(t) D
\end{equation}
(these formal computations can be justified by restricting to the dense domains).
We have the following estimates:
\begin{align*}
\| Z(t) \| &\leq \frac{1}{1+t^2}\\
\|D^2 Z(t)\| &\leq 1\\
\|Z(t) D\| &\leq \frac{1}{2\sqrt{1+t^2}}.
\end{align*}
The first two are clear, and the third follows from $\sup_{x} \frac{x}{1+x^2 + t^2} = \frac{1}{2\sqrt{1+t^2}}$. Therefore, the norm of the operator (\ref{integrand}) is bounded by
\begin{equation*}
\frac{2 \sqrt{2}}{\ell} \left(\frac{1}{1+t^2}  + \frac{1}{1+t^2} + \frac{1}{4(1+t^2)} \right) = \frac{ 9\sqrt{2}}{2\ell} \frac{1}{1+t^2}.
\end{equation*}
Therefore \eqref{integral} is true and the integral converges absolutely. Moreover, \eqref{integral} implies $\| FJ+JF\| \leq \frac{C}{\ell}$, with $C= \frac{9 \sqrt{2}}{2}$.
\end{proof}

\begin{lem}\label{anticommutator-compact}
The anticommutator $FJ+JF$ is compact.
\end{lem}

\begin{proof}
We use the integral formula \eqref{integral}, and revert to using the index $x$. By an application of \cite[Prop.\ 10.5.1]{HigRoe}, the operators $Z_x(t) P_x$ and $P_x Z_x (t) D_x$ are compact. Since $D^2_x Z_x (t)$ and $Z_x (t) D_x$ are bounded, it follows that the integrand \eqref{integrand}, namely $Z_x(t) P_x - Z_x(t) P_x D_x^2 Z_x(t) + Z_x(t)D_x P_x Z_x(t) D_x$, is compact.
Since the integral \eqref{integral} converges absolutely, we get that $F_x J_x+J_x F_x$ is a compact operator, for each $x \in X$.
This is not yet enough to guarantee compactness of $FJ+JF$, see the discussion in Section \ref{hilbert-bundles}. However, by \cite[Proposition 3.7]{Eb13}, the family $x \mapsto F_x$ is a Fredholm family, and the proof in loc.\ cit.\ shows that $F$ is locally norm-continuous, with respect to a local trivialization of the Hilbert bundle that comes from a local trivialization of the original fibre bundle which is compatible with the cylindrical structure at infinity. The operator family $J$ is locally norm-continuous with respect to the same local trivializations, and hence $FJ+JF$ is locally norm-continuous. Therefore, using the criterion mentioned in Section \ref{hilbert-bundles}, compactness of $FJ+JF$ follows.
\end{proof}

\begin{proof}[Proof of Theorem \ref{thm:additivity}]
Let $\kappa: Y \to (0, \infty)$ be a lower bound for the scalar curvature of $E$ and $E'$, i.e. a function such that for $y \in Y$
$$\kappa (y) \leq \scal(g_y), \scal (g'_y).$$
By the Schr\"odinger--Lichnerowicz formula we have $D_y^2 \geq \frac{\kappa(y)}{4}$. Consider the operator homotopy $H(s)=\cos (s) F + \sin(s) J$. By Lemma \ref{bound-commutator}
\begin{align*}
H(s)_y^2 &= \cos(s)^2 F_y^2 + \sin(s)^2 + \cos(s)\sin(s) (F_y J_y+J_y F_y)\\
& \geq \cos(s)^2 \frac{\kappa(y)/4}{1+\kappa(y)/4} + \sin (s)^2 - \frac{C}{2 \ell_y},
\end{align*}
$C$ being the constant from Lemma \ref{bound-commutator}. Therefore $H_y(s)$ is invertible for all $s \in [0,\pi/2]$ and $y \in Y$, provided that
\begin{equation}\label{keyestimate}
\ell_y>\frac{2 C(1+ \kappa(y)/4)}{\kappa(y)}
\end{equation}
for all $y \in Y$. This proves Theorem \ref{thm:additivity} under the additional assumption that the length $\ell_y$ of the straight cylinder in the middle is long enough to satisfy \eqref{keyestimate} for all $y \in Y$.

The key observation to treat the general case is now that stretching the cylinder $(a_0,a_1)$ to arbitrarily big length $\ell$ does \emph{not} affect the lower bound $\kappa$ for the scalar curvature (the metric is a product metric on the cylinder, and stretching the cylinder in the $\bR$-direction does not change the scalar curvature).
Here is how the stretching is done. We will simultaneously change the function to $\bR$ and the metric on the common piece $E_{(a_0,a_1)}= E'_{(a_0,a_1)}$. 
Let $b: X \times \bR \to [0,\infty)$ be a function (the restriction to $x \times \bR$ should be smooth) with support in $X \times (a_0,a_1)$. We change the metric $g=g'$ on $E_{(a_0,a_1)}= E'_{(a_0,a_1)}$ by adding $b(x,t)^2 dt^2$. The straight cylinder now has length $\int_{a_0(x)}^{a_1(x)} \sqrt{1+b^2(x,\tau)} d \tau$. We change the projection maps $E, E'\to \bR$ to ($ z\in E,E'$) to
\begin{equation*}
\tilde{t} (z):= t(z) + \int_{-\infty}^{t(z)} \sqrt{1+b^2(\pi(z),\tau)} d\tau.
\end{equation*}

Let $a'_1 (x):= a_0(x) + \int_{a_0(x)}^{a_1(x)} \sqrt{1+b^2(x,\tau)} d\tau$; with the new metric and the new projection functions, the two bundles are cylindrical over a $X \times (a_0,a'_1)$. Clearly, the new family with the stretched cylinder is concordant to the original one. If the function $b$ is picked such that for $y \in Y$, the inequality 
$$\int_{a_0(y)}^{a_1(y)} \sqrt{1+b^2(y,\tau)} d\tau \geq 3 \frac{C(1+ \kappa(y) /4)}{\kappa (y)}$$
holds, then the homotopy \eqref{additivity-proof-homotopy} is through invertible operators over $Y$.
\end{proof}

\subsubsection{A more useful formulation of the additivity theorem}
 
We reformulate the additivity theorem slightly, to a form better adapted to our needs. Let $(X,Y)$ be as above and let $W: M_0 \rightsquigarrow M_1$, $W': M_1 \rightsquigarrow M_2$ be $d$-dimensional spin cobordisms. Let $E \to X$ ($E' \to X$ resp.) be a $W$-bundle ($W'$-bundle, resp.) with trivialised boundary and a spin structure. Let $h_i \in \Riem^+ (M_i)$ be psc metrics and let $g$ ($g'$, resp.) be a fibrewise Riemannian metric on $E$ ($E'$, resp.) which coincides over the boundaries with $h_i$. Assume that over $Y$, the metrics $g$ and $g'$ have positive scalar curvature.

\begin{cor}\label{additivity-concordance-version}
Under the above assumptions, the following holds:
\[
\ind{E,g} + \ind{E',g'} = \ind{E \cup_{X \times M_1} E', g \cup g'} \in KO^{-d}(X,Y).
\]
\end{cor}

\begin{proof}
Let $E_0:= X \times (-\infty,0] \times M_0 \cup_{X \times M_0} E$, $E_1:= X \times [1, \infty) \times M_1$, $E_2 := X \times (-\infty,1] \times M_1$ and $E_3:= E' \cup_{X \times M_2} [2, \infty) \times M_2$, with the metrics extended by cylindrical metrics. By Theorem \ref{thm:additivity}
\[
 \ind{E_{01}} + \ind{E_{23}} = \ind{E_{03}}+\ind{E_{21}}=0 \in KO^{-d}(X,Y).
\]
The manifold bundle $E_{21}$ has positive scalar curvature and thus $\ind{E_{21}}=0$. Tracing through the definitions shows that $ \ind{E_{01}} = \ind{E,g}$, $\ind{E_{23}} = \ind{E',g'}$ and $\ind{E_{03}}= \ind{E \cup_{X \times M_1} E',g \cup g'}$. This completes the proof.
\end{proof}

\subsubsection{Additivity property of the index difference}

Here we show what the additivity theorem yields for the index difference.

\begin{thm}\label{additivity-index-difference}
Let $M_0 \stackrel{V}{\rightsquigarrow}M_1 \stackrel{W}{\rightsquigarrow} M_2$ be spin cobordisms, $h_i \in \Riem^+ (M_i)$ be boundary conditions, and $g \in \Riem^+ (V)_{h_0,h_1}$ and $m \in \Riem^+ (W)_{h_1,h_2}$ be psc metrics satisfying these boundary conditions. Then
\begin{equation*}
\xymatrix{
\Riem^+ (V)_{h_0,h_1} \times \Riem^+ (W)_{h_1,h_2}\ar[rr]^-{\mu} \ar[d]_{\inddiff_{g}\times \inddiff_m } & & \Riem^+ (V \cup W)_{h_0,h_2} \ar[d]^{\inddiff_{g \cup m}} \\
\loopinf{+d+1} KO \times \loopinf{+d+1} KO \ar[rr]^-{+} & & \loopinf{+d+1} \KO
}
\end{equation*}
is homotopy commutative.
\end{thm}

\begin{proof}
Write $X := \Riem^+ (V)_{h_0,h_1}\times \Riem^+ (W)_{h_1,h_2}$. On the space $X \times \bI$, we consider the following trivial fibre bundles equipped with Riemannian metrics:
\begin{enumerate}[(i)]
\item $E_{01}$ has fibre $(-\infty,0] \times M_0 \cup V \cup [1,\infty) \times M_1$; the metric over $(x,y,t)$ is $\frac{1-t}{2}g + \frac{1+t}{2}x$, extended to the cylinders by a product.
\item $E_{23}$ has fibre $(-\infty,1] \times M_1 \cup W \cup [2, \infty) \times M_2$; the metric over $(x,y,t)$ is $\frac{1-t}{2}m+\frac{1+t}{2} y$, extended to the cylinders by a product.
\end{enumerate}

These two bundles are cylindrical over a neighborhood of $X \times \bI \times \{1\}$ and agree there. As in Section \ref{sec:additivity}, these bundles are partitioned along $M_1$ into $E_i$, $i=0,\ldots,3$. By Theorem \ref{thm:additivity}, we get 
\[
 \ind{E_{01}} + \ind {E_{13}} = \ind{E_{03}} + \ind{E_{21}}.
\]
But $\ind{E_{21}}=0$ because the bundles $E_1$ and $E_2$ have positive scalar curvature. By construction $\ind{E_{01}}=\inddiff_g$, $\ind{E_{23}}=\inddiff_m$ and $\ind{E_{03}}=\inddiff_{g \cup m} \circ \mu$. 
\end{proof}

\begin{rem}
Let $W^d$ be closed and $\phi: S^{k-1} \times D^{d-k+1} \to W$ be an embedding, with $3 \leq k \leq d-2$. We use the notation of the proof of Theorem \ref{surgery-invariance}. There are maps 
\begin{align*}
\mu_{ g_{\round}^{k-1}+ g_{\tor}^{d-k+1}}:\Riem^+ (W^{\circ})_{g_{\round}^{k-1}+ g_{\round}^{d-k}}  &\lra  \Riem^+ (W)\\
\mu_{g_{\tor}^{k}+ g_{\round}^{d-k}}: \Riem^+ (W^{\circ})_{g_{\round}^{k-1}+ g_{\round}^{d-k}}  &\lra  \Riem^+ (W'),
\end{align*}
and the surgery equivalence map $\SE_\phi$ is the composition of a homotopy inverse of the first with the second map. Hence, the index difference satisfies an appropriate cobordism invariance. We leave the precise formulation to the interested reader, as we will not use this fact.
\end{rem}

\subsubsection{Propagating a detection theorem}
We spell out a consequence of Theorem \ref{additivity-index-difference} which is important for the global structure of this paper. If one manages to prove a detection theorem for the index difference for a certain single spin cobordism $W^d$ (satisfying the conditions listed below), then the same detection theorem follows for \emph{all} spin manifolds of dimension $d$.

\begin{prop}\label{inheritance-of-detection}
Let $W : \emptyset \leadsto S^{d-1}$ be a simply-connected spin cobordism of dimension $d \geq 6$, which is spin cobordant to $D^d$ relative to its boundary. Let $g \in \cR^+(W)_{h_\round^{d-1}}$ be in a path component which corresponds to that of $g_\tor^d \in \cR^+(D^d)_{h_\round^{d-1}}$ under a surgery equivalence $\cR^+(W)_{h_\round^{d-1}} \simeq \cR^+(D^d)_{h_\round^{d-1}}$.

Let $W'$ be an arbitrary $d$-dimensional spin cobordism with boundary $M'$ and let $h' \in \Riem^+ (M')$ and $g' \in \Riem^+ (W')_{h'}$.
Let $X$ be a CW complex, $\hat{a}: X \to \loopinf{+d+1} \KO$ be a map and let a factorisation
$$X \stackrel{\rho}{\lra} \Riem^+ (W)_{h_\round^{d-1}} \stackrel{\inddiff_{g}}{\lra} \loopinf{+d+1} \KO$$
of $\hat{a}$ up to homotopy be given. Then there exists another factorisation 
$$X \stackrel{\rho'}{\lra} \Riem^+ (W')_{h'} \stackrel{\inddiff_{g'}}{\lra} \loopinf{+d+1} \KO$$
of $\hat{a}$ up to homotopy.
\end{prop}

\begin{proof}
Let $W_0 = W \setminus \mathrm{int}(D^d)$ be $W$ with an open disc in the interior removed. As $W$ is simply-connected and spin cobordant to $D^d$ relative to its boundary, it follows that $W_0$ is simply-connected and spin cobordant to $[0,1] \times S^{d-1}$ relative to its boundary. Thus by Theorem \ref{lemma:GluingNullCob} there is a metric $g_0 \in \cR^+(W_0)_{h_\round^{d-1}, h_\round^{d-1}}$ such that the map
$$\mu_{g_0} : \cR^+(D^d)_{h_\round^{d-1}} \lra \cR^+(D^d \cup W_0)_{h_\round^{d-1}}$$
is a weak homotopy equivalence, and $g_\tor^d \cup g_0$ is isotopic to $g$. As $X$ is a CW-complex, there is a map $\rho' : X \to \Riem^+ (D^d )_{h_{\round}^{d-1}}$ so that $\mu_{g_0} \circ \rho' \simeq \rho$, and using Theorem \ref{additivity-index-difference} the composition
$$X \stackrel{\rho'}{\lra} \Riem^+ (D^d)_{h_{\round}^{d-1}} \stackrel{\inddiff_{g_{\tor}^d}}{\lra} \loopinf{+d+1} \KO$$
is homotopic to $\hat{a}$.

Let $W'$ be a spin cobordism with a psc metric $g'$ as in the assumption of the proposition. We can change $g'$ by an isotopy supported in a small disc in the interior so that $g'$ agrees with $g_{\tor}^d$ in that disc, by an easy application of Theorem \ref{chernysh-walsh-theorem}. This does not change the homotopy class of the index difference, and we can write $W'=D^d \cup W'_0 $ and $g'= g_{\tor}^d \cup g_1$. The desired factorisation is 
$$X \stackrel{\rho'}{\lra} \Riem^+ (D^d)_{h_{\round}^{d-1}} \stackrel{\mu_{g_1}}{\lra} \Riem^+ (W')_{h'} \stackrel{\inddiff_{g'}}{\lra} \loopinf{+d+1} \KO$$
and the composition is homotopic to $\hat{a}$, again by Theorem \ref{additivity-index-difference}.
\end{proof}

\subsection{The relative index construction in an abstract setting}\label{sec:abstract-nonsense}

For the proof of Theorems \ref{factorization-theorem:even-case} and \ref{factorization-theorem:odd-case}, we need a precise tool to express secondary indices in terms of primary indices. Our tool is the \emph{relative index construction}.

Write $I=[0,1]$. Let $f:(X,x_0) \to (Y,y_0)$ be a map of pointed spaces (in practice, $f$ will be a fibration of some kind).
Recall the notions of \emph{mapping cylinder} $\Cyl(f):=( X \times I) \coprod Y/\sim$, $(x,1)\sim f(x)$ and \emph{homotopy fibre} $\hofib_y (f)=\{ (x,c) \in X \times Y^I \vert f(x)=c(0), c(1)=y\}$. There is a natural map $\epsilon_{y_0}:f^{-1} (y_0) \to \hofib_{y_0} (f)$, $x \mapsto (x,c_{y_0})$, where $c_{y_0}$ denotes the constant path at $y_0$. Let $(x_0,c_{y_0}) \in \hofib_{y_0}(f)$ and $x_0 \in f^{-1}(y_0)$ be the basepoints, so that $\epsilon_{y_0}$ is a pointed map.
There is a natural map 
$$\eta_{y_0}:(\bI, \{\pm 1\}) \times (\hofib_{y_0} (f),*)  \lra (\Cyl (f), X \times \{0\}\cup (\{x_0\} \times I))$$
defined by
$$(t,x,c) \longmapsto \begin{cases}
(x,1+t) & t \leq 0 \\
c(t) & t \geq 0
\end{cases}$$
and the composition $\eta_{y_0} \circ (\id_{\bI} \times \epsilon_{y_0})$ is homotopic to the map $\iota_{y_0}:\bI \times f^{-1}(y_0) \to \Cyl (f)$; $(t,x)\mapsto (x, \frac{1+t}{2})$ (as maps of pairs).
Moreover, there is the \emph{fibre transport map} from the loop space of $Y$ based at $y_0$ to the homotopy fibre (it is a based map):
\begin{align*}
\tau: \Omega_{y_0} Y &\lra \hofib_{y_0} (f)\\
c &\longmapsto (x_0,c).
\end{align*}
We write $KO^{-p}(f):= KO^{-p}(\Cyl (f), X  \times \{0\}\cup (  \{x_0\} \times I) )$. There is an induced map
\begin{align*}
\trg:KO^{-p}(f) &\lra \Omega KO^{-p} (\hofib_{y_0} (f),*)\\
\alpha &\longmapsto \eta_{y_0}^{*} \alpha,
\end{align*}
called the \emph{transgression}. (Often, we consider the composition $\epsilon_{y_0}^{*} \circ \trg$ and call this composition transgression as well.) The inclusion $i: Y \to \Cyl(f)$ induces 
\begin{align*}
\bas: KO^{-p}(f) &\lra KO^{-p}(Y,y_0)\\
\alpha &\longmapsto i^* \alpha.
\end{align*}
We call $\bas(\alpha)$ the \emph{base class} of the relative class $\alpha$. Finally, there is a map, the \emph{loop map} $\Omega:KO^{-p} (Y,y_0) \to \Omega KO^{-p} (\Omega_{y_0} Y,*)$, given by pulling back by the evaluation map $\bI \times \Omega_{y_0} Y \to Y$. 

\begin{lem}\label{relativeindex:formal-lemma}
Let $\alpha \in KO^{-p} (f)$ be given. Then 
$$\tau^* \trg (\alpha) = \Omega \bas (\alpha) \in \Omega KO^{-p} (\Omega_{y_0} Y,*).$$
\end{lem}

\begin{proof}
The proof is easier to understand than the statement. The diagram 
\begin{equation*}
\xymatrix{
(\bI, \{\pm 1 \}) \times (\Omega_{y_0} Y,*) \ar[r]^-{ev} \ar[d]^{\tau} & (Y,y_0) \ar[d]^{i}\\
(\bI, \{\pm 1 \}) \times (\hofib_{y_0} (f) ,*) \ar[r]^-{\eta_{y_0}} & (\Cyl(f),X \cup (\{x_0\} \times I))
}
\end{equation*}
is homotopy commutative (in space pairs), and hence the associated diagram in $K$-theory is commutative. By tracing through the definitions, this commutativity is expressed by the formula 
in the statement of the lemma.
\end{proof}

This lemma has the following homotopy-theoretic interpretation. Recall that we write relative $K$-theory classes as homotopy classes of maps, via the bijection in Theorem \ref{atsingkar}.

\begin{cor}\label{relativeindex:formal-lemma-homotopy-version}
Let $\alpha \in KO^{-p} (f)$ be a relative $K$-theory class. Then the diagram
\begin{equation*}
\xymatrix{
\Omega_{y_0} Y \ar[d]^{\tau} \ar[dr]^{\Omega \bas(\alpha)} &  \\
\hofib_{y_0} (f) \ar[r]_-{\trg (\alpha)}  & \loopinf{+p-q+1} \KO
}
\end{equation*}
is homotopy commutative.
\end{cor}

We will use this to translate index-theoretic results into homotopy-theoretic conclusions.

\begin{rem}\label{rem:relative-index-hmtpythy}
An instructive way to view these constructions is as follows. First note that $f^* \bas (\alpha)\in KO^{-p}(X,x_0)$ is the zero class (there is a canonical concordance to an acyclic cycle). 
One might say that $\alpha \in KO^{-p}(f)$ induces a homotopy commutative diagram
\[
\xymatrix{
\hofib_{y_0} (f) \ar[r]^{\trg(\alpha)} \ar[d] & \loopinf{+p+1} \KO \ar[d] \\
X \ar[r] \ar[d]^{f} & \ast \ar[d]\\
Y \ar[r]^{\bas(\alpha)} & \loopinf{+p} \KO
}
\]
whose columns are homotopy fibre sequences.
\end{rem}

\subsection{Increasing the dimension}\label{sec:increasing-dimension}

The first application of the generalities from Section \ref{sec:abstract-nonsense} is the derivation of Theorem \ref{factorization-theorem:odd-case} from Theorem \ref{factorization-theorem:even-case} (the true ingredients are Theorems \ref{spectralflowindex}, \ref{chernysh-thm:qfibr} and the additivity theorem).

\begin{thm}\label{dimension-increasing}
Let $W^d$ be a compact spin manifold with boundary $M$. Let $h_0\in \Riem^+ (M)$ and $g_0 \in \Riem^+ (W)_{h_0}$. Under the equivalence $\Riem^+ (W)_{h_0} \overset{\sim}\to \hofib_{h_0}(\res)$ of Theorem \ref{chernysh-thm:qfibr}, fibre transport gives a weak homotopy class of map $T:\Omega_{h_0} \Riem^+ (M) \to \Riem^+ (W)_{h_0}$, which makes the diagram
\begin{equation*}
\xymatrix{
\Omega_{h_0}\Riem^+ (M) \ar[rr]^{T}  \ar[dr]_{-\Omega \inddiff_{h_0}} &   & \Riem^+ (W)_{h_0}  \ar[dl]^{\inddiff_{g_0}} \\
 &  \loopinf{+d+1}\KO.
}
\end{equation*}
weakly homotopy commutative.
\end{thm}

Before we embark on the proof of Theorem \ref{dimension-increasing}, we show how it accomplishes the promised goal.

\begin{proof}[Proof of Theorem \ref{factorization-theorem:odd-case}, assuming Theorems \ref{factorization-theorem:even-case} and \ref{dimension-increasing}]
By Proposition \ref{inheritance-of-detection}, it is enough to consider the case $W= D^{2n+1}$ and $g_0=g_{\tor}^{2n+1}$.
Then $h_0$ is the round  metric in $\Riem^+ (S^{2n})$. Consider the diagram
\begin{equation*}
\xymatrix{
\loopinf{+2} \MTSpin (2n) \ar[r]^-{-\Omega \rho} & \Omega_{h_0} \Riem^+ (S^{2n}) \ar[r]^{T} \ar[dr]_{-\Omega \inddiff_{h_0}} &  \Riem^+ (D^{2n+1})_{h_0} \ar[d]^{\inddiff_{g_0}} \\
 & & \loopinf{+2n+2} \KO.
}
\end{equation*}
The map $\Omega \rho$ is provided by Theorem \ref{factorization-theorem:even-case}, the composition $(-\Omega \inddiff_{h_0}) \circ (-\Omega \rho) \simeq \Omega (\inddiff_{h_0} \circ \rho)$ is weakly homotopic to $\Omega^2 \hat{\mathscr{A}}_{2n}$, also by Theorem \ref{factorization-theorem:even-case}. The triangle is weakly homotopy commutative by Theorem \ref{dimension-increasing}, and $T \circ (-\Omega \rho)$ is the map whose existence is claimed by Theorem \ref{factorization-theorem:odd-case}.
\end{proof}

The rest of this section is devoted to the proof of Theorem \ref{dimension-increasing}. 
Consider the restriction map $\res:\Riem^+ (W) \to \Riem^+ (M)$. We will construct and analyse a relative class $\beta \in KO^{-d}(\res)$. 

Metrics of positive scalar curvature on $M$ can be extended to (arbitrary) metrics on $W$, by the following procedure. Pick a cut-off function $a: W \to [0,1]$ which is supported in the collar around $M$ where $g_0$ is of product form and is equal to $1$ near the boundary. For any $h \in \Riem^+ (M)$, let $\sigma(h):= a(h+dt^2) + (1-a) g_0$. The result is an extension map $\sigma: \Riem^+ (M) \to \Riem (W)$ (to $\Riem (W)$, not to $\Riem^+ (W)$!) such that $
\sigma(h_0)=g_0$ and such that $\sigma(h)$ restricted to a collar is equal to $h + dt^2$.

Consider the trivial fibre bundle with fibre $W$  over the mapping cylinder $\Cyl (\res)$ of the restriction map, and define the following fibrewise Riemannian metric $m$ on it: over the point $h \in \Riem^+ (M) \subset \Cyl(\res)$ take the metric $\sigma(h)$, and over the point $(g,t) \in \Riem^+ (W)\times [0,1]$ take the metric $m_{(g,t)}:=t \sigma(\res(g)) + (1-t) g$. By the construction of $\sigma$,
\begin{enumerate}[(i)]
\item the metric $m$ is psc when restricted to the boundary bundle $\Cyl (\res) \times M$,
\item $m_{(g,0)}= g$ for $g \in \Riem^+ (W)$ and
\item $m_{(g_0,t)} = t \sigma (\res(g_0)) + (1-t) g_0 =g_0$, for $t \in [0,1]$.
\end{enumerate}
 
Therefore, $m$ has positive scalar curvature over $(\Riem^+ (W) \times \{0\}) \cup (\{g_0\} \times I) \subset \Cyl (\res)$. We define the class
$$\beta = \ind{W,m} \in KO^{-d} (\Cyl(\res), (\Riem^+ (W) \times \{0\}) \cup (\{g_0\} \times I)).$$
The choices made in the construction of $\beta$ are the metric $g_0$ and the cut-off function $a$. The function $a$ is a convex choice, and so $\beta$ depends  only on $g_0$.

\begin{prop}\label{dimension-increasing-lemma}\mbox{}
\begin{enumerate}[(i)]
\item\label{it:dimension-increasing-lemma:1} Let $\epsilon_{h_0} : \Riem^+ (W)_{h_0} \to \hofib_{h_0} (\res)$ be the fibre comparison map. Then the map $\epsilon_{h_0}^{*} (\trg(\beta)): \Riem^+ (W)_{h_0} \to \loopinf{+d+1} \KO$ is homotopic to $- \inddiff_{g_0}$.
\item\label{it:dimension-increasing-lemma:2} $\bas (\beta) : \Riem^+ (M) \to \loopinf{+d} \KO$ is weakly homotopic to $\inddiff_{h_0}$.
\end{enumerate}
\end{prop}

\begin{proof}
The first part is straightforward. The class $\epsilon_{h_0}^{*} (\trg(\beta))$ is represented by the cycle $\iota^*_{h_0} \beta$, and $\iota_{h_0}: \bI \times \Riem^+ (W)_{h_0} \to \Cyl(\res)$ is the map $(t,g) \mapsto (g,\frac{1+t}{2})$. So $\iota^*_{h_0} \beta$ is represented by the cycle $(t,g) \mapsto \indx{W,m_{g,\frac{1+t}{2}}}$, but
$$m_{g,\frac{1+t}{2}}= \frac{1+t}{2}\sigma (\res(g)) + \frac{1-t}{2} g = \frac{1+t}{2}\sigma (h_0) + \frac{1-t}{2} g= \frac{1+t}{2} g_0 + \frac{1-t}{2} g,$$
because $g \in \Riem^+ (W)_{h_0}$ and $\sigma(h_0)=g_0$, and this represents minus the index difference.

The second part is deeper and uses the additivity theorem and the spectral flow theorem. The base class $\bas(\beta)$ lies in $KO^{-d}(\Riem^+ (M), h_0)$, and by tracing through the definitions, we find that 
$$\bas(\beta) = \ind{W,\sigma(h)},$$
the notation indicates that $\bas(\beta)$ is represented by the cycle whose value at $h \in \Riem^+ (M)$ is the Dirac operator of the metric $\sigma(h)$ on the manifold $W$.
For two metrics $h_0,h_1 \in \Riem^+ (M)$, denote by $[h_0,h_1]$ the metric $ds^2 + \varphi(s) h_1 + (1-\varphi(s))h_0$ on $M \times [0,1]$, where $\varphi:[0,1] \to [0,1]$ is a smooth function that is $0$ near $0$ and $1$ near $1$. Now we calculate, using the additivity theorem (Corollary \ref{additivity-concordance-version}), that
\begin{equation*}
\begin{aligned}
\ind{W,\sigma(h)} + \ind{M \times [0,1], [h,h_0]} + \ind{W^{op}, g_0}\\
 = \ind{W \cup M \times [0,1] \cup W^{op}, \sigma(h) \cup [h,h_0] \cup g_0}=0\in KO^{-d}(\Riem^+ (M), h_0).
\end{aligned}
\end{equation*}

To see that the index is zero, note first that the manifold $W \cup (M \times [0,1]) \cup W^{op}$ is closed and therefore the index on the right hand side does not depend on the choice of the metric, by Lemma \ref{index:independent-of-metric}. Moreover $g_0 \cup [h_0,h_0] \cup g_0$ is a psc metric on $W \cup (M \times [0,1]) \cup W^{op}$ which agrees with $\sigma(h) \cup [h,h_0]\cup g_0$ over the basepoint $h_0 \in \Riem^+ (M)$, and therefore the right-hand side is zero. On the other hand, the contribution $\ind{W^{op}, g_0}$ on the left hand side is zero because $g_0$ has positive scalar curvature, by Bochner's vanishing argument with the Lichnerowicz--Schr\"odinger formula. Therefore
\begin{equation*}
\ind{W,\sigma(h)} = - \ind{M \times [0,1], [h,h_0]} = \ind{M \times [0,1], [h_0,h]}.
\end{equation*}
The last equality is true by the addivitity theorem and by Lemma \ref{index:independent-of-metric}.
But, by Theorem \ref{spectralflowindex}, the right hand side is equal to $\inddiff_{h_0} \in  KO^{-d-1} (\Riem^+ (M))$. 
\end{proof}

\begin{proof}[Proof of Theorem \ref{dimension-increasing}]
Consider the diagram:
\begin{equation}\label{dimension-increasing:index-theory}
\xymatrix{
\Riem^+ (W)_{h_0} \ar[r]^-{\epsilon_{h_0} } \ar[dr]_-{-\inddiff_{g_0}} & \hofib_{h_0}(\res)\ar[d]^-{\trg(\beta)} & \Omega_{h_0}\Riem^+ (M) \ar[l]_-{\tau}  \ar[dl]^-{\Omega \inddiff_{h_0}}\\
 &  \loopinf{+d+1}\KO. &
}
\end{equation}
The left triangle is homotopy commutative by Proposition \ref{dimension-increasing-lemma} (\ref{it:dimension-increasing-lemma:1}), and the right triangle is homotopy commutative by Proposition \ref{dimension-increasing-lemma} (\ref{it:dimension-increasing-lemma:2}) and Corollary \ref{relativeindex:formal-lemma-homotopy-version}. 
The final ingredient is the quasifibration theorem (Theorem \ref{chernysh-thm:qfibr}): to get the map $T$, take the homotopy inverse of $\epsilon_{h_0}$ provided by Theorem \ref{chernysh-thm:qfibr} and compose it with $\tau$.
\end{proof}

\begin{rem}
The homotopy inverse to $\epsilon_{h_0}$ is given by an explicit construction, which was sketched before Theorem \ref{chernysh-thm:qfibr}.
\end{rem}

As mentioned in Remark \ref{rem:NoExtension}, we can now explain why the index difference map cannot be extended from $\Riem^+ (W)_{h_0}$ to $\Riem^+ (W)$.

\begin{cor}\label{cor:inddiff-doesnt-extend}
If $\dim (W) \geq 7$, $\partial W \neq \emptyset$ and $\Riem^+ (W)_{h_0} \neq \emptyset$, then there does not exist a map $\Riem^+ (W) \to \loopinf{+d+1} \KO$ which extends $\inddiff_{g_0}: \Riem^+ (W)_{h_0} \to \loopinf{+d+1} \KO$.
\end{cor}

\begin{proof}
Fix $g_0 \in \Riem^+ (W)_{h_0}$ and consider the class $\beta$ constructed in Proposition \ref{dimension-increasing-lemma}. By that Proposition and Remark \ref{rem:relative-index-hmtpythy}, we obtain a homotopy commutative diagram whose columns are fibration sequences
\begin{equation*}
\xymatrix{
{\Riem^+ (W)_{h_0}} \ar[rr]^{-\inddiff_{g_0}} \ar[d] && {\loopinf{+d+1} \KO} \ar[d] \\
{\Riem^+ (W)} \ar[rr] \ar[d] && {*} \ar[d]\\
{\Riem^+ (M)} \ar[rr]^{\inddiff_{h_0}} && {\loopinf{+d} \KO}.
}
\end{equation*}
A hypothetical extension of $\inddiff_{g_0}$ to $\Riem^+ (W)$ would force the map $\inddiff_{h_0}$ to be zero on homotopy groups in positive degrees. This follows from the long exact homotopy sequence and contradicts Theorem \ref{thm:intro:htpyGps}.
\end{proof}

\subsection{The Atiyah--Singer index theorem}\label{sec:atiyah-singer}

In the case when the spin manifold bundle $\pi:E \to X$ has closed $d$-dimensional fibres, the index $\ind{E}\in KO^{-d}(X)$ can be expressed in homotopy-theoretic terms using the Clifford-linear version of the Atiyah--Singer family index theorem. We recall the result here. 

\subsubsection{$KR$-theory and the Atiyah--Bott--Shapiro construction}
The formulation of the index theorem which we shall use is in terms of Atiyah's $KR$-theory \cite{AtKR}. Let $Y$ be a locally compact space with an involution $\tau$ (a ``Real space" in the terminology of \cite{AtKR}). 
Classes in the \emph{compactly supported} Real $K$-theory $KR_c (Y)$ are represented by triples $(E, \iota,f)$, with $E \to Y$ a finite-dimensional ``Real vector bundle'' as defined in \cite[\S 1]{AtKR}, $\iota$ a $\bZ/2$-grading on $E$ and $f: E \to E$ an odd, self-adjoint bundle map which is an isomorphism outside a compact set. Both $\iota$ and $f$ are required to be Real homomorphisms as defined in \cite[\S 1]{AtKR}. Out of the set of concordance classes of triples $(E,\iota,f)$, one defines $KR_c (Y)$ by a procedure completely analogous to that of Definition \ref{defn-fpq}. 
Recall the Atiyah--Bott--Shapiro construction \cite{ABS}: let $X$ be a compact space, $V,W \to X$ be Riemannian vector bundles and $p: V \oplus W \to X$ be their sum. Let $V^+ \oplus W^-$ be the total space of $V \oplus W$, equipped with the involution $(x,y) \mapsto (x,-y)$. Let $E $ be a real $\Cl(V^+ \oplus W^-)$-module. Then the triple $(p^* E \otimes \bC, \iota, \gamma)$ with $\gamma_{v,w} := \iota (c(v)+i c(w))$ represents a class $\abs(E) \in KR_c(V^+ \oplus W^-)$, the \emph{Atiyah--Bott--Shapiro class} of $E$.

If $V \to X$ is a rank $d$ spin vector bundle with spin structure $\spinor_V$, then the $KO$-theory Thom class of $V$ is defined to be
$$\lambda_V := \abs (\spinor_V) \in KR_c (V^+ \oplus \bR^{0,d}) \stackrel{}{\cong} KO_c^{d} (V)\stackrel{}{\cong} \widetilde{KO}^d (\mathrm{Th} (V)),$$
where the first isomorphism is proved in \cite[\S 2-3]{AtKR} (and the second one is a standard property of compactly supported $K$-theory).

\subsubsection{The Clifford-linear index theorem}

Let $\pi:E \to X$ be a spin manifold bundle with closed $d$-dimensional fibres over a compact base space. Assume that $E \subset X \times \bR^n$ and that $\pi$ is the projection to the first factor. Let $N_v E$ be the vertical normal bundle, so that $T_v E \oplus N_v E = E \times \bR^n$ and let $U \subset X \times \bR^n$ be a tubular neighborhood of $E$. The spin structure on $T_v E$ induces a spin structure on $N_v E$. The open inclusion $N_v E \cong U \subset X \times \bR^n$ induces a map (via a Pontrjagin--Thom type collapse construction)
\begin{equation}\label{gysin-map-k-theory}
\psi:KR_c (N_v E^+\oplus \bR^{0,n-d}) \lra KR_c (X \times \bR^{n,n-d}) \cong KO^{-d} (X), 
\end{equation}
where the last map is the $(1,1)$-periodicity isomorphism \cite[Theorem 2.3]{AtKR}. The Clifford-linear index theorem is

\begin{thm}\label{cl-index-1}
With this notation, $\ind{E} = \psi (\lambda_{N_v E}) \in KO^{-d}(X)$.
\end{thm}

This result is due to Hitchin \cite[Proposition 4.2]{HitchinSpin} and is of course a consequence of the Atiyah--Singer index theorem for families of real elliptic operators \cite{AtSiV}. However, Hitchin's explanation of the argument (as well as the treatment in \cite[\S 16]{SpinGeometry}) leaves out a critical detail. For sake of completeness, we discuss this detail.

Recall that $\spinor_{T_v E}$ has an action $c$ of the Clifford algebra $\Cl^{d,0}$. The Dirac operator $\Dir_x$ on $E_x$ anticommutes with the self-adjoint bundle endomorphism $\iota c(v)$, for each $(x,v) \in X \times \bR^{d,0}$. Using the isomorphism \eqref{eqn:bott-isomorphism}, the family of Dirac operators defines a family 
$$R_{x,v} := \normalize{\Dir_x} + \iota c(v)$$
over $X \times \bR^d$ which is invertible unless $v=0$ and so gives an element in the group $KO^0 (X \times D^d , X \times S^{d-1})$. Hitchin uses the real family index theorem \cite{AtSiV} to compute the index of the family $(R_{x,v})_{(x,v) \in X \times \bR^{d,0}}$. However, the family index theorem as proven in \cite{AtSiV} only compares elements in absolute $K$-theory, not in relative $K$-theory, and it is not completely evident how the topological index of the family $(R_{x,v})$ (as an element in relative $K$-theory) is defined. Therefore, a slight extension of the real family index theorem is necessary.

\begin{thm}\label{relative-family-index}
Let $\pi: E \subset X \times \bR^n \to X$ be a closed manifold bundle over a compact base. Let $Q: \Gamma (E;V_0)\to \Gamma (E;V_1)$ be a family of real elliptic pseudo-differential operators of order $0$. Assume that $Y \subset X$ is a closed subspace such that for $y \in Y$, the operator $Q_y$ is a bundle isomorphism. The symbol class $\smb (Q)$ is then an element in $KR_c (T_v (E_{X-Y})^-)$ and the family index $\ind{Q} \in KO (X,Y) = KO_c(X-Y)$ can be computed as the image of $\smb (Q)$ under the composition
\begin{equation*}
\begin{split}
KR_c (T_v (E_{X-Y})^-) \lra KR_c (T_v (E_{X-Y})^- \oplus N_v (E_{X-Y})^- \oplus N_v (E_{X-Y})^+) = \\
KO_c^n (N_v E_{X-Y}) \lra KO_c^n ((X-Y) \times \bR^n) \cong KO_c (X-Y)
\end{split}
\end{equation*}
of the Thom isomorphism, the map induced by the inclusion and Bott periodicity.
\end{thm}

\begin{proof}
For $Y=\emptyset$, this is precisely the Atiyah--Singer family index theorem \cite{AtSiV}. In the general case, form $Z= X \cup_Y X$, $E'= E \cup_{E|_Y} E$. Since $Q$ is a bundle isomorphism over $Y$, we can form the clutching $V'_0 = V_0 \cup_Q V_1$ and $V'_1 = V_1 \cup_{\id} V_1$ over $E'$ and extend the family $Q$ by the identity over the second copy $X_2$ of $X$. We obtain a new family $P$, which coincides with $Q$ over the first copy $X_1$ of $X$ and is a bundle isomorphism over $X_2$. Therefore, under the map $KO(X,Y) \cong KO(Z, X_2) \to KO(Z)$, the topological (resp. analytical) index of $Q$ is mapped to the topological (resp. analytical) index of $P$. By the index theorem, the topological and analytical indices of $P$ agree. Finally, as the inclusion $X_2 \to Z$ splits, the map 
$KO(Z, X_2) \to KO(Z)$ is injective, and the indices agree in $KO(X,Y) =KO(Z, X_2)$.
\end{proof}

\begin{proof}[Proof of Theorem \ref{cl-index-1}]
We have to compute the family index of the family $R_{x,v} := \normalize{\Dir_x} + \iota v$ over $X \times D^d$, as an element in $KO(X \times (D^d ,S^{d-1}))$. 
Consider the homotopy ($t \in [0,1]$)
$$R_{t,x,v} := \sqrt{1-t^2 |v|^2}\normalize{\Dir_x} + \iota v$$
of families of order $0$ pseudo-differential operators. Because 
$$R_{t,x,v}^{2} = (1-t^2 |v|^2) \frac{\Dir_x^2}{1+\Dir_x^2} + |v|^2,$$
we find that $R_{t,x,v}$ is invertible for $|v|=1$. If $\xi \in T_vE$ has norm $1$, then 
$$\smb_{R_{t,x,v}^{2}}(\xi) = (1-t^2 |v|^2) \frac{|\xi|^2}{1+|\xi|^2} + |v|^2$$
and this is invertible for all $x$, $v$, and $t$, so the family $R_{t,x,v}$ is elliptic. Thus we can replace the original family $R_{x,v}=R_{0,x,v}$ by $R_{1,x,v}$, which over $X \times S^{d-1}$ is just the family $\iota v$ of bundle automorphisms. Thus the relative family index theorem (Theorem \ref{relative-family-index}) applies. The computation of the topological index of this family is done in the proof of \cite[Proposition 4.2]{HitchinSpin}.
\end{proof}

\subsection{Fibre bundles, Madsen--Tillmann--Weiss spectra, and index theory}\label{sec:FibBunIndDiff}

We can reformulate the Atiyah--Singer index theorem in terms of the Madsen--Tillmann--Weiss spectrum $\MTSpin (d)$. The basic reference for these spectra is \cite{GMTW} and the connection to the index theorem was pointed out in \cite{Eb09}. We refer to those papers for more detail.
Later, we will present the variation for manifolds with boundary, and finally relate the index difference to the family index. 

\subsubsection{Madsen--Tillmann--Weiss spectra}
Let $\gamma_{d} \to B \Spin (d)$ be the universal spin vector bundle. By definition, the Madsen--Tillmann-Weiss spectrum $\MTSpin (d)$ is the Thom spectrum of the virtual vector bundle $-\gamma_{d}$, which may be described concretely as follows. 
\begin{defn}\label{defn:spingrassmannian}
For $n \geq d$, we define the spin Grassmannian $\Gr_{d,n}^{\Spin}$ as the homotopy fibre of the natural map   
$$ B\Spin (d) \times B\Spin (n-d) \lra B\Spin (n).$$
\end{defn}
The spin Grassmannian comes with a map $\theta: \Gr_{d,n}^{\Spin} \to \Gr_{d,n} = \frac{O(n)}{O(d) \times O(n-d)}$ to the ordinary Grassmannian, and we let 
$$V_{d,n} \subset \Gr^\Spin_{d,n} \times \bR^n \text{ and } V_{d,n}^{\bot} \subset \Gr^\Spin_{d,n} \times \bR^n$$ 
be the pullback of the tautological $d$-dimensional vector bundle on $\Gr_{d,n}$ and its orthogonal complement.

By stabilising with respect to $n$, one obtains structure maps 
$$\sigma_n : \Sigma\mathrm{Th}(V_{d,n}^{\bot}) \lra  \mathrm{Th}(V_{d,n+1}^{\bot}),$$
and by definition the sequence of these spaces form the spectrum $\MTSpin (d)$. The ($n$ times looped) adjoints of the $\sigma_n$ yield maps
$$\Omega^{n} \mathrm{Th}(V_{d,n}^{\bot}) \lra \Omega^{n+1} \mathrm{Th}(V_{d,n+1}^{\bot}),$$
and by definition $\loopinf{} \MTSpin(d)$ is the colimit over these maps. Similarly the space $\Omega^{\infty+l} \MTSpin(d)$ is the colimit of the $\Omega^{n+l} \mathrm{Th}(V_{d,n}^{\bot})$, for $l \in\bZ$. There are maps $V_{d-1,n-1}^{\bot} \to V_{d,n}^{\bot}$ which are compatible with the structure maps of the spectra and give a map of spectra $\MTSpin (d-1) \to \Sigma \MTSpin (d)$. On infinite loop spaces this induces a map $\loopinf{} \MTSpin (d-1) \to \loopinf{-1}\MTSpin (d)$.

\subsubsection{Spin fibre bundles and the Pontrjagin--Thom construction}
A bundle $\pi: E \subset X \times \bR^n \to X$ of $d$-dimensional closed manifolds with a fibrewise spin structure has a Pontrjagin--Thom map
$$\alpha_E: X \lra \Omega^n \mathrm{Th} (V_{d,n}^{\bot}) \lra \loopinf{} \MTSpin (d)$$
whose homotopy class does not depend on the embedding of $E$ into $X \times \bR^n$. 
We also need a version for manifolds with boundary. Let $W^d$ be a manifold with boundary $M$, and consider fibre bundles $E \to X$ with structure group given by the diffeomorphisms that fix the boundary pointwise. The boundary bundle is then trivialised, $\partial E \cong X \times M$. Assume that there is a (topological) spin structure on the vertical tangent bundle $T_v E$, which is constant on the boundary bundle. 
Under favourable circumstances (which always hold for the manifolds to be considered in this paper), the following lemma shows that such spin structures exist.

\begin{lem}\label{lem:existence-spin-structure}
Let $W$ be a manifold with boundary $M$ such that $(W, M)$ is 1-connected, and let $\pi: E \to X$ be a smooth fibre bundle with fibre $W$ and trivialised boundary bundle, $\partial E \cong X \times M$. For each spin structure on $W$ there is a unique spin structure on $T_v E$ which is isomorphic to the given one on the fibre and which is constant over $\partial E$.
\end{lem}
\begin{proof}
The Leray--Serre spectral sequence for the fibration pair $(E, \partial E) \to X$ proves that $(E,\partial E)$ is homologically $1$-connected and that the fibre inclusion $(W,M) \to (E, \partial E)$ induces an injection $H^2 (E,\partial E; \bZ/2) \to H^2 (W,M; \bZ/2)$. The obstruction to extending the spin structure on $\partial E$ to all of $E$ lies in $H^2 (E, \partial E; \bZ/2)$ and goes to zero in $H^2 (W,M;\bZ/2)$ since $W$ is assumed to be spin and the spin structure on $M$ is assumed to extend over $W$. Thus the obstruction is trivial, which shows the existence of the spin structure. Uniqueness follows from $H^1 (E,\partial E; \bZ/2)=0$.
\end{proof}
A closed spin manifold $M^{d-1}$ determines a point $M \in \loopinf{-1} \MTSpin (d)$, namely the image under 
$$\loopinf{} \MTSpin (d-1) \lra  \loopinf{-1} \MTSpin (d)$$
of the point in $ \loopinf{} \MTSpin (d-1)$ determined by the Pontrjagin--Thom map of the trivial bundle $M \to *$. Of course, this point is not unique, but depends on an embedding of $M$ into $\bR^{\infty}$ and a tubular neighborhood, which is a contractible choice.

If $W$ is a $d$-dimensional manifold with boundary $M$, and $\pi : E \to X$ is a smooth manifold bundle with fibre $W$ equipped with a trivialisation $\partial E \cong X \times M$ of the boundary and a fibrewise spin structure which is constant along the boundary, then there is a Pontrjagin--Thom map
$$\alpha_E : X \lra \Omega_{[\emptyset, M]} \Omega^{\infty-1} \MTSpin(d)$$
to the space of paths in $\Omega^{\infty-1} \MTSpin(d)$ from the basepoint $\emptyset$ to $[M]$. A spin nullbordism $V:M\rightsquigarrow \emptyset$ determines a path in $\Omega^{\infty-1} \MTSpin(d)$ from $M$ to $\emptyset$ and thus a homotopy equivalence
$$\theta_V:\Omega_{[\emptyset, M]} \Omega^{\infty-1} \MTSpin(d) \overset{\sim}\lra  \Omega^{\infty} \MTSpin(d)$$
and we define $\alpha_{E,V}:=\theta_V \circ \alpha_E: X \to \loopinf{} \MTSpin (d)$. For two nullbordisms $V_0$ and $V_1$ of $M$ the maps $\alpha_{E,V_0}$ and $\alpha_{E,V_1}$ differ by loop addition with the constant map $\alpha_{V_0^{op}\cup V_1}$. 

\subsubsection{The index theorem and homotopy theory}

The vector bundles $V_{d,n}^\perp \to \Gr^\Spin_{d,n}$ have spin structures, so have $KO$-theory Thom classes $\lambda_{V_{d,n}^{\bot}} \in KO^{n-d} (\mathrm{Th} (V_{d,n}^{\bot}))$. These assemble to a unique spectrum $KO$-theory class $\kothom{d} \in KO^{-d} (\MTSpin (d))$, alias a spectrum map
$$\kothom{d}: \MTSpin (d) \lra \Sigma^{-d} \KO.$$
\begin{remark}
Arguments for this---especially uniqueness---are not so well-known, so we briefly give one (see also \cite[VIII.6]{BMMS} for a related discussion). The collection $\{\lambda_{V_{d,n}^{\bot}}\}_n$ defines a class in the inverse limit $\lim_n KO^{n-d}(\mathrm{Th} (V_{d,n}^{\bot}))$, so by Milnor's $\lim^1$ exact sequence it suffices to show that $\lim_n^1 KO^{n-d-1}(\mathrm{Th} (V_{d,n}^{\bot}))$ vanishes. The $KO$-theory Thom isomorphisms given by the $\lambda_{V_{d,n}^{\bot}}$ show that the inverse system $\{KO^{n-d-1}(\mathrm{Th} (V_{d,n}^{\bot}))\}_n$ is pro-isomorphic to the inverse system $\{KO^{-1}(\Gr^\Spin_{d,n})\}_n$. If we let $E\Spin(d)^{(k)}$ be the $k$-fold join of $\Spin(d)$, which is the $k$th step in Milnor's model for $E\Spin(d)$, with quotient $B\Spin(d)^{(k)} := E\Spin(d)^{(k)}/\Spin(d)$, then as $\Gr^\Spin_{d,n} \to B\Spin(d)$ is $(n-d-1)$-connected we can find a lift $B\Spin(d)^{(k)} \to \Gr^\Spin_{d,n}$ as long as $n \gg k$. This gives a map of direct systems of spaces with homotopy 
equivalent direct limits, and so a pro-homomorphism
$$\Phi: \{KO^{-1}(\Gr^\Spin_{d,n})\}_n \lra \{KO^{-1}(B\Spin(d)^{(k)})\}_k$$
with an associated map of Milnor exact sequences
\begin{equation*}
\xymatrix{
\lim\limits_n{\!}^1 KO^{-1}(\Gr^\Spin_{d,n}) \ar@{^{(}->}[r] \ar[d]^-{\lim^1 \Phi} & KO^{0}(B\Spin(d)) \ar@{->>}[r] \ar@{=}[d]& \lim\limits_n KO^{0}(\Gr^\Spin_{d,n}) \ar[d]^-{\lim \Phi}\\
\lim\limits_k{\!}^1 KO^{-1}(B\Spin(d)^{(k)}) \ar@{^{(}->}[r] & KO^{0}(B\Spin(d)) \ar@{->>}[r] & \lim\limits_k KO^{0}(B\Spin(d)^{(k)}).
}
\end{equation*}
But $\{KO^{-1}(B\Spin(d)^{(k)})\}_k$ is Mittag-Leffler by \cite[Corollary 2.4]{AtiyahSegal} (or rather its analogue for $KO$-theory, which may be deduced from \cite[Theorem 7.1]{AtiyahSegal}) so has vanishing $\lim^1$, so by the commutativity of the left-hand square the inverse system $\{KO^{-1}(\Gr^\Spin_{d,n})\}_n$ also has vanishing $\lim^1$.
\end{remark}

The infinite loop map of $\kothom{d}$ is denoted $\Ahat_d: =\loopinf{}\kothom{d}: \MTSpin (d) \to \loopinf{+d} \KO$.
With these definitions, we arrive at the following version of the index theorem.

\begin{thm}\label{cl-index-2}
Let $\pi: E \to X$ be a bundle of closed $d$-dimensional spin manifolds. Then the maps
$$\ind{E}, (\loopinf{} \kothom{d}) \circ \alpha_E: X \lra \loopinf{+d} \KO$$
are weakly homotopic.
\end{thm}
 
The translation of Theorem \ref{cl-index-1} into Theorem \ref{cl-index-2} is exactly parallel to the translation of the complex family index theorem described in \cite{Eb09}.

For manifold bundles with nonempty boundary, we need a psc metric $h \in \Riem^+ (M)$ to be able to talk about $\ind{E,h}:=\ind{E,dt^2+ h}\in KO^{-d} (X)$. To express the index in this situation in terms of homotopy theory, an additional hypothesis on $h$ is needed.
\begin{thm}\label{cl-index-3}
Let $\pi: E \to X$ be a bundle of $d$-dimensional spin manifolds, with trivialised boundary $X \times M$. Let $h \in \Riem^+ (M)$ and let $V:M \rightsquigarrow \emptyset$ be a nullbordism which carries a psc metric $g \in \Riem^+ (V)_{h}$. Then the maps
$$\ind{E,h}, (\loopinf{} \kothom{d}) \circ \alpha_{E,V}: X \lra \loopinf{+d} \KO$$
are weakly homotopic.
\end{thm}

\begin{proof}
By the additivity theorem (Corollary \ref{additivity-concordance-version}), $\ind{E,h} = \ind{E \cup_{\partial E} (X \times V)}$. The result follows in a straightforward manner from Theorem \ref{cl-index-2}, by the definition of $\alpha_{E,V}$ as $\alpha_{E \cup_{\partial E} (X \times V)}$.
\end{proof}

In most cases of interest to us, we will be able to take $V=W^{op}$, where $W$ is a fibre of $E$.

\subsubsection{Spin fibre bundles and the index difference}\label{subsec:spin-fibre-bundles-and-inddiff}
We now show how to fit the index difference into this context.
Let $W$ be a $d$-dimensional spin manifold with boundary $M$ and collar $[-\epsilon,0] \times M \subset W$, such that $(W,M)$ is $1$-connected. Let $\pi : E \to X$ be a smooth fibre bundle with fibre $W$ and structure group $\Diff_\partial (W)$, the diffeomorphisms which fix the collar pointwise, and with underlying $\Diff_\partial (W)$-principal bundle $Q \to X$. We assume that $X$ is paracompact. By Lemma \ref{lem:existence-spin-structure}, there is a spin structure on the vertical tangent bundle $T_v E \to E$, which is constant along $\partial E = X \times M \subset E$. 
Let $h_0 \in \Riem^+ (M)$ be fixed and write $\Riem^+ (W)_{h_0} = \Riem^+ (W)_{h_0}^\epsilon$, on which $\Diff_\partial(W)$ acts by pullback of metrics; there is an induced fibre bundle
\begin{equation*}
p:Q \times_{\Diff_\partial (W)} \Riem^+ (W)_{h_0} \lra X.
\end{equation*}
Observe that a point in $Q \times_{\Diff_\partial (W)} \Riem^+ (W)_{g_0}$ is a pair $(x,g)$, where $x \in X$ and $g$ is a psc metric on $\pi^{-1} (x)$ with boundary condition $h_0$. Choose a basepoint $x_0 \in X$ and identify $\pi^{-1}(x_0)$ with $W$. Then $p^{-1}(x_0)$ may be identified with $ \cR^+(W)_{h_0}$ and we also choose a basepoint $g_0\in p^{-1}(x_0)$.

We will now introduce an element $\beta =\beta_{\pi,g_0}\in KO^{-d} (p)$, depending only on the bundle $\pi$ and the metric $g_0$. 
To begin the construction of $\beta$, choose a fibrewise Riemannian metric $k$ on the fibre bundle $\pi : E\to X$ such that
\begin{enumerate}[(i)]
\item on $\pi^{-1}(x_0)=W$, the metric $k$ is equal to $g_0$,
\item near $\partial E$, $k$ has a product structure and the restriction to $\partial E$ is equal to $h_0$.
\end{enumerate}
It is easy to produce such a metric using a partition of unity, and of course $k$ will typically not have positive scalar curvature. Now let $\tilde{E} \to \Cyl (p)$ be the pullback of the bundle $\pi$ along the natural map $\Cyl (p)\to X$.
The bundle $\tilde{E}$ has the following fibrewise metric: over a point $x \in X \subset \Cyl (p)$, we take the metric $k_x$, and over a point $(x,g,t) \in Q \times_{\Diff_\partial (W)} \Riem^+ (W)_{h_0} \times [0,1]$, we take the metric $(1-t)g + t k_x$. This metric satisfies the boundary condition $h_0$, and it has psc if $t=0$ or if $(x,g)=(x_0,g_0)$.
Since $E$ and hence $\tilde{E}$ is spin, there is a Dirac operator for this metric, so a well-defined element $\beta \in KO^{-d} (p)$ (defined using the elongation of the bundle $\tilde{E}$). The following properties of this construction are immediately verified (for the last one, one uses Corollary \ref{additivity-concordance-version}).

\begin{prop}\label{relative-index-class-properties}\mbox{}
\begin{enumerate}[(i)]
\item\label{it:relative-index-class-properties:1} The base class of $\beta$ is the usual family index of $E$, with the metric $k$: that is $\bas (\beta) = \ind{E,k} \in KO^d (X)$ (this class only depends on the boundary condition $h$).
\item\label{it:relative-index-class-properties:2} The transgression of $\beta$ to $p^{-1}(x_0) = \Riem^+ (W)_{h_0}$ is the index difference class $\inddiff_{g_0} \in \Omega KO^{-d} (\Riem^+ (W)_{h_0})$. 
\item The class $\beta$ is natural with respect to pullback of fibre bundles.
\item Let $V: M \leadsto M'$ be a spin cobordism and $m \in \Riem^+ (V)_{g_0,g_1 }$ be a psc metric. Let $\pi': E \cup_{\partial E} (X \times V)\to X$ be the bundle obtained by fibrewise gluing in $V$. We obtain a commutative diagram
\[
 \xymatrix{
 Q \times_{\Diff_\partial (W)} \Riem^+ (W)_{h_0} \ar[rr]^{\mu_m}\ar[dr]_{p} & &  Q \times_{\Diff_\partial (W)} \Riem^+ (W\cup V)_{h_1} \ar[dl]^{p'}\\
 & X ,& 
}
\]
and the image of $\beta_{\pi',g_0\cup m} \in KO^{-d}(p')$ in $KO^{-d}(p)$ agrees with $\beta_{\pi,g_0}$
\end{enumerate}
\end{prop}

Proposition \ref{relative-index-class-properties} gives a more precise statement of the diagram \ref{intro-index-fibre-diagram}, as follows.
Let $p$ be the universal $W$-bundle over $B  \Diff_{\partial}(W)$. The cycle $\beta$ defines a nullhomotopy of the map $ \ind{E,k} \circ p$, which can be viewed as a map $\Riem^+ (W)_h \hq \Diff_{\partial} (W) \to \Path \loopinf{+d} \KO$ to the path space. On the fibre of $\Riem^+ (W)_h \hq \Diff_{\partial} (W)\to B \Diff_{\partial} (W)$, $\beta$ induces the index difference map $\inddiff_{g_0}$. However, in the next chapter, we will use Proposition \ref{relative-index-class-properties} instead of the more informal diagram \ref{intro-index-fibre-diagram}.

\section{Proof of the main results}\label{sec:PfMain}

In this section we will prove Theorem \ref{factorization-theorem:even-case}. Before beginning the proof in 
earnest, we will establish a result (Theorem \ref{abelianness}) which will be a fundamental tool in the proof, but is 
also of independent interest.
For the purpose of exposition we have structured the proof of Theorem \ref{factorization-theorem:even-case} into three parts.
In Section \ref{sec:ConstructingMaps} the main constructive argument of the paper is carried out, which is stated as Theorem \ref{factorization-theorem}. This 
construction assumes the existence of a space $X$ which can be approximated  homologically by $\hocolim_k B \Diff_{\partial} (W_k)$, where $W_k$ is a certain sequence of spin cobordisms, and the output of this construction is a map $\Omega X \to \Riem^+ (W_k)$.
Sections \ref{sec:Starting} and \ref{finishingproof:alldim} provide the data assumed for Theorem \ref{factorization-theorem}. In Section \ref{sec:Starting}, we set up the general framework and finish the proof in the case $2n=6$ (which is done by directly quoting \cite{GRW}). Section \ref{finishingproof:alldim} deals with the general case, which instead uses \cite{GRWHomStab2}.


\subsection{Action of the diffeomorphism group on the space of psc metrics}\label{sec:action}

As in the previous section, for a manifold $W$ with boundary $\partial W$ and collar $[-\epsilon,0] \times \partial W \subset W$, and a boundary condition $h \in \cR^+(\partial W)$, we have a space $\cR^+(W)_h = \cR^+(W)_h^\epsilon$ and there is a right action of the group $\Diff_\partial(W)$ of diffeomorphisms of $W$ which are the identity on the collar on $\cR^+(W)_h$ by pullback of metrics. This in particular induces a homomorphism
\begin{equation}\label{action-homomorphism}
\Gamma (W) := \pi_0(\Diff_\partial(W)) \lra \pi_0 (\Aut(\Riem^+(W)_h)),
\end{equation}
to the group of homotopy classes of self homotopy equivalences of the space $\Riem^+(W)_h$.

\begin{thm}\label{abelianness}
Let $W$ be a compact simply-connected $d$-dimensional spin manifold with boundary $\partial W=S^{d-1}$. Assume that $d \geq 5$ and that $W$ is spin cobordant to $D^d$ relative to its boundary.
Then for $h=h_\round^{d-1}$ the image of the homomorphism {\rm(\ref{action-homomorphism})} is an abelian group.
\end{thm}
The conclusion can also be expressed by saying that the action in the homotopy category of $\Gamma (W)$ on $\Riem^+ (W)_{h_\round^{d-1}}$ is through an abelian group.

The proof of this theorem will consist of an application of the cobordism invariance of spaces of psc metrics (Theorem \ref{chernysh-walsh-theorem}) and a formal argument of Eckmann--Hilton flavour. We first present the formal argument, and then explain how to fit our geometric situation into this framework.

\begin{lem}\label{lem:abstract-Eckmann-Hilton-trick}
Let $\cC$ be a nonunital topological category with objects the integers, and let $G$ be a topological group which acts on $\cC$, i.e.\ $G$ acts on the morphism spaces $\cC(m,n)$ for all $m,n \in \bZ$, and the composition law in $\cC$ is $G$-equivariant. Contrary to the usual notation of category theory, let us write composition as
\begin{align*}
\cC(m,n) \times \cC(n,k) & \lra \cC(m,k)\\
(x, y) & \longmapsto x \cdot y.
\end{align*}
Suppose that
\begin{enumerate}[(i)]
\item $\cC(m,n)=\emptyset$ if $n \leq m$.
\item\label{it:EH:2} For each $m \neq 0$ there exists a $u_m \in \cC(m,m+1)$ such that the composition maps
\begin{align*}
u_m \cdot -:  \cC(m+1,n) &\lra \cC(m,n) \quad\text{for $n>m+1$}\\
- \cdot u_m :  \cC (n,m) &\lra \cC(n,m+1) \quad\text{for $n<m$}
\end{align*}
are homotopy equivalences.
\item\label{it:EH:3} There exists an $x_0 \in \cC(0,1)$ such that the composition maps
\begin{align*}
- \cdot x_0:  \cC(m,0) &\lra \cC (m,1) \quad\text{for $m<0$}\\
x_0 \cdot - :  \cC (1,n) &\lra \cC (0,n) \quad\text{for $n>1$}
\end{align*}
are homotopy equivalences.
\item The $G$-action on $\cC (m,n)$ is trivial unless $m  \leq 0$ and $1 \leq n$.
\end{enumerate}
Then for any $f, g \in G$ the maps $f,g : \cC(0,1)\to \cC(0,1)$ commute up to homotopy. 
\end{lem}

\begin{remark}
To understand the motivation for the above set-up and its proof, the reader may consider the following discrete analogue. Let $M$ be a unital monoid, and $X$ be a set with commuting left and right $M$-actions, and in addition an action of a group $G$ on $X$ by left and right $M$-set maps. Finally, suppose that there is an $x_0 \in X$ such that the maps $x_0 \cdot -, - \cdot x_0 : M \to X$ are both bijections. Then $G$ acts on $X$ through its abelianisation.
\end{remark}

\begin{proof}
For a point $y$ in a space $Y$, we denote by $[y] \in \pi_0 (Y)$ the path component it belongs to. For $f \in G$ there are elements $y_f \in \cC (-1,0)$ and $z_f \in \cC(1,2)$ such that 
\begin{align*}
[y_f \cdot x_0] &= [u_{-1} \cdot fx_0] \in \pi_0 (\cC (-1,1))\\
[x_0 \cdot z_f] &= [fx_0 \cdot u_1] \in \pi_0 (\cC (0,2)). 
\end{align*}
To see this, first note that $u_{-1} \cdot fx_0 = f (u_{-1} \cdot x_0)$ as the composition is $G$-equivariant and $G$ acts trivially on $\mathcal{C}(-1,0)$, and then note that the maps 
$$\cC(-1,0) \stackrel{- \cdot x_0}{\lra} \cC (-1,1) \stackrel{f}{\longleftarrow} \cC(-1,1) \stackrel{u_{-1} \cdot -}{\longleftarrow} \cC (0,1)$$
are all homotopy equivalences and hence induce bijections on $\pi_0$. Choose $y_f$ so that $[y_f]$ corresponds to $[x_0]$ under these bijections. The argument for $z_f$ is analogous, using the homotopy equivalences
$$\cC(1,2) \stackrel{ x_0 \cdot -}{\lra} \cC (0,2) \stackrel{f}{\longleftarrow} \cC(0,2) \stackrel{ - \cdot u_1}{\longleftarrow} \cC (0,1)$$
instead. 
We now claim that the maps
$$ fg (- \cdot u_{-1} \cdot x_0 \cdot u_1), \,  gf (- \cdot u_{-1} \cdot x_0 \cdot u_1) : \cC(-2,-1) \lra \cC (-2,2)$$
are homotopic, which follows by the concatenation of homotopies
\begin{align*}
fg (- \cdot u_{-1} \cdot x_0 \cdot u_1) &= f (- \cdot u_{-1} \cdot g x_0 \cdot u_1) \\
&\simeq  f (- \cdot u_{-1} \cdot  x_0 \cdot z_g) \\
&= (- \cdot u_{-1} \cdot f x_0 \cdot z_g) \\
& \simeq (- \cdot y_f \cdot  x_0 \cdot z_g) \\
& \simeq (- \cdot y_f \cdot  gx_0 \cdot u_1) \\
& =  g(- \cdot y_f \cdot  x_0 \cdot u_1) \\
& \simeq g(- \cdot u_{-1} \cdot f x_0 \cdot u_1) \\
&= gf(- \cdot u_{-1} \cdot  x_0 \cdot u_1).
\end{align*}
As the map $-\cdot u_{-1} \cdot x_0 \cdot u_1 : \cC (-2,-1) \to \cC(-2,2)$, is a homotopy equivalence, it follows that the maps $fg, gf: \cC(-2,2) \to \cC(-2,2)$ are homotopic. 
Finally, the diagram
\begin{equation*}
\xymatrix{
\cC(0,1) \ar[d]^{h} \ar[rr]^{u_{-2} \cdot u_{-1} \cdot - \cdot u_1} & &\cC(-2,2) \ar[d]^{h}\\
\cC(0,1) \ar[rr]^{u_{-2} \cdot u_{-1} \cdot - \cdot u_1} & &\cC(-2,2)
}
\end{equation*}
commutes for each $h \in G$ and the horizontal maps are homotopy equivalences. It follows that $fg \simeq gf: \cC(0,1) \to \cC(0,1)$, as required.
\end{proof}

\begin{proof}[Proof of Theorem \ref{abelianness}]
Consider a closed disc $D^d \subset S^{d-1} \times (0,1)$. By Theorem \ref{chernysh-walsh-theorem} we may find $h \in \Riem^+ (S^{d-1} \times [0,1])_{h_{\round},h_{\round}}$ which is equal to $g_{\tor}^d$ in the disc $D^d$ and is in the same path component as the cylinder metric $h_{\round}^{d-1} + dt^2$. By cutting out the disc, we obtain a psc metric on $T:= S^{d-1} \times [0,1] \setminus \inter(D^d)$, also denoted $h$. Let us denote by $P = S^{d-1}$ the boundary component of $T$ created by cutting out this disc.

Gluing in $(T, h)$ along $S^{d-1} \times \{0\}$ gives a map 
$$\mu_h : \Riem^+ (W)_{h_{\round}} \lra \Riem^+ (W \cup_{S^{d-1} \times \{0\}} T)_{h_{\round}, h_{\round}}.$$
This is a homotopy equivalence, because its composition with the map 
$$\Riem^+ (W\cup_{S^{d-1} \times \{0\}}T)_{h_\round, h_{\round}} \lra \Riem^+ (W\cup_{S^{d-1} \times \{0\}}T \cup_{P} D^d)_{h_\round},$$
which glues in the torpedo metric on $D^d$, is homotopic to the gluing map $\mu_{h_\round+dt^2}$ after identifying the target with $\Riem^+ (W\cup_{S^{d-1} \times \{0\}}S^{d-1} \times [0,1])_{h_\round}$. These last two maps are homotopy equivalences, by Theorem \ref{chernysh-walsh-theorem} and Corollary \ref{cor:collarstretching} (\ref{it:CollarStretching}) respectively. Let us write $V:= W \cup_{S^{d-1} \times \{0\}} T$, considered as a cobordism $S^{d-1}=P \leadsto S^{d-1}\times \{1\} = S^{d-1} $. 

We now let $G:= \Diff_{\partial} (W)$ and $\cC(0,1)= \Riem^+ (V)_{h_{\round},h_{\round}}$. The group $G$ acts on $V$, by extending diffeomorphisms as the identity on $T$, and the gluing map $\mu_h : \Riem^+ (W)_{h_{\round}} \to \Riem^+ (V)_{h_{\round},h_{\round}}$ is $G$-equivariant and a homotopy equivalence. To prove the theorem it is therefore enough to show that the image of the action map $\pi_0 (G) \to \pi_0 (\Aut \cC(0,1))$ is an abelian group. This will follow by an application of Lemma \ref{lem:abstract-Eckmann-Hilton-trick}. Define a nonunital category $\cC$ having objects the integers, and
\[
\cC(m,n)=
\begin{cases}
\Riem^+ ((S^{d-1} \times [m,0]) \cup V \cup (S^{d-1} \times [1,n]))_{h_{\round},h_{\round}} & m \leq 0, n \geq 1,\\
\Riem^+ (S^{d-1} \times [m,n])_{h_{\round},h_{\round}} & m <n\leq 0 \text{ or }  1 \leq m < n,\\
\emptyset & \text{otherwise.}
\end{cases}
\]
Composition is defined by the gluing maps, and the $G$-action is defined by letting $G$ act trivially on cylinders. For $m \neq 0$, we let $u_m \in \cC(m,m+1) = \Riem^+ (S^{d-1} \times [m,m+1])_{h_{\round},h_{\round}}$ be the cylinder metric $h_{\round} + dt^2$, so that assumption (\ref{it:EH:2}) of Lemma \ref{lem:abstract-Eckmann-Hilton-trick} holds by Corollary \ref{cor:collarstretching}.
Finally, note that $V$ is cobordant relative to its boundary to $S^{d-1} \times [0,1]$ and so Theorem \ref{lemma:GluingNullCob} shows that there is a psc metric $x_0 \in \cC(0,1)$ satisfying assumption (\ref{it:EH:3}) of Lemma \ref{lem:abstract-Eckmann-Hilton-trick}.
\end{proof}

\subsection{Constructing maps into spaces of psc metrics}\label{sec:ConstructingMaps}

\subsubsection*{Statement of the main construction theorem}

Let $2n \geq 6$ and suppose that
$$W : \emptyset \leadsto S^{2n-1}$$
is a simply-connected spin cobordism, which is spin cobordant to $D^{2n}$ relative to its boundary. Let
$$K := ([0,1] \times S^{2n-1}) \# (S^n \times S^n) : S^{2n-1} \leadsto S^{2n-1}.$$
For $i = 0,1,2,\ldots$ let $K\vert_i := S^{2n-1}$ and $K\vert_{[i,i+1]} : K\vert_i \leadsto K\vert_{i+1}$ be a copy of $K$. Also, consider $W$ as a cobordism to $K\vert_0$. Then we write
$$W_k:= W \cup K\vert_{[0,k]} := W \cup \bigcup_{i=0}^{k-1} K\vert_{[i,i+1]}: \emptyset\leadsto  K\vert_k$$
for the composition of $W$ and $k$ copies of $K$, so $W_0 = W$. Define the group $D_k:= \Diff_{\partial} (W_k)$, and write $B_k:= B D_k$ for the classifying space of this group and $\pi_k : E_k :=E D_k \times_{D_k} W_k \to B_k$ for the universal bundle. There is a homomorphism $D_k \to D_{k+1}$ given by extending diffeomorphisms over $K\vert_{[k, k+1]}$ by the identity, and this induces a map $\lambda_k: B_k \to B_{k+1}$ on classifying spaces. Let $B_\infty := \hocolim_k B_k$ denote the mapping telescope.

By Lemma \ref{lem:existence-spin-structure}, there is a unique fibrewise spin structure on each bundle $E_k$. Thus there is the family of Dirac operators on the fibre bundles $\pi_k$. We now list further hypotheses on the manifold $W$, which will allow us to carry out an obstruction-theoretic argument. We will later make particular choices of $W$ and construct the data assumed in these assumptions.

\begin{assumptions}\label{key-properties}
We are given
\begin{enumerate}[(i)]
\item\label{it:properties:1}  a space $X$ with  an acyclic map $\Psi:B_{\infty} \to X$,

\item\label{it:properties:2} a class $\hat{a} \in KO^{-2n} (X)$ such that $\Psi^*(\hat{a})$ restricts to $\ind{E_k,h_\round^{d-1}} \in KO^{-2n} (B_k)$, for all $k$, up to phantom maps.
\end{enumerate}
\end{assumptions}
\begin{remark}
Recall that a map $f:X \to Y$ of spaces is called \emph{acyclic} if for each $y \in Y$ the homotopy fibre $\hofib_y (f)$ has the singular homology of a point. This is equivalent to $f$ inducing an isomorphism on homology for every system of local coefficients on $Y$; if $Y$ is not simply-connected then it is stronger than merely being a homology equivalence. If $Y$ (and hence $X$) is connected and $F= \hofib_y (f)$, we get from the long exact homotopy sequence 
\[
\pi_2 (X) \to \pi_2 (Y) \to \pi_1 (F) \to \pi_1 (X) \to \pi_1 (Y) \to 1
\]
that $\ker (\pi_1 (f))$ is a quotient of the perfect group $\pi_1 (F)$ and hence itself perfect. It follows that $f: X\to Y$ may be identified with the Quillen plus construction applied to the perfect group $\ker (\pi_1 (f))$, by \cite[Theorem 5.2.2]{Ros} \cite[Theorem 3.5]{HH}.
\end{remark}
\begin{remark}
Recall that a map $f : X \to Y$ to a pointed space is called \emph{phantom} if it is weakly homotopic to the constant map to the basepoint. Maps $f_0, f_1 : X \to \Omega Z$ to a loop space are said to \emph{agree up to phantom maps} if their difference $f_0 \cdot f_1^{-1}$ is phantom: this is equivalent to $f_0$ and $f_1$ being weakly homotopic in the sense of Definition \ref{defn:intro:weaklyhomotopic}.
\end{remark}

In order to make certain homotopy-theoretic arguments we (functorially) replace certain spaces by CW complexes, by writing $\gR_k:=|\Sing_{\bullet} \Riem^+ (W_k)_{h_\round^{2n-1}}|$ and $\gX :=|\Sing_{\bullet}  X|$. The rest of this section is devoted to the proof of the following result.

\begin{thm}\label{factorization-theorem}
If the spin cobordism $W : \emptyset \leadsto S^{2n-1}$ is such that $W$ is simply-connected and is spin cobordant to $D^{2n}$ relative to its boundary,  and Assumptions \ref{key-properties} hold, then there is a map $\rho:\Omega \gX \to \gR_0$ such that the composition with $\inddiff_{m_0}: \gR_0 \to \Omega^{\infty+2n+1} \KO$ agrees with $\Omega \hat{a}$, up to phantom maps.
\end{thm}

In Sections \ref{sec:Starting} and \ref{finishingproof:alldim} we will show how a tuple $(W, X, \hat{a})$ satisfying these hypotheses can be constructed, but in the rest of this section we will prove Theorem \ref{factorization-theorem}, and so suppose that the hypotheses of this theorem hold. 

\subsubsection*{Setting the stage for the obstruction argument}

Theorem \ref{factorization-theorem} will be proved by an obstruction-theoretic argument, which needs some preliminary constructions. Before we begin, let us collect the important consequences of our work so far.

\begin{prop}\label{abelianness-lemma}\mbox{}
\begin{enumerate}[(i)]
\item\label{it:abelianness-lemma:1} There is a surgery equivalence $\cR^+(W_0)_{h_\round^{2n-1}} \simeq \cR^+(D^{2n})_{h_\round^{2n-1}}$, and so in particular $\cR^+(W_0)_{h_\round^{2n-1}}$ is non-empty. Thus we may choose an $g_{-1} \in \cR^+(W_0)_{h_\round^{2n-1}}$ which lies in the component of $g_\tor^{2n} \in \cR^+(D^{2n})_{h_\round^{2n-1}}$ under the surgery equivalence.
\item\label{it:abelianness-lemma:2} There are metrics $g_i \in \cR^+(K\vert_{[i,i+1]})_{h_\round^{2n-1}, h_\round^{2n-1}}$ so that the gluing maps
$$\mu_{g_i} : \Riem^+ (W_i)_{h_\round^{2n-1}} \lra \Riem^+ (W_{i+1})_{h_\round^{2n-1}}$$
are homotopy equivalences. Let 
$$m_i := g_{-1} \cup g_0 \cup g_1 \cup \cdots \cup g_{i-1} \in \cR^+(W_i)_{h_\round^{2n-1}}.$$
\item\label{it:abelianness-lemma:3} The action homomorphism $\Gamma (W_k) \to \pi_0 (\Aut (\cR^+(W_k)_{h_\round^{2n-1}}))$ has abelian image.
\end{enumerate}
\end{prop}

\begin{proof}
This is straightforward from the previous work:
Because $W_0 = W$ is spin cobordant to $D^{2n}$ relative to its boundary by assumption, the first part follows from Theorem \ref{surgery-invariance}. Because the manifold $K\vert_{[i,i+1]} = ([0,1] \times S^{2n-1}) \# (S^n \times S^n)$ is cobordant to a cylinder relative to its boundary, the second part follows from Theorem \ref{lemma:GluingNullCob}.
The third assertion follows from Theorem \ref{abelianness}, again using the assumption that $W_0$ is spin cobordant to $D^{2n}$ relative to its boundary.
\end{proof}

We introduce the abbreviations $R_k:= \Riem^+ (W_k)_{h_\round^{2n-1}}$, $T_k:= ED_k \times_{D_k} R_k$, write $p_k : T_k \to B_k$ for the projection map and write $\mu_k := \mu_{h_k}: R_{k} \to R_{k+1}$ for the gluing maps defined by the metrics $h_k$ of Proposition \ref{abelianness-lemma} (\ref{it:abelianness-lemma:2}). The map $\mu_k$ is $D_{k}$-equivariant (by construction), so there is an induced map between the Borel constructions
\begin{equation*}
\xymatrix{
R_{k}  \ar[r]^{\mu_k} \ar[d] & R_{k+1} \ar[d]\\
T_{k} \ar[r]^{\nu_k} \ar[d]^-{p_k} & T_{k+1}\ar[d]^-{p_{k+1}}\\
B_{k} \ar[r]^{\lambda_k} & B_{k+1}.
}
\end{equation*}
By Proposition \ref{abelianness-lemma} (\ref{it:abelianness-lemma:2}), the top map is a weak homotopy equivalence, so the lower square is weakly homotopy cartesian. 
Using the (unique) spin structure that each fibre bundle $\pi_k : E_k \to B_k$ has (cf.\ Lemma \ref{lem:existence-spin-structure}), the construction of Section \ref{sec:FibBunIndDiff} gives relative $KO$-classes $\beta_k \in KO^{-d} (p_k) $. Let $\hocolim_k p_k : \hocolim_k T_k \to \hocolim_k B_k$ be the induced map on mapping telescopes.

\begin{prop}\label{concordance-stable}
There is a relative $KO$-class $\beta_{\infty}  \in KO^{-d} (\hocolim_k p_k)$, such that the restriction to $KO^{-d}(p_k)$ is equal to $\beta_k$.
\end{prop}

\begin{proof}
The map
$$KO^{-d}(\hocolim_k p_k) \lra  \varprojlim KO^{-d}(p_k)$$
is surjective by Milnor's $\lim^1$ sequence. The classes $\beta_k \in KO^{-d}(p_k)$ give a consistent collection by Proposition \ref{relative-index-class-properties}, and so there exists a $\beta_\infty \in KO^{-d}(\hocolim_k p_k)$ restricting to them.
\end{proof}

Let us move to the simplicial world, in which we will carry out the proof of Theorem \ref{factorization-theorem}. The simplicial group $\Sing_{\bullet} D_k$ (obtained by taking singular simplices) acts on the simplicial set $\Sing_{\bullet} R_k$. We denote the geometric realisations by $\gD_k:=|\Sing_{\bullet} D_k|$ and $\gR_k:=|\Sing_{\bullet} R_k|$, and also write $\gB_k :=B\gD_k$. Let $\gD_{\infty}$ and $\gR_{\infty}$ be the colimits of the induced maps
$$\lambda_k : \gD_k \lra \gD_{k+1} \quad \text{ and }\quad \mu_k : \gR_k \lra \gR_{k+1},$$
which, as these maps are cellular inclusions, are also homotopy colimits. Let $\gB_\infty := B\gD_\infty$ and $\Psi : \gB_\infty \to \gX := \vert \Sing_\bullet X \vert$ be a map in the homotopy class induced by
$$\gB_\infty = \underset{k \to \infty}\colim \, \gB_k \overset{\simeq}\longleftarrow \underset{k \to \infty}\hocolim \, \gB_k \lra \vert\Sing_\bullet(B_\infty)\vert \overset{\vert\Sing_\bullet \Psi\vert}\lra \vert\Sing_\bullet(X)\vert = \gX.$$
Let 
$$\gp_k:\gT_k:= E \gD_k \times_{\gD_k} \gR_k  \lra \gB_k$$
be the induced fibre bundles, for $k \in \bN \cup \{\infty\}$. 
There is a commutative diagram
\[
\xymatrix{
\gT_{\infty} \ar[d]^{\gp_{\infty}} &  \hocolim_k \gT_k  \ar[d]^{\hocolim_k \gp_k} \ar[l]_-{\simeq} \ar[r] & \hocolim_k T_k \ar[d]^{\hocolim_k p_k}\\
\gB_{\infty} & \hocolim_k \gB_k \ar[l]_-{\simeq} \ar[r] &\hocolim_k B_k.
}
\]

As the four left spaces in the diagram are all CW-complexes, there exists a class $\bm{\beta}_{\infty} \in KO^{-2n} (\gp_{\infty})$ whose pullback to $KO^{-2n} (\hocolim_k \gp_k)$ coincides with the pullback of the class $\beta_{\infty}$ constructed in Proposition \ref{concordance-stable} to $\Cyl (\hocolim_k \gp_k)$. 
The square
\begin{equation}\label{cartesian}
\begin{gathered}
\xymatrix{
\gT_k \ar[d]^{\gp_k} \ar[r] & \gT_{\infty}\ar[d]^{\gp_{\infty}} \\
\gB_k \ar[r] & \gB_{\infty}.
}
\end{gathered}
\end{equation}
is homotopy cartesian, as both maps are Serre fibrations, the map $\gR_k \to \gR_\infty$ on fibres is a homotopy equivalence by Proposition \ref{abelianness-lemma} (\ref{it:abelianness-lemma:2}), and all the spaces have the homotopy type of CW complexes. The following property of the class $\bm{\beta}_{\infty}$ is clear from the construction and Proposition \ref{concordance-stable}.

\begin{lem}\label{lem:RelCycle}
The pullback of $\bm{\beta}_{\infty}$ to each $\Cyl(\gp_k)$ agrees with the pullback of $\beta_k$ along the weak homotopy equivalence $\Cyl(\gp_k) \to \Cyl(p_k)$.
\end{lem}

\subsubsection*{The obstruction argument}

The topological group $\gD_\infty$ acts on $\gR_\infty$, and the map $\gp_\infty : \gT_\infty \to \gB_\infty$ is the associated Borel construction. In particular, there is an associated homomorphism
$$\gD_\infty \lra \mathrm{Aut}(\gR_\infty)$$
of topological monoids, which on classifying spaces gives a map
$$h: \gB_\infty \lra B\mathrm{Aut}(\gR_\infty).$$

\begin{lem}\label{second-abelianness-lemma}
The monodromy map $\pi_1 (\gB_{\infty}) \to \pi_0 \Aut (\gR_{\infty})$ has abelian image.
\end{lem}
\begin{proof}
Because a group is commutative if all its finitely generated subgroups are, and because the diagram \eqref{cartesian} is homotopy cartesian, it is enough to prove that the monodromy map $\pi_1 (\gB_k) \to \pi_0 (\Aut (\gR_k))$ has abelian image for each $k$. But $\pi_1 (\gB_k)=\pi_0(\gD_k) = \Gamma(W_k)$, and under this identification, the monodromy map becomes the homomorphism $\Gamma(W_k) \to \pi_0 (\Aut (\gR_k))$ induced from the action, and we have shown that this acts through an abelian group in Proposition \ref{abelianness-lemma} (\ref{it:abelianness-lemma:3}).
\end{proof}

\begin{prop}\label{result-of-obstruction-argument}
The exists a commutative and homotopy cartesian square
\begin{equation*}
\begin{gathered}
\xymatrix{
\gT_{\infty} \ar[r] \ar[d]^{\gp_{\infty}} & \gT_{\infty}^+ \ar[d]^{\gp_{\infty}^+}\\
\gB_{\infty} \ar[r]^{\Psi} & \gX.
}
\end{gathered}
\end{equation*}
Moreover, there is a unique class $\bm{\beta}_\infty^+ \in KO^{-2n}(\gp^+_\infty)$ which restricts to $\bm{\beta}_{\infty} \in KO^{-2n} (\gp_{\infty})$.
\end{prop}

\begin{proof}
First, we invoke May's general classification theory for fibrations \cite{May}. The result (loc.\ cit.\ Theorem 9.2) is that there is a universal fibration $E \to B \Aut (\gR_{\infty})$ with fibre $\gR_{\infty}$ over the classifying space of the topological monoid $\Aut (\gR_{\infty})$ and a homotopy cartesian diagram
\begin{equation*}
 \xymatrix{
\gT_{\infty} \ar[d]^-{\gp_\infty} \ar[r] & E \ar[d]\\
 \gB_{\infty} \ar[r]^-h & B \Aut (\gR_{\infty}).
 }
\end{equation*}

By Assumption \ref{key-properties} (\ref{it:properties:1}) there is an acyclic map $\Psi : \gB_\infty \to \gX$. We claim that the obstruction problem 
\begin{equation*}
\xymatrix{
\gB_{\infty} \ar[r]^{\Psi} \ar[dr]_{h} & \gX \ar@{..>}[d]^{g}\\
 &  B \Aut (\gR_{\infty})
}
\end{equation*}
can be solved, up to homotopy. Since $\Psi$ is acyclic, by the universal property of acyclic maps \cite[Proposition 3.1]{HH} it is sufficient to prove that $\ker (\pi_1 (\Psi)) \subset \ker (\pi_1 (h))$. The group $\ker (\pi_1 (\Psi))$ is perfect, but 
by Lemma \ref{second-abelianness-lemma} the group $\pi_1 (h)$ has abelian image, and so $\pi_1 (h)(\ker (\pi_1 (\Psi)))$ is trivial; in other words, $\ker (\pi_1 (\Psi)) \subset \ker (\pi_1 (h))$. 

Therefore, we have constructed a factorisation 
$$h: \gB_\infty \overset{\Psi}\lra \gX \overset{g}\lra B \Aut (\gR_{\infty})$$
up to homotopy. 
Let us denote by $\gp_\infty^+ : \gT_\infty^+ \to \gX$ the fibration obtained by pulling $E$ back along the map $g$, so there is an induced commutative square
\begin{equation*}
\begin{gathered}
\xymatrix{
\gT_{\infty} \ar[r]^{\Psi'} \ar[d]^{\gp_{\infty}} & \gT_{\infty}^+ \ar[d]^{\gp_{\infty}^+}\\
\gB_{\infty} \ar[r]^{\Psi} & \gX.
}
\end{gathered}
\end{equation*}
By construction, the square is homotopy cartesian. Therefore, since $\Psi$ is acyclic, so is $\Psi'$. Thus the class $\bm{\beta}_\infty \in KO^{-d}(\gp_\infty)$ from Lemma \ref{lem:RelCycle} extends to a unique class $\bm{\beta}_\infty^+ \in KO^{-d}(\gp^+_\infty)$. 
\end{proof}

Now we \emph{define} the map $\rho_{\infty}: \Omega \gX \to \gR_\infty$ as the fibre transport of the fibration $\gp_{\infty}^+$.
Since $\gR_0 \to \gR_{\infty}$ is a homotopy equivalence, we can lift $\rho_{\infty}$ to a map $\rho: \Omega \gX \to \gR_0$. It remains to check that $\rho$ has the property stated in Theorem \ref{factorization-theorem}.
We apply the relative index construction (i.e. Corollary \ref{relativeindex:formal-lemma-homotopy-version}) to the class $\bm{\beta}_\infty^+$ and obtain a homotopy commutative diagram
\begin{equation*}
\xymatrix{
& \Omega \gX  \ar[dl]_{\rho}\ar[d]^{\rho_{\infty}} \ar[drr]^{\Omega \bas(\bm{\beta}_\infty^+)} &  &\\
\gR_0 \ar[r]^{\simeq} &\gR_{\infty} \ar[rr]_-{\trg (\bm{\beta}_\infty^+)} & &   \loopinf{+2n+1} \KO.
}
\end{equation*}

The composition $\gR_0 \to \gR_{\infty} \stackrel{\trg (\bm{\beta}_\infty^+)}{\to}  \loopinf{+2n+1} \KO$ is homotopic to (the pullback to $\gR_0$ of) the index difference with respect to the psc metric $m_0 \in \Riem^+ (W_0)_{h_\round^{2n-1}}$, by construction, Lemma \ref{lem:RelCycle} and Proposition \ref{relative-index-class-properties}. Recall that in Assumption \ref{key-properties} (\ref{it:properties:2}) we have chosen a map $\hat{a}: X \to \loopinf{+2n} \KO$ which restricts to the family index on each $B_k$.

\begin{prop}\label{prop:phantom}\mbox{}
\begin{enumerate}[(i)]
\item The classes $\bas(\bm{\beta}_\infty^+)$ and $\hat{a}$ in $[X,\loopinf{+2n}\KO]$ agree up to phantom maps.
\item The classes $\Omega \bas(\bm{\beta}_\infty^+)$ and $\Omega  \hat{a}$ in $[\Omega X,\loopinf{+2n+1}\KO]$ agree up to phantom maps.
\end{enumerate}
\end{prop}

The proof of Theorem \ref{factorization-theorem} will be completed by the second statement of Proposition \ref{prop:phantom}. 
For the proof, we need general results about the relation between phantom maps into loop spaces and homology equivalences.

\begin{lem}\label{preservation-of-phantoms0}
If the pointed map $k:X \to \Omega Z$ is a phantom, then so is its adjoint $k^{ad}: \Sigma X \to Z$ under the loop/suspension adjunction.
If $f:X \to \Omega Z$ is a pointed phantom map, then so is the loop map $\Omega f: \Omega X \to \Omega^2 Z$. 
\end{lem}
\begin{proof}
Without loss of generality, we can assume that $X$ is a connected CW complex. Let $l:F \to \Sigma X$ be a map from a finite CW complex, and we wish to show that $k^{ad} \circ l$ is nullhomotopic. After adding a disjoint basepoint to $F$, we may assume that $l$ is a pointed map. As any finite subcomplex of $\Sigma X$ is contained in the suspension $\Sigma L$ of a finite (pointed) subcomplex $L \subset X$, we can write $l= (\Sigma i) \circ j$, where $i: L \to X$ is the inclusion and $j: F \to \Sigma L$ some map. But then $k^{ad} \circ l=k^{ad} \circ \Sigma i \circ j = (k \circ i)^{ad} \circ j$. Since $k$ is a phantom and $L$ is finite, $k \circ i$ is nullhomotopic. Because the target of $k\circ i$ is an $H$-space, $k \circ i$ is nullhomotopic as a pointed map, by \cite[4A.2, 4A.3]{Hatcher}. Hence the adjoint $(k \circ i)^{ad}$ is nullhomotopic as desired.

The second part is similar: let $g:K \to \Omega X$ be a pointed map from a finite CW complex. Then $(\Omega f )\circ g$ is adjoint to $f \circ g^{ad}$. Because $f$ is a phantom, $f \circ g^{ad}$ is nullhomotopic, and nullhomotopic as a pointed map since its target is an H-space. Therefore $(\Omega f )\circ g$ is nullhomotopic, as desired.
\end{proof}

\begin{lem}\label{preservation-of-phantoms}
Let $f:X \to Y$ be a homology equivalence and $h: Y \to \Omega Z$ be a map to a loop space. If $h \circ f$ is a phantom, then so is $h$.
\end{lem}

\begin{proof}
Without loss of generality, we can assume that $X$ and $Y$ are connected CW complexes and that $h$ and $f$ are pointed maps. It is also enough to prove that $h$ becomes nullhomotopic when composed with \emph{pointed} maps with finite CW source. 
Let $F$ be a finite pointed complex and $g:F \to Y$ be a map; we wish to show that $h \circ g$ is nullhomotopic. It is enough to prove that $(h \circ g)^{ad}= h^{ad} \circ (\Sigma g): \Sigma F \to \Sigma Y \to Z$ is nullhomotopic. Now as $f$ was assumed to be a homology equivalence of connected CW complexes, $\Sigma f : \Sigma X \to \Sigma Y$ is a homotopy equivalence. Thus there exists a map $m: \Sigma F \to \Sigma X$ with $(\Sigma f) \circ m \simeq \Sigma g$.
Hence
$$(h \circ g)^{ad}= h^{ad} \circ (\Sigma g) \simeq h^{ad} \circ (\Sigma f) \circ m = (h \circ f)^{ad} \circ m.$$

But by assumption, $h \circ f$ is a phantom and thus by Lemma \ref{preservation-of-phantoms0}, $(h \circ f)^{ad}$ is a phantom as well. As $\Sigma F$ is finite, we find that $(h \circ g)^{ad}$ and thus $h \circ g$ is nullhomotopic, as claimed.
\end{proof}

\begin{proof}[Proof of Proposition \ref{prop:phantom}]
The first statement implies the second, by Lemma \ref{preservation-of-phantoms0}. We claim that the two compositions 
\[
\bas(\bm{\beta}_\infty^+) \circ \Psi, \hat{a} \circ \Psi: \gB_{\infty} \lra \gX \lra \loopinf{+2n}\KO
\]
agree up to phantoms. Since $\Psi$ is a homology equivalence, this will imply---by Lemma \ref{preservation-of-phantoms}---that $\bas(\bm{\beta}_\infty^+) $ and $\hat{a}$ agree up to phantoms, as desired. Let $f:K \to \gB_{\infty}$ be a map from a finite CW complex. By compactness, it lands in some $\gB_{k}$. To finish the proof, it is thus enough to check that the compositions 
\[
\gB_k \lra \gX \stackrel{\bas(\bm{\beta}_\infty^+)}{\lra} \loopinf{+2n} \KO \; \text{ and } \; \gB_k \lra \gX \overset{\hat{a}}\lra \loopinf{+2n} \KO
\]
agree up to phantoms. But by Assumption \ref{key-properties} (\ref{it:properties:2}), the composition $\gB_k \to \gX \overset{\hat{a}}\to \loopinf{+2n} \KO$ is homotopic to (the pullback to $\gB_k$ of) the family index $\ind{E_k,h_\round^{2n-1}} \in KO^{-2n} (B_k)$, up to phantoms.
Moreover, by Proposition \ref{result-of-obstruction-argument}, Lemma \ref{lem:RelCycle} and Proposition \ref{relative-index-class-properties}, the composition $\gB_k \to \gX \stackrel{\bas(\bm{\beta}_\infty^+)}{\to} \loopinf{+2n} \KO$ is homotopic to (the pullback to $\gB_k$ of) the family index $\ind{E_k,h_\round^{2n-1}} \in KO^{-2n} (B_k)$ as well. 
\end{proof}

\subsection{Starting the proof of Theorem \ref{factorization-theorem:even-case}}\label{sec:Starting}

Consider a sequence of spin cobordisms
\begin{equation}\label{eq:seqSpinCob}
\emptyset \stackrel{W}{\rightsquigarrow} K\vert_0 \stackrel{K\vert_{[0,1]}}{\rightsquigarrow} K\vert_1  \stackrel{K\vert_{[1,2]}}{\rightsquigarrow} K\vert_2  \stackrel{K\vert_{[2,3]}}{\rightsquigarrow} K\vert_3 \rightsquigarrow \cdots
\end{equation}
as in the previous section. The associated fibre bundles $\pi_k : E_k \to B_k$ admit unique spin structures---which are compatible---by Lemma \ref{lem:existence-spin-structure}. Hence we obtain a map
$$\underset{k \to \infty}\hocolim \, B_k \lra \underset{k \to \infty}\hocolim \, \Omega_{[\emptyset, K\vert_k]} \Omega^{\infty-1} \MTSpin(2n)$$
on homotopy colimits. Each of the maps
$$\Omega_{[\emptyset, K\vert_k]} \Omega^{\infty-1} \MTSpin(2n) \lra \Omega_{[\emptyset, K\vert_{k+1}]} \Omega^{\infty-1} \MTSpin(2n),$$
which concatenates a path with the path obtained from the Pontrjagin--Thom construction applied to $K\vert_{[k,k+1]}$, is a homotopy equivalence, and as $K\vert_{-1}=\emptyset$  we obtain a map
\begin{equation}\label{eq:alphainfty}
\alpha_\infty : B_\infty := \underset{k \to \infty}\hocolim \, B_k \lra \loopinf{}_0 \MTSpin (2n)
\end{equation}
well-defined up to homotopy. 

\emph{Suppose for now that the map $\alpha_\infty$ is acyclic.} Then it satisfies Assumption \ref{key-properties} (\ref{it:properties:1}), and for Assumption \ref{key-properties} (\ref{it:properties:2}) we take the class
$$\Omega^{\infty}(\lambda_{-2n}) \in KO^{-2n}(\Omega^\infty_0 \MTSpin(2n))$$
represented by the infinite loop map of the $KO$-theory Thom class of $\MTSpin(2n)$.

\begin{prop}
The maps $\ind{E_k, h_\round^{2n-1}} : B_k \to \Omega^{\infty+2n} \KO$ and $B_k \to B_{\infty} \stackrel{\alpha_{\infty}}{\to} \Omega^{\infty}_0 \MTSpin (2n) \stackrel{\Omega^\infty(\lambda_{-2n})}{\to} \Omega^{\infty+d} \KO$ agree up to phantom maps.
\end{prop}

\begin{proof}
Let $C$ be compact and $C \to B_k$ be a map which classifies a spin fibre bundle $E \to C$ with fibre $W_k$ and trivialised boundary. On the boundary bundle $C \times K\vert_{k}$, there is the psc metric $h_\round^{2n-1}$. We give the trivial fibre bundle $C \times W_{k}^{op}$ the fibrewise psc metric $m_k$ constructed in Proposition \ref{abelianness-lemma} \ref{it:abelianness-lemma:2}. We let $E':=E_k \cup_{C \times K\vert_k } (C \times W_{k}^{op})$, which is the bundle $E_k$ with a trivial bundle glued to it, and has closed fibres. Since $m_k$ is psc, we find by the additivity theorem (Corollary \ref{additivity-concordance-version}) that
$$\ind{E_k,h_\round^{2n-1}} = \ind{E'}.$$
By construction, the map $C \to B_k \to B_{\infty} \stackrel{\alpha_{\infty}}{\to} \Omega^{\infty}_0 \MTSpin (2n)$ is the map $\alpha_{E'}: C \to \Omega^{\infty}_0 \MTSpin (2n)$. By the Atiyah--Singer family index theorem (Theorem \ref{cl-index-3}), these two maps out of $C$ are homotopic. 
\end{proof}

We have verified the hypotheses of Theorem \ref{factorization-theorem}, which thus provides a map
$$\rho: \Omega^{\infty+1} \MTSpin(2n) \lra \cR^+(W_0)_{h_\round^{2n-1}}$$
such that the composition
$$\Omega^{\infty+1} \MTSpin(2n) \overset{\rho}\lra \cR^+(W_0)_{h_\round^{2n-1}} \overset{\inddiff_{h_{-1}}}\lra \Omega^{\infty+2n+1} \KO$$
is homotopic to $\Omega^\infty(\lambda_{-2n})$ up to phantom maps. This establishes Theorem \ref{factorization-theorem:even-case} for the manifold $W=W_0$, and the general case then follows from Proposition \ref{inheritance-of-detection}.


Thus in order to finish the proof of Theorem \ref{factorization-theorem:even-case} in dimension $2n$ we must produce a spin cobordism $W : \emptyset \leadsto S^{2n-1}$ such that the following three conditions are satisfied
\begin{enumerate}[(i)]
\item $W$ is 1-connected, 

\item $W$ is spin cobordant to $D^{2n}$ relative to its boundary,

\item the associated map $\alpha_\infty$ is acyclic.
\end{enumerate}

The main result of \cite{GRW} gives criteria on $W$ for the map $\alpha_\infty$ to be a homology isomorphism, as long as $2n \geq 6$. This result is enough to prove Theorem \ref{factorization-theorem:even-case} in the case $2n=6$, and we explain it first.

\subsubsection{Finishing the proof of Theorem \ref{factorization-theorem:even-case}: the 6-dimensional case}\label{finishingproof:dim6}

In this case we let $W=D^{6}$, which clearly satisfies the first two conditions. The manifold $W_k$ is then the $k$-fold connected sum $\#^k (S^3 \times S^3)$, minus a disc. To establish the acyclicity of the map $\alpha_\infty$ in this case, we use the following theorem.

\begin{thm}[Galatius--Randal-Williams]\label{grw-theorem1}
The map 
\begin{equation}\label{eq:Acyclic}
\alpha_\infty : \underset{k \to \infty}\hocolim \, B \Diff_\partial (\#^k (S^3 \times S^3)\setminus D^6) \lra  \loopinf{}_0 \MTSpin(6)
\end{equation}
is a homology equivalence.
\end{thm}

This is a consequence of the general theorem \cite[Theorem 1.8]{GRW}, and is explicitly discussed in \cite[\S 1.5]{GRW}. There is one minor observation required: in \cite{GRW} the notation $\Diff_\partial(W)$ denotes the colimit of the groups of diffeomorphisms fixing smaller and smaller collars, whereas in this paper it denotes the group of diffeomorphisms fixing a single collar. But the natural homomorphism between these groups is a weak homotopy equivalence by an argument like that of Lemma \ref{collar-stretching}, and we will not distinguish between them.

Having a homology equivalence is not quite enough for the obstruction theory in the previous subsection, where we needed the map to be acyclic. However, we also have the following.

\begin{prop}
The space $\loopinf{}_0 \MTSpin (6)$ is simply-connected.
\end{prop}

This follows from recent calculations of Galatius and the third named author \cite[Lemma 5.7]{GRWAb}. We have been informed by 
B\"okstedt that his calculations with Dupont and Svane \cite{BDS} can be used to give an alternative proof.
It follows from this proposition that the map \eqref{eq:Acyclic} is actually acyclic (as the target is simply-connected, so there are no local coefficient systems to check), which finishes the proof of Theorem \ref{factorization-theorem:even-case} for $d=6$. 

\subsection{Finishing the proof of Theorem \ref{factorization-theorem:even-case}}\label{finishingproof:alldim}


There are three ways in which the case $2n > 6$ is more difficult to handle than the case $2n=6$. Firstly, the infinite loop space $\Omega^\infty_0 \MTSpin (2n)$ is not necessarily simply-connected, so it is not automatic that the homology equivalences coming from \cite{GRW} are acyclic. Secondly, the manifold obtained by the countable composition of the cobordisms $K = ([0,1] \times S^{2n-1}) \# (S^n \times S^n)$ does not form a ``universal spin-end" in the sense of \cite[Definition 1.7]{GRW} unless $2n = 6$, and so the results of \cite{GRW} do not apply. Thirdly, even if the results of \cite{GRW} did apply to this stabilisation, we must show that there is a spin cobordism $W : \emptyset \leadsto S^{2n-1}$ whose structure map $\ell_W : W \to B\Spin(2n)$ is $n$-connected and which in addition satisfies the conditions given in Section \ref{sec:Starting}.

The first two of these difficulties can be avoided by appealing instead to the results of \cite{GRWHomStab2}, which build on those of \cite{GRW}. These results upgrade those of \cite{GRW} to always give acyclic maps instead of merely homology equivalences, and to allow more general stabilisations than by ``universal ends". The third difficulty must be confronted directly, and we shall do so shortly. First, let us state the version of the result of \cite{GRWHomStab2} which we shall use, and show how to extract it from \cite{GRWHomStab2}.

\begin{thm}[Galatius--Randal-Williams]
Let $W : \emptyset \leadsto S^{2n-1}$ be a spin cobordism such that the structure map $\ell_W : W \to B\Spin(2n)$ is $n$-connected. Then the map
$$\alpha_\infty : \underset{k \to \infty} \hocolim \, B\Diff_\partial(W_k) \lra \Omega_0^\infty \MTSpin(2n)$$
is acyclic.
\end{thm}
\begin{proof}
We adopt the notation used in the introduction of \cite{GRWHomStab2}. Let $\theta : B\Spin(2n) \to BO(2n)$ be the map classifying the universal spin bundle of dimension $2n$. Let us write $\hat{\ell}_W : TW \to \theta^*\gamma_{2n}$ for the bundle map covering $\ell_W$, write $P := S^{2n-1}$, and let $\hat{\ell}_P := \hat{\ell}_W\vert_P$. We have the cobordism $K = ([0,1] \times P) \# (S^n \times S^n)$ and we wish to choose a bundle map $\hat{\ell}_K : TK \to \theta^*\gamma_{2n}$ such that 
\begin{enumerate}[(i)]
\item\label{it:jhlg:1} $\hat{\ell}_K \vert_{\{0\} \times P} = \hat{\ell}_K \vert_{\{1\} \times P} = \hat{\ell}_P$, and

\item $\hat{\ell}_K$ restricted to $W_{1,1} \subset K$ is admissible in the sense of \cite[Definition 1.2]{GRWHomStab2}.
\end{enumerate}
As the map $\theta$ is 2-co-connected, and we have assumed that $n \geq 3$, the manifold $W_{1,1}$ has a unique $\theta$-structure. As admissible $\theta$-structures always exist by \cite[\S 2.2]{GRWHomStab2}, it must be admissible, so it is enough to find a $\hat{\ell}_K$ satisfying (\ref{it:jhlg:1}). But as $K$ may be obtained from $\{0,1\} \times P$ by attaching a single 1-cell followed by $n$-cells and higher, the only obstruction to finding a bundle map $\hat{\ell}_K$ extending given bundle maps over $\{0\} \times P$ and $\{1\} \times P$ is whether these are coherently oriented. Thus there is such a $\hat{\ell}_K$.

The space $\Bun^\theta_{\partial, n}(W_k ; \hat{\ell}_{P})$, of those $\theta$-structures on $W_k$ which restrict to $\hat{\ell}_P$ on $\partial W_k = K\vert_k = P$ and have $n$-connected underlying map, is homotopy equivalent to the space of ($n$-connected) relative lifts
\begin{equation*}
\xymatrix{
P \ar[r]^-{\ell_P} \ar@{^(->}[d] & B \ar[d]^-{\theta}\\
W_k \ar[r]^-\tau \ar@{-->}[ru] & BO(2n)
}
\end{equation*}
of a Gauss map $\tau$ of $W_k$. The obstructions to trivialising an element of the $p$th homotopy group of this space therefore lie in the groups
$$H^{i}(D^{p+1} \times W_k, \partial(D^{p+1} \times W_k) ; \pi_i(BO(2n), B)),\quad i \geq 1.$$
As the map $\theta$ is 2-co-connected these obstruction groups vanish for $i \geq 3$. As the pair $(W_k, P)$ is 1-connected, the pair $(D^{p+1} \times W_k, \partial(D^{p+1} \times W_k))$ is $(p+2)$-connected, and so these obstruction groups also vanish for $i \leq p+2$ so in particular for $i \leq 2$. Therefore the obstruction groups vanish for all $p$, so $\pi_p(\Bun^\theta_{\partial, n}(W_k ; \hat{\ell}_{P}))=0$ for all $p$. 
Thus the natural map gives a weak homotopy equivalence
\begin{equation*}
\Bun^\theta_{\partial, n}(W_k ; \hat{\ell}_{P}) \hcoker \Diff_\partial(W_k) \overset{\sim}\lra B\Diff_\partial(W_k).
\end{equation*}

On the other hand, $\Bun^\theta_{\partial, n}(W_k ; \hat{\ell}_{P}) \hcoker \Diff_\partial(W_k)$ is a path component of the space $\mathscr{N}^\theta_n(P, \hat{\ell}_P)$ of \cite[Definition 1.1]{GRWHomStab2}, and by \cite[Theorem 1.5]{GRWHomStab2} there is an acyclic map
$$\underset{k \to \infty} \hocolim \, \mathscr{N}^\theta_n(P, \hat{\ell}_P) \lra \Omega^\infty \MTSpin(2n),$$
where the homotopy colimit is formed using the endomorphism $- \cup (K, \hat{\ell}_K)$ of $\mathscr{N}^\theta_n(P, \hat{\ell}_P)$. Restricting to the path component 
$$\underset{k \to \infty} \hocolim \, \Bun^\theta_{\partial, n}(W_k ; \hat{\ell}_{P}) \hcoker \Diff(W_k, K\vert_k) \overset{\sim} \lra \underset{k \to \infty} \hocolim \, \Diff(W_k, K\vert_k)$$
this map identifies with $\alpha_\infty$.
\end{proof}

It remains to provide a spin cobordism $W$ which satisfies all conditions we needed so far. We equip $D^{2n}$ with the unique (up to isomorphism) spin structure and $S^{2n-1}$ with that induced on the boundary.
\begin{prop}
For $2n \geq 6$ there exists a spin cobordism $W: \emptyset \leadsto S^{2n-1}$ such that
\begin{enumerate}[(i)]
\item $W$ is spin cobordant relative to its boundary to $D^{2n}$,
\item the structure map $\ell_W  :W \to B \Spin (2n)$ is $n$-connected (as $B \Spin (2n)$ is simply-connected, it follows that $W$ is simply-connected).
\end{enumerate}
\end{prop}

\begin{proof}
This works for more general $\theta$-structures, but we confine ourselves to the case we need. The space $B \Spin (2n)$ is simply-connected and has finitely-generated homology in each degree, so there exists a finite complex $X$ and an $n$-connected map $f:X \to B \Spin (2n)$. Let $f^* \gamma_{2n} \to X$ be the pullback of the universal spin vector bundle. The structure map $D^{2n}\to B \Spin (2n)$ factors as
$$ D^{2n}\stackrel{h}{\lra} X \stackrel{f}{\lra} B \Spin (2n),$$
which yields an isomorphism $TD^{2n} \cong h^{*} f^* \gamma_{2n}$. 
By surgery below the middle dimension in the interior of $D^{2n}$ (compare \cite[Theorem 1.2]{Wall} which also applies to the case with boundary), we may find a cobordism relative boundary $Y$ from $D^{2n}$ to a manifold $W$, and a map $g: Y \to X$ extending $h$ such that
\begin{enumerate}[(i)]
\item the map $g|_W: W \to X$ is $n$-connected, 
\item there is a stable isomorphism $TY \cong g^* f^* \gamma_{2n} \oplus \bR$ extending the  isomorphism $TD^{2n} \cong h^* f^* \gamma_{2n}$.
\end{enumerate}
It follows from (ii) that $Y$ is a spin cobordism and so is $W$, as well. The stable tangent bundle of $W$ is classified by the map $\varphi:W \stackrel{g|_W}{\to} X \stackrel{f}{\to} B \Spin (2n) \stackrel{j}{\to} B \Spin$ and since $g|_W$ and $f$ are $n$-connected and $j:B \Spin (2n) \to B \Spin$ is $2n$-connected, it follows that $\varphi$ is $n$-connected.
There is a homotopy commutative diagram
\[
 \xymatrix{
 W \ar[rr]^{\varphi} \ar[dr]^{\ell_W} & & B \Spin\\
 & B \Spin (2n) \ar[ur]^{j}
  }
\]
from which the $n$-connectivity of $\ell_W$ follows.
\end{proof}

\section{Computational results}\label{sec:computation}

In this section, we will derive the computational consequences of Theorems \ref{factorization-theorem:even-case} and \ref{factorization-theorem:odd-case}. To do so, we will study the effect of the maps
\begin{align*}
\Omega^{\infty+1} \kothom{2n}:\Omega^{\infty+1}_0 \MTSpin(2n) &\lra  \loopinf{+2n+1} \KO\\
\Omega^{\infty+2} \kothom{2n}:\Omega^{\infty+2}_0 \MTSpin(2n) &\lra \loopinf{+2n+2} \KO
\end{align*}
on homotopy and homology, and in particular their images; Theorems \ref{factorization-theorem:even-case} and \ref{factorization-theorem:odd-case} show that these maps factor through $\cR^+(S^{2n})$ and $\cR^+(S^{2n+1})$ respectively, so the images of these maps are contained in the the images of the respective secondary index maps. This section is almost entirely homotopy-theoretic, and except for Theorem \ref{splitting-theorem}, we shall not mention spaces of psc metrics any further.

Recall that $\MTSpin (d)$ is the Thom spectrum $\Th (- \gamma_d)$ of the additive inverse of the universal vector bundle $\gamma_d \to B \Spin (d)$. For any virtual spin vector bundle $V \to X$ of rank $r \in \bZ$, we denote the $KO$-theoretic Thom class by $\lambda_{V} \in \KO^{r} (\Th (V))=[\Th (V), \Sigma^{r} \KO]$. Note that there is a unique lift of $\lambda_{V}$ to $\Sigma^{r} \ko$, the appropriate suspension of the connective $KO$-spectrum, which we denote by the same symbol. In the special case $V=-\gamma_d$, we denote the Thom class by $\kothom{d} \in [\MTSpin (d), \Sigma^{-d}(\ko)]$. We are interested in the groups
$$ J_{d,k} := \mathrm{Im}\left((\kothom{d})_*: \pi_k (\MTSpin (d))\to \pi_{d+k}(\ko)\right).$$

\subsection{Multiplicative structure of Madsen--Tillmann--Weiss spectra}

The spectrum $\ko$ has a ring structure, and the algebraic structure of $\pi_* (\ko)$ is well-known, due to Bott periodicity. There are elements $\eta \in \pi_1 (\ko)$, $\kappa\in \pi_4 (\ko)$ and $\beta \in \pi_8 (\ko)$, such that
\begin{equation}\label{homotopy-ko}
 \pi_* (\ko)=\bZ[\eta, \kappa,\beta]/(2 \eta, \eta^3, \kappa^2 - 4 \beta, \kappa \eta).
\end{equation}

Even though $\MTSpin (d)$ is not itself a ring spectrum, there is a useful product structure available to us as the collection $\{\MTSpin (d)\}_{d \geq 0}$ form what one might call a \emph{graded ring spectrum}. Namely, there are maps
$$\mu:\MTSpin (d) \wedge \MTSpin (e) \lra \MTSpin (d+e)$$
which come from the bundle maps $\gamma_d \times \gamma_e \to \gamma_{d+e}$ which cover the Whitney sum maps $B \Spin (d) \times B \Spin (e) \to B \Spin (d+e)$. The usual multiplicative property of Thom classes translates into the statement that the diagram
\begin{equation*}
\xymatrix{
\MTSpin (d) \wedge \MTSpin (e)\ar[d]^{\kothom{d} \wedge \kothom{e}} \ar[r]^-{\mu} & \MTSpin (d+e) \ar[d]^{\kothom{(d+e)}}\\
\Sigma^{-d} \ko \wedge \Sigma^{-e} \ko \ar[r] & \Sigma^{-(d+e)} \ko
}
\end{equation*}
(where the bottom horizontal map is the ring spectrum structure map) commutes up to homotopy. On the level of homotopy groups, this commutativity means that for $a \in \pi_k (\MTSpin (d))$, $b \in \pi_l (\MTSpin (e))$, we have
\begin{equation}\label{product-ahat}
( \kothom{(d+e)})_* (\mu(a,b))= (\kothom{d})_* (a)  \cdot (\kothom{e})_* (b) \in \pi_{d+e+k+l} (\ko).
\end{equation}

In order to write down elements in $\pi_k (\MTSpin (d))$, the interpretation of this homotopy group in terms of Pontrjagin--Thom theory is useful.

\begin{thm}\label{pontrjagin-thom-theorem}
The group $\pi_k (\MTSpin (d))$ is isomorphic to the cobordism group of triples $(M,V,\phi)$, where $M$ is a closed $(k+d)$-manifold, $V \to M$ a spin vector bundle of rank $d$ and $\phi: V \oplus \epsilon^k_{\bR} \cong TM$ a stable isomorphism of vector bundles.
\end{thm}

This is just a special case of the classical Pontrjagin--Thom theorem, see e.g. \cite[Chapter II]{Stong}. There are homomorphisms
\begin{equation}\label{mtspectra-stabmaps}
\pi_{k+1} (\MTSpin (d-1)) \lra \pi_{k} (\MTSpin (d)) \lra \Omega^{\Spin}_{d+k},
\end{equation} 
where the symbol $ \Omega^{\Spin}_{d+k}$ denotes the ordinary spin cobordism group of $(d+k)$-manifolds.
The first homomorphism sends $[M,V,\phi]$ to $[M,V \oplus \bR, \phi]$, and the second forgets $V$ and $\phi$ (but keeps the spin structure on $M$ that is induced by them). The homomorphism $\pi_{k} (\MTSpin (d )) \to \Omega^{\Spin}_{d+k}$ is surjective for $k \leq 0$ and bijective for $k<0$. The image of $\pi_{k} (\MTSpin (d)) \to \Omega^{\Spin}_{d+k}$ (for $k>0$) is the group of all cobordism classes which contain manifolds whose stable tangent bundle splits off a $k$-dimensional trivial summand.
Any $d$-dimensional spin manifold $M$ defines an element $[M,TM,\id] \in \pi_0 (\MTSpin (d))$, but this construction does \emph{not} descend to a homomorphism $\Omega^{\Spin}_{d} \to \pi_0 (\MTSpin (d))$, as a $(d+1)$-dimensional spin cobordism does not generally admit a destabilisation of its tangent bundle to a $d$-dimensional vector bundle compatible with the tangent bundle along its boundary (e.g.\ $D^{d+1}$ as a nullbordism of $S^d$ for $d$ even).

The product has a pleasant description in terms of manifolds: if $[M_i,V_i,\phi_i] \in \pi_{k_i}(\MTSpin (d_i))$, $i=0,1$, then
$$[M_0 \times M_1, V_0 \times V_1, \phi_0 \times \phi_1] = \mu([M_0,V_0,\phi_0], [M_1,V_1,\phi_1]) \in \pi_{k_0+k_1}(\MTSpin (d_0+d_1)).$$
It is a consequence of the Atiyah--Singer index theorem that
$$( \kothom{d})_* ([M,V,\phi]) = \ind{\Dir_M } \in KO^{-d-k} = \pi_{k+d}(\ko)$$
for $[M,V,\phi] \in \pi_k (\MTSpin (d))$. 
From now on, we will denote this invariant by the classical notation $\hat{\mathscr{A}}(M)$. For $k+d \equiv 0 \pmod 4$, the value of $\hat{\mathscr{A}}(M)$ can be computed in terms of characteristic classes by the formula
\begin{equation}\label{ahat-genus}
\hat{\mathscr{A}}(M)  = 
\begin{cases}
\langle \hat{A}(TM), [M] \rangle  \cdot \beta^r & \text{ if $d+k=8r$},\\
\frac{1}{2}\langle \hat{A}(TM), [M] \rangle \cdot \beta^r \kappa & \text{ if $d+k=8r+4$}.
\end{cases}
\end{equation}

For each $d \geq 0$ there is a class
$$e_d := [*,\bR^d, \id] \in \pi_{-d} (\MTSpin (d)),$$
which is a generator for the group $\pi_{-d} (\MTSpin (d))\cong \bZ$. These classes clearly satisfy $\mu(e_d , e_e) = e_{d+e}$ and $\hat{\mathscr{A}} (e_d)=1$. Moreover, $e_0$ is a unit for the multiplication $\mu$. Multiplication by $e_1$ defines a map 
\begin{equation}\label{suspension-mt-spectra}
\eta_d:\mathrm{S}^{-1} \wedge \MTSpin (d) \lra \MTSpin (1) \wedge \MTSpin (d) \lra \MTSpin (d+1),
\end{equation}
which coincides with the analogous map in \cite[\S 3]{GMTW} and which on homotopy groups induces the first map in (\ref{mtspectra-stabmaps}). The composition
$$\mathrm{S}^{-d} \stackrel{e_d}{\lra} \MTSpin (d) \stackrel{\kothom{d}}{\lra} \Sigma^{-d} \ko$$
is the $d$th desuspension of the unit map of the ring spectrum $\ko$. To sum up, we obtain a homotopy commutative diagram:

\begin{equation}\label{sequence-of-mt-spectra}
\begin{gathered}
\xymatrix{
\mathrm{S}^0 \ar[rr]^{e_0} \ar[drr]^-{e_1} \ar[ddrr]^-{e_2} &  & \MTSpin (0) \ar[d]^{\eta_1}  \ar[rr]^-{\kothom{0}}& & \ko\\
& & \Sigma\MTSpin (1) \ar[d]^{\eta_2}  \ar[urr]^-{\kothom{1}} & &\\
& &\Sigma^2  \MTSpin (2) \ar[d]^{\eta_3}  \ar[uurr]_-{\kothom{2}} & &\\
& & \vdots &  &
}
\end{gathered}
\end{equation}
From \eqref{product-ahat} and \eqref{sequence-of-mt-spectra}, we obtain

\begin{cor}\label{containment-image-thomclass}
There are inclusions $J_{d,k} \supseteq J_{d-1, k+1}$ and $J_{d,k} J_{e,l} \subseteq J_{d+e,k+l}$.
\end{cor}

\subsection{Proof of Theorem \ref{thm:intro:htpyGps}}

In this section we shall provide the homotopy theoretic calculations which, when combined with Theorem \ref{factorization-theorem:even-case} and \ref{factorization-theorem:odd-case}, establish Theorem \ref{thm:intro:htpyGps}. We first investigate the effect of the maps $\kothom{d}$ on rational homotopy groups.

\begin{thm}\label{rational-surjectivity-htpy-thy}
For each $d\geq 2$ and $d+k \equiv 0 \pmod 4$, the map
$$({\kothom{d}})_*:\pi_{k} (\MTSpin (d)) \otimes \bQ \lra \pi_{k+d} (\ko)\otimes \bQ \cong \bQ$$
is surjective. 
\end{thm}
\begin{proof}
The proof is a standard calculation with characteristic classes, but we present the details as they will be used later on. 

Let $\pi:V \to X$ be a complex vector bundle of rank $n$ whose underlying real bundle has a spin structure. The spin structure determines a Thom class $\lambda_V \in KO^{2n}(\mathrm{Th}(V))$, which for this proof we shall write as $\lambda_V^\Spin$. On the other hand, the complex structure determines a Thom class $\lambda_V^\bC \in K^{2n}(\mathrm{Th}(V))$. The groups $SO(2)$, $\Spin(2)$, and $U(1)$ are all isomorphic, but $\Spin(2) \to SO(2)$ is a double cover. Identifying all these groups with $U(1)$, it follows that a spin structure on a complex line bundle is precisely a complex square root. In particular, the spin structure on $V$ determines a square root $\det (V)^{1/2}$ of the complex determinant line bundle of $V$. The relation between the Thom classes $\lambda_V^\Spin$ and $\lambda_V^\bC$ under the complexification map $c: \KO \to \KU$ is given by the following formula\footnote{It is important here to adopt the correct convention for $K$-theory Thom classes of complex vector bundles: one should take 
the convention used in \cite[Theorem C.8]{SpinGeometry}, which is characterised by the identity $(\lambda_L^{\bC})^2 = (1-L) \cdot \lambda_L^{\bC} \in K^0(\mathrm{Th}(L))$ when $L \to \bC\bP^\infty$ is the universal line bundle.}, cf.\ \cite[(D.16)]{SpinGeometry}:
$$c(\lambda^{\Spin}_{V}) =  \det (V)^{-1/2} \cdot \lambda^{\bC}_{V}  \in K^{2n}(\mathrm{Th}(V)).$$
If $V \oplus V^{\bot} \cong \epsilon^n_\bC$, we obtain, using that $\det (V)\otimes  \det (W) = \det (V \oplus W)$, the formula
$$c(\lambda^{\Spin}_{V^{\bot}}) =  \det (V)^{1/2} \cdot \lambda_{V^{\bot}}^{\bC} \in K^{2n}(\mathrm{Th}(V^\perp)).$$
This relation is preserved under stabilisation, and therefore we get an equation in the $K$-theory of the Thom spectrum
$$c(\lambda^{\Spin}_{-V}) = \det (V)^{1/2} \cdot \lambda_{-V}^{\bC} \in K^{-2n}(\Th (-V)).$$

Similarly, if $u_V \in H^{2n}(\mathrm{Th}(V);\bQ)$ is the cohomological Thom class, then
$$\ch(\lambda^\bC_V) = I(V)\cdot u_V \in \prod_i H^{i+2n}(\mathrm{Th}(V);\bQ)$$
where $I(-)$ is the genus associated to $\frac{1-e^x}{x}$ cf.\ \cite[p.\ 241]{SpinGeometry}. This is also multiplicative and stable, hence gives an equation in spectrum cohomology
$$\ch(\lambda^\bC_{-V}) = \tfrac{1}{I(V)}\cdot u_{-V} \in \prod_i H^{i-2n}(\mathrm{Th}(-V);\bQ).$$

After these generalities let us begin the proof of the theorem. By Corollary \ref{containment-image-thomclass}, it is enough to consider the case $d=2$, and as $\Spin(2)$ can be identified with $U(1)$ we may identify $B\Spin(2)$ with $\bC\bP^\infty$. Under this identification the universal rank 2 spin bundle is identified with $L^{\otimes 2}$, the (realification of the) tensor square of the universal complex line bundle over $\bC\bP^\infty$. Specialising the above general theory to this case, we obtain
\begin{equation}\label{spin-thom-vs-complex-thom}
c(\lambda_{-L^{\otimes 2}}^{\Spin}) = L \cdot \lambda_{-L^{\otimes 2}}^{\bC} \in K^{-2} (\MTSpin (2)). 
\end{equation}
and
\begin{equation*}
\ch(\lambda_{-L^{\otimes 2}}^\bC) = \frac{c_1(L^{\otimes 2})}{1-e^{c_1(L^{\otimes 2})}} \cdot u_{-L^{\otimes 2}}.
\end{equation*}

If we define $x := c_1(L)$, which generates $H^2(\bC\bP^\infty;\bZ)$, then we obtain
$$\ch(c(\lambda_{-L^{\otimes 2}}^{\Spin}))= \frac{2x}{1-e^{2x}} e^x \cdot u_{-L^{\otimes 2}} = -\frac{x}{\mathrm{sinh}(x)} \cdot u_{-L^{\otimes 2}} \in H^* (\MTSpin (2); \bQ),$$
and following \cite[Appendix B]{MS} the identity $\tfrac{1}{\mathrm{sinh}(2v)} = \tfrac{1}{\mathrm{tanh}(v)} - \tfrac{1}{\mathrm{tanh}(2v)}$ yields
$$\frac{x}{\mathrm{sinh}(x)} = 1 + \sum_{m=1}^\infty (-1)^m  \frac{(2^{2m}-2) B_m}{(2m)!} x^{2m},$$
where $B_m$ is the $m$th Bernoulli number (our notation for Bernoulli numbers also follows \cite[Appendix B]{MS}). We hence obtain the formula
\begin{equation}\label{sinh-series}
\ch(c(\lambda_{-L^{\otimes 2}}^{\Spin})) = -\left(1+ \sum_{m=1}^\infty (-1)^m  \frac{(2^{2m}-2) B_m}{(2m)!} x^{2m} \right)\cdot u_{-L^{\otimes 2}}.
\end{equation}
The point of the proof is now that $\tfrac{(2^{2m}-2) B_m}{(2m)!} \neq 0$ for all $m>0$. 
More precisely, consider the diagram
\begin{equation*}
 \xymatrix{
\pi_{4m-2}(\MTSpin (2))\otimes \bQ \ar[r]^-{(\kothom{2})_*}  \ar[d]^{h}& \pi_{4m}(\ko)\otimes \bQ \ar[d]^{h}\\ 
H_{4m-2} (\MTSpin (2); \bQ) \ar[d]^{\simeq} \ar[r]^-{(\kothom{2})_*}  & H_{4m} (\ko;\bQ) \ar[d]^{\langle \ph_{4m}, - \rangle}\\ 
 H_{4m}(B \Spin (2); \bQ) \ar[r]^-{e} & \bQ.
 }
\end{equation*}
The upper vertical maps are the Hurewicz homomorphisms, which are isomorphisms by Serre's finiteness theorem; the left bottom vertical arrow is the inverse Thom isomorphism, and the right bottom vertical map is the evaluation against the $m$th component of the Pontrjagin character, which is also an isomorphism. The upper square is commutative, and the above calculation shows that the lower square commutes if $e$ is the evaluation against $(-1)^{m+1}  \frac{(2^{2m}-2) B_m}{(2m)!} x^{2m} \neq 0 \in H^{4m}(B \Spin (2); \bQ)$. As this class is nonzero, the lower horizontal map is onto, and so is the upper horizontal map, as claimed.
\end{proof}

We now describe the effect of the maps $\kothom{d}$ onto the $\bZ/2$ summands.

\begin{thm}\label{mod2-surjectivity-htpy-thy}
For each $d\geq 0$ and $0 \leq i \equiv 1,2 \pmod 8$, the map 
$$({\kothom{d}})_*:\pi_i (\Sigma^d \MTSpin (d))  \lra \pi_i (\ko)\cong \bZ/2$$
is surjective.
\end{thm}

\begin{proof}
That the unit map $S^0 \to \ko$ hits all $2$-torsion follows from the work of Adams on the $J$-homomorphism \cite[Theorem 1.2]{Adams}. But the $d$-fold desuspension of the unit map is $\kothom{d}\circ e_d$, by the remarks before diagram \eqref{sequence-of-mt-spectra}. 
\end{proof}

\subsection{Integral surjectivity}

The following implies Theorem \ref{thm:intro:integral}.

\begin{thm}\label{integral-surjectivity}
The homomorphism $(\kothom{d})_* :\pi_k (\MTSpin (d)) \to \pi_{d+k}(\ko)$ is surjective for all $k \leq d+1$.
\end{thm}

\begin{proof}
If $k+d \not\equiv 0 \pmod 4$, then the map is surjective by Theorem \ref{mod2-surjectivity-htpy-thy}. In particular, if $k=d+1$ then $k+d$ is odd so the claim follows in this case. For the case $k \leq d$, using the multiplicative structure \eqref{containment-image-thomclass} and \eqref{homotopy-ko} it will be enough to create elements $\fk \in \pi_2 (\MTSpin (2))$ with $\ahat (\fk)=\kappa$ and $\fb \in \pi_4 (\MTSpin (4)) $ with $\ahat (\fb)=\beta$. For both cases, we use Theorem \ref{pontrjagin-thom-theorem}, and both elements will be given by a $(2m-1)$-connected $4m$-manifold, $m = 1,2$,  
with the desired value of $\ahat(M)$, plus a spin vector bundle $V \to M$ of rank $2m$, and a stable isomorphism $V \oplus \epsilon^{2m}_\bR \cong TM$.
We will write $\mu_M \in H^{4m}(M)$ for the generator with $\langle \mu_M, [M] \rangle =1$. Recall the formulae for the low-dimensional $\hat{A}$-classes and the Hirzebruch classes
\begin{equation}\label{low-dim-ahat-l-class}
\begin{aligned}
\hat{A}_1 &= - \frac{1}{2^3 \cdot 3} p_1 \quad\quad &\hat{A}_2 &= \frac{1}{2^7 \cdot 3^2 \cdot 5} (-4 p_2 + 7 p_{1}^{2})\\
L_1 &= \frac{1}{3}p_1  & L_2 &= \frac{1}{3^2 \cdot 5}(7p_2-p_{1}^{2}). 
\end{aligned}
\end{equation}

We first construct the element $\fk \in \pi_2 (\MTSpin (2))$. Let $K$ be a $K3$ surface, which is a simply-connected spin manifold. It is well-known that the intersection form of $K$ is $q=2 (-E_8) \oplus 3 H$, the direct sum of two times the negative $E_8$-form and three hyperbolic summands. The signature of this form is $-16$, which by Hirzebruch's signature theorem means that $p_1 (TK)=  -48 \mu_K$. Therefore, $\hat{A} (TK)=2 \mu_K$ and so $\ahat (K)=\kappa$ by \eqref{ahat-genus}, as required. 
We claim that there exists a complex line bundle $L \to K$ such that $p_1 (L^{\otimes 2}) = p_1 (TK)$. Let $c \in H^2 (K)$ and $L_c$ be the line bundle with $c_1 (L_a)=c$. Since 
$$p_1 (L_{c}^{\otimes 2}) = c_1 (L_{c}^{\otimes 2})^2 = 4 q(c) \cdot \mu_K,$$
we have to pick $c$ such that $q(c)=-12$. It is easy to see that a quadratic form which contains a hyperbolic summand represents any even number, and therefore such a $c$ exists. We now claim that $TK$ and $L_c^{\otimes 2} \oplus \epsilon^2_\bR$ are stably isomorphic. To see this, we must show that the triangle in the following diagram commutes.
\begin{equation*}
\xymatrix{
 & B\Spin(2) \ar[d] \\
K \ar[r]^-{TK} \ar[ru]^{L_c^{\otimes 2}}& B\Spin \ar[r]^{p_1/2} & K(\bZ,4)
}
\end{equation*}
As the map $p_1/2 : B\Spin \lra K(\bZ,4)$ is $8$-connected, it is enough to show that $p_1(TK) = p_1(L_c^{\otimes 2} \oplus \epsilon^2) \in H^4(K;\bZ)$, but we have arranged for this to be true. Hence there is a stable isomorphism $\phi:L_{c}^{\otimes 2} \oplus \epsilon^2 \cong TK$ and the element $\fk=[K,L_c,\phi] \in \pi_2 (\MTSpin (2))$ has the desired properties.

In the $8$-dimensional case the proof is similar but more difficult (with the exception that every vector bundle on a $3$-connected $8$-manifold is spin, so we do not have to take care of this condition). Let $P$ be the $8$-dimensional $(-E_8)$-plumbing manifold, which is 3-connected as the Dynkin diagram for $E_8$ is contractible. Consider $\natural^{28} P$, the boundary connected sum of $28$ copies of $P$. This is a 3-connected 
parallelisable manifold, with signature $-2^5 \cdot 7$, and its boundary $\partial (\natural^{28} P) $ is diffeomorphic to $S^7$, by the calculation of Kervaire--Milnor \cite{KervMil} that the group of homotopy 7-spheres is a cyclic group of order 28. The manifold $M := \natural^{28} N \cup_{S^7} D^8$ is parallelisable away from a point, and therefore $p_1 (TM)=0$. The Hirzebruch signature theorem and \eqref{low-dim-ahat-l-class} shows that 
$$p_2 (TM)=-2^5 \cdot 3^2 \cdot 5  \mu_M$$
and $\hat{A}(TM) = \mu_M$, so that $\ahat (M)=\beta$, by \eqref{ahat-genus}. 
For $r \geq 0$ let $K_r:= \sharp^r (S^4 \times S^4)$ and $M_r := M \sharp K_r$. This is still parallelisable away from a point, whence $p_1 (TM_r)=0$, and since $M_r$ is cobordant to $M$, we still have $\ahat(M_r)=\beta$. We now claim that for $r=12$, we can find a $4$-dimensional vector bundle $V_r \to K_r$ so that the connected sum of $V_r$ with the trivial bundle over $M$ yields a vector bundle over $M_r$ with the same Pontrjagin classes as $TM_r$. 
Consider the exact sequence
$$KO^{-1} (M^{(4)}_r) \stackrel{\delta}{\to} \widetilde{KO}^{0} (S^{8}) \to KO^0 (M_r) \to KO^{0}(M^{(4)}_r) \to \widetilde{KO}^1 (S^{8}) =0$$
coming from the cofibre sequence $M_r^{(4)} \to M_r \to S^8$. As $M^{(4)}_r$ is a bouquet of $4$-spheres we have $KO^{-1} (M^{(4)}_r) \cong \bZ/2$, and since $\widetilde{KO}^{0} (S^{8}) \cong \bZ$ it follows that the map $\delta $ is zero. Because $KO^{0}(M^{(4)}_r)$ is free abelian, it follows that $KO^0 (M_r)$ is torsion-free, and this implies that the Pontrjagin character $\mathrm{ph} : KO^0 (M_r) \to H^* (M_r; \bQ)$ is injective. Therefore, a stable vector bundle on $M_r$ is determined by its Pontrjagin classes. Thus if we are able to find a 4-dimensional vector bundle $V_r \to K_r$ such that $\epsilon^4_\bR \sharp V_r \to M \sharp K_r = M_r$ has the same Pontrjagin classes as $TM_r$, then it will be stably isomorphic to $TM_r$.

Isomorphism classes of $4$-dimensional stably trivial vector bundles on $K_r- *\simeq \bigvee^{r} (S^4 \vee S^4)$ are in bijection with $2 H^4 (K_r)\cong(\bZ^2 )^r$, the group of cohomology classes that are divisible by $2$. The bijection is given by the Euler class. For each such vector bundle, there is an obstruction in 
$$ \pi_7 (B \Spin (4)) \cong \pi_6 (S^3 \times S^3) \cong \bZ/12 \oplus \bZ/12$$
against extending the vector bundle over $K_r$ (the isomorphism $\pi_6 (S^3)\cong \bZ/12$ is classical, see \cite[p. 186]{Toda}). To explain this obstruction, write $x_1, y_1, \ldots, x_r, y_r \in \pi_4(\bigvee^{r} (S^4 \vee S^4))$ for the inclusions of the wedge summands, so that the attaching map for the 8-cell of $K_r$ is the sum of Whitehead products $\sum_{i=1}^r [x_i, y_i]$. Thus, if we write $\rho \in \pi_4(B\Spin(4))$ for the class with Euler class 2 and first Pontrjagin class zero, then the stably trivial vector bundle on $K_r - *$ with Euler class $(a_1, b_1, \ldots, a_r, b_r) \in (\bZ^2)^r \cong 2H^4(K_r)$ has obstruction
$$\sum_{i=1}^r [a_i \cdot \rho, b_i \cdot \rho] = \left(\sum_{i=1}^r a_i b_i\right) \cdot [\rho,\rho] \in \pi_7(B\Spin(4))$$
against extending over $K_r$. Therefore, if we take $r=12$, all $a_i$ to be equal and all $b_i$ to be equal, then the obstruction is zero and the vector bundle can be extended. 

Thus, for each $c \in  2 H^4 (S^4 \times S^4)$, we obtain a $4$-dimensional vector bundle $V_c$ on $K_{12} = \sharp^{12} (S^4 \times S^4)$ with Euler class
$$(c, c, \ldots, c) \in H^4 (K_1 ; \bZ)^r \cong H^4 (K_r ; \bZ).$$ 
When restricted to the $4$-skeleton, the bundle $V_c$ is stably trivial and so $p_1 (V_c)=0$. 

Now we take the connected sum of $V_c$ with the trivial vector bundle on $M$ and get a vector bundle $W_{12}\to M_{12} = M \sharp K_{12}$, with $p_1 (W_{12}) =0$ and Euler class
$$e(W_{12}) = (c,c, \ldots, c,0)\in H^4 (K_1 )\oplus \cdots \oplus H^4 (K_1) \oplus H^4 (M)=H^4 (M_{12}).$$
 Let $q_0$ be the intersection form on $H^4 (S^4 \times S^4)$, and compute
$$ p_2 (W_{12}) = e(W_{12})^2 = q (c,c, \ldots, c,0) \mu_{M_{12}}= 12 q_0 (c)\mu_{M_{12}}.$$
In order to achieve that $p_2 (W_{12})=p_2 (TM_{12})=2^5 \cdot 3^2 \cdot 5  \mu_{M_{12}}$, we have to find an even $c$ so that $q_0 (c) = -2^3 \cdot 3 \cdot 5 $. As in the four-dimensional case, we can find an even $c$ with $q_0(c)=8s$, for each $s \in \bZ$, and picking $s=-15$ finishes the proof.
\end{proof}

\subsubsection{Homological conclusions}

We can use Theorem \ref{integral-surjectivity} to obtain results on the image of $H_* (\loopinf{+1} \MTSpin (d); \bF) \to H_* (\loopinf{+d+1} \ko; \bF)$ when $\bF$ is a field. When $\bF = \bQ$ or $\bF_2$, the result is particularly nice. For $\bF=\bQ$ and $d \geq 2$, we find that 
$$H_* (\loopinf{+1} \MTSpin (d); \bQ) \lra H_* (\loopinf{+d+1} \ko; \bQ)$$
is surjective, using Theorem \ref{rational-surjectivity-htpy-thy}. 
For $\bF = \bF_2$, we have the following result, proving Theorem \ref{thm:intro:mod2}.

\begin{prop}\label{surjectivity-mod2-homology}
For $n \geq 0$, the Thom class maps
\begin{align*}
\loopinf{+1} \kothom{2n}:\Omega^{\infty+1} \MTSpin(2n) &\lra \loopinf{+2n+1} \ko\\
\loopinf{+2} \kothom{2n}:\Omega^{\infty+2} \MTSpin(2n) &\lra \loopinf{+2n+2} \ko
\end{align*}
are surjective on $\bF_2$-homology.
\end{prop}
\begin{proof}
We require some information about $H_*(\loopinf{+k}\ko;\bF_2)= H_*(\Omega^{k}(\bZ \times BO);\bF_2)$ as algebras over the Dyer--Lashof algebra. 
When $k \equiv 0,1,2,4 \,\mathrm{mod}\, 8$, so $\Omega^k(\bZ \times BO)$ is disconnected, the class $\xi =[1] \in H_0(\Omega^k(\bZ \times BO);\bF_2)$ of the path component corresponding to a generator of $\pi_k(\bZ \times BO)$, and its inverse $\xi^{-1} = [-1]$, generate $H_*(\loopinf{+k}\ko;\bF_2)$ as an algebra over the Dyer--Lashof algebra. For the remaining $k$ the unique nontrivial class $\xi$ in the lowest nonvanishing reduced homology group generates $H_*(\loopinf{+k}\ko;\bF_2)$ as an algebra over the Dyer--Lashof algebra. These claims follow from the calculations of Kochman \cite{Kochman} and Priddy \cite{Priddy}, a pleasant reference for which is \cite{Nagata}. 

By Theorems \ref{mod2-surjectivity-htpy-thy} and \ref{integral-surjectivity} the class $\xi$ is in the image of 
$$\pi_*(\Omega^{\infty+j} \MTSpin(2n)) \overset{{´\loopinf{+j} \kothom{2n}}_*}\lra \pi_*(\loopinf{+2n+j} \ko) \lra H_*(\loopinf{+2n+j} \ko;\bF_2)$$
for $j=1,2$, 
and if $k \equiv 0,1,2,4 \,\mathrm{mod}\, 8$ then $\xi^{-1}$ is too. Thus these classes are also in the image of $H_*(´\loopinf{+j} \kothom{2n};\bF_2)$. As $´\loopinf{+j} \kothom{2n}$ is an infinite loop map its image on $\bF_2$-homology is closed under multiplication and Dyer--Lashof operations, so the map is surjective on $\bF_2$-homology as claimed.
\end{proof}

\subsection{Away from the prime 2}\label{sec:awayfrom2}

Away from the prime $2$, we can considerably improve the surjectivity result Theorem \ref{integral-surjectivity} by applying work of Madsen and Schlichtkrull \cite{MadsenSchlichtkrull}. All spaces $X$ that occur in the sequel are infinite loop spaces, so they have localisations $X \to X_{(p)}$ which induce algebraic localisation at $p$ on homotopy and homology (see e.g.\ \cite{Sullivan}).

\begin{thm}\label{madsenschlichtkrull}
Let $p$ be an odd prime. There is a loop map $f : (\loopinf{+2}_0 \ko)_{(p)} \to \Omega^\infty_0 \MTSpin(2)_{(p)}$ such that the composition
$$(\loopinf{+2}_0 \ko)_{(p)} \overset{f}\lra \Omega^\infty_0 \MTSpin(2)_{(p)} \overset{\Omega^\infty\lambda_{-2}}\lra (\loopinf{+2}_0 \ko)_{(p)}$$
on $\pi_{4m-2}(-)$ is multiplication by a $p$-local unit times $(2^{2m-1}-1) \cdot \mathrm{Num}(\tfrac{B_m}{2m})$.
\end{thm}
\begin{proof}
Let us write $\Omega^2_0(\bZ \times BO)$ for $\loopinf{+2}_0 \ko$, $\beta : BU \overset{\sim}\to \Omega^2_0(\bZ \times BU)$ for the Bott equivalence and $\beta^{-1}$ for its homotopy inverse. The proof will be based on work of Madsen--Schlichtkrull \cite{MadsenSchlichtkrull}. Their work allows us to construct the following homotopy commutative diagram, where all infinite loop spaces are implicitly localised at $p$, and $k \in \bN$ is chosen such that its residue class generates $(\bZ/p^2)^\times$.
\begin{equation*}
\xymatrix{
 & \bC\bP^\infty \ar[d]^{L - 1} & & \bC\bP^\infty \ar[d]^-{L-1}\\
\Omega_0^2(\bZ \times BO) \ar@{}[dr]|{\mbox{\CircNum{1}}} \ar[r]^-{\rho}\ar[d]_-f  & BU \ar@{}[drr]|{\mbox{\CircNum{2}}} \ar[rr]^-{\beta^{-1} \circ \Omega^2(1-\psi^k) \circ \beta} \ar[d]^-{\Omega \tilde{s}}& &  BU \ar[d]^-{\Omega s}\\
\Omega^\infty_0 \MTSpin(2) \ar@{}[dr]|{\mbox{\CircNum{3}}} \ar[d]_-{\Omega^\infty\lambda_{-2}} \ar[r]^-\simeq & \Omega^\infty_0 \MTSO(2) \ar[d]^-{\beta^{-1} \circ \Omega^\infty(r(t)\cdot\lambda_{-L}^\bC)} \ar[rr]^-{\Omega^\infty\omega}& & Q_0(\bC\bP^\infty_+) \\
\Omega_0^2(\bZ \times BO) \ar[r]^-{\rho} & BU}
\end{equation*}
Here $\MTSO(2) = \Th(-\gamma_2)$ is the Thom spectrum of minus the tautological bundle $\gamma_2 \to BSO(2)$; it receives a spectrum map from $\MTSpin(2)$ coming from the map $B\Spin(2) \to BSO(2)$. Identifying $BSO(2)$ with $\bC\bP^\infty$, and $\gamma_2$ with the tautological complex line bundle $L \to \bC\bP^\infty$, the map $\omega$ is the ``inclusion'' map $\Th(-L) \to \Th(-L \oplus L) = \Sigma^\infty \bC\bP^\infty_+$. The map $\rho = \beta^{-1} \circ \Omega^2 c$ is obtained by looping the complexification map twice then using the inverse Bott equivalence, and $\psi^k$ is the $k$th Adams operation. The diagram is constructed as follows.
\begin{enumerate}[(i)]
\item The square $\CircNum{2}$ is constructed in \cite[Section 7]{MadsenSchlichtkrull}, specifically in the proof of Theorem 7.8 of that paper, and the only property we require of it is the fact that $\Omega s \circ (L-1) \simeq \mathrm{inc} - [1]$, where $\mathrm{inc} : \bC\bP^\infty \to Q_1(\bC\bP^\infty_+)$ is the standard inclusion, and $-[1]$ denotes its translation from the 1 component to the 0 component, cf.\ proof of Lemma 7.5 in \cite{MadsenSchlichtkrull}.

\item The map $f$ is defined so as to make the square $\CircNum{1}$ commute. This uses the fact that the lower map in $\CircNum{1}$ is an equivalence, as we are working at an odd prime so $B\Spin(2) \to BSO(2)$ is a $p$-local equivalence.

\item In the square $\CircNum{3}$ the left hand map is the infinite loop map of $\lambda_{-2}$, the $KO$-theory Thom class of $\MTSpin(2)$.

\item In the square $\CircNum{3}$ the right hand map is the infinite loop map which corresponds under the Bott isomorphism to the class $r(t) \cdot \lambda_{-L}^\bC \in K^{-2}_{(p)}(\MTSO(2))$ in the $p$-local $K$-theory of $\MTSO(2)$, where $\lambda_{-L}^\bC \in K^{-2}_{(p)}(\MTSO(2))$ is the Thom class, and
$$r(t) = \sqrt{1 +t} \in K^0_{(p)}(BSO(2)) = \bZ_{(p)}[[t]] \quad \quad t = L-1$$
is the formal power series expansion of $\sqrt{1 +t}$, whose coefficients lie in $\bZ[\tfrac{1}{2}]$ so are $p$-local integers for any odd prime $p$, so this defines an element in $p$-local $K$-theory. Under the map $B\Spin(2) \to BSO(2)$ the line bundle $L$ pulls back to $L^{\otimes 2}$, and so $r(t)$ pulls back to $L$.

\item The commutativity of square $\CircNum{3}$ is another way of expressing formula \eqref{spin-thom-vs-complex-thom} from the proof of Theorem \ref{rational-surjectivity-htpy-thy}.
\end{enumerate}

We wish to compute the effect of the composition $\Omega^\infty\lambda_{-2} \circ f$ on $\pi_{4m-2}$, where it must be multiplication by some $p$-local integer $A_m \in \bZ_{(p)}$. As the map
$$\rho: \pi_{4m-2}(\Omega^2_0(\bZ \times BO)) \lra \pi_{4m-2}(BU)$$
is an isomorphism, the effect of the composition $\beta^{-1} \circ \Omega^\infty(r(t)\cdot \lambda_{-L}^\bC) \circ \Omega \tilde{s}$ on $\pi_{4m-2}$ must also be multiplication by $A_m$. Because the map $\beta^{-1} \circ \Omega^\infty(r(t)\cdot \lambda_{-L}^\bC) \circ \Omega \tilde{s} : BU \to BU$ is a loop map it sends primitives in $H^*(BU;\bQ)$ to primitives, and so sends the Chern character class $\ch_{2m-1}$ to a multiple of $\ch_{2m-1}$. As classes in $\pi_{4m-2}(BU)$ are detected faithfully by their evaluations against $\ch_{2m-1}$, it follows that the map $\beta^{-1} \circ \Omega^\infty(r(t)\cdot \lambda_{-L}^\bC) \circ \Omega \tilde{s}$ sends $\ch_{2m-1}$ to $A_m \cdot \ch_{2m-1}$. Thus we may compute in rational cohomology.

Let us identify $BSO(2)$ with $\bC\bP^\infty$, so the tautological bundle is given by the universal complex line bundle $L$. Write $x := c_1(L) \in H^2(\bC\bP^\infty;\bZ)$. Firstly, by the same technique used to establish formula \eqref{sinh-series} in the proof of Theorem \ref{rational-surjectivity-htpy-thy} we have
$$\ch(r(t) \cdot \lambda_{-L}^\bC) = e^{x/2} \cdot \frac{x}{1-e^x} \cdot u_{-L} = -\frac{x/2}{\sinh(x/2)} \cdot u_{-L}.$$
and so (as the Bott isomorphism corresponds to the (double) suspension isomorphism under the Chern character) we have
$$(\beta^{-1} \circ \Omega^\infty(r(t)\cdot \lambda_{-L}^\bC))^* \ch_{2m-1} = (-1)^{m+1}\frac{(2^{2m}-2) \cdot B_m}{(2m)! \cdot 2^{2m}} \sigma^*(x^{2m} \cdot u_{-L}),$$
where $\sigma^* : H^*(\MTSO(2)) \to H^*(\Omega^\infty_0\MTSO(2))$ denotes the cohomology suspension. As $\omega^*(x^i) = x^{i+1}\cdot u_{-L}$, we may write the above equation as
\begin{equation}\label{eq:1}
(\beta^{-1} \circ \Omega^\infty(r(t)\cdot \lambda_{-L}^\bC))^* \ch_{2m-1} = (-1)^{m+1}\frac{(2^{2m}-2) \cdot B_m}{(2m)! \cdot 2^{2m}}(\Omega^\infty \omega)^*(\sigma^*(x^{2m-1})).
\end{equation}

Secondly, we wish to compute $(\Omega s)^*(\sigma^*(x^{2m-1}))$. The class $\sigma^*(x^{2m-1})$ is primitive and $\Omega s$ is a loop map, so this class will again be primitive and so a multiple of $\ch_{2m-1}$. To determine which multiple, we may pull it back further to $\bC\bP^\infty$, where
$$(\Omega s \circ (L-1))^*(\sigma^*(x^{2m-1})) = (\mathrm{inc}-[1])^*(\sigma^*(x^{2m-1})) = x^{2m-1}.$$
As $(L-1)^*\ch_{2m-1} = x^{2m-1} / (2m-1)!$, we find that
\begin{equation}\label{eq:2}
(\Omega s)^*(\sigma^*(x^{2m-1})) = (2m-1)! \cdot \ch_{2m-1}.
\end{equation}

Thirdly, we wish to compute $(\beta^{-1} \circ \Omega^2(1-\psi^k) \circ \beta)^* \ch_{2m-1}$. Again, this will be primitive and we may find which multiple of $\ch_{2m-1}$ it is by pulling back further to $\bC\bP^\infty$. We have $\beta^{-1} \circ (1-\psi^k) \circ \beta \simeq 1-k \cdot \psi^k$, and so 
the map $\beta^{-1} \circ \Omega^2(1-\psi^k) \circ \beta \circ (L-1)$ is homotopic to $L - k\cdot L^{\otimes k} +k-1$ and so pulls $\ch_{2m-1}$ back to $\tfrac{1-k^{2m}}{(2m-1)!} \cdot x^{2m-1}$. Thus
\begin{equation}\label{eq:3}
(\beta^{-1} \circ \Omega^2(1-\psi^k) \circ \beta)^* \ch_{2m-1} = (1-k^{2m}) \cdot \ch_{2m-1}.
\end{equation}

Fourthly, the commutativity of $\CircNum{2}$ along with \eqref{eq:1}, \eqref{eq:2}, and \eqref{eq:3} gives
$$(\Omega^\infty(r(t)\cdot \lambda_{-L}^\bC)\circ \Omega \tilde{s})^* \ch_{2m-1} = (-1)^{m+1}\frac{(2^{2m}-2) \cdot B_m}{2m \cdot 2^{2m}} \cdot (1-k^{2m}) \cdot \ch_{2m-1}$$
and so
$$A_m = (2^{2m-1}-1) \cdot \mathrm{Num}(\tfrac{B_m}{2m}) \cdot \left ( (-1)^{m+1} \cdot \frac{1}{2^{2m-1}} \cdot \frac{(1-k^{2m})}{\mathrm{Den}(\tfrac{B_m}{2m})} \right ).$$
Finally, it is well-known that $\frac{(1-k^{2m})}{\mathrm{Den}(B_m/2m)}$ is a $p$-local unit whenever $p$ is an odd prime and $k$ generates $(\bZ/p^2)^\times$. This follows from Lemma 2.12 and Theorem 2.6 of \cite{AdamsII}, together with von Staudt's theorem. Hence $A_m$ is $(2^{2m-1}-1) \cdot \mathrm{Num}(\tfrac{B_m}{2m})$ times a $p$-local unit, as claimed.
\end{proof}

There are two types of consequences of Theorem \ref{madsenschlichtkrull}. The first concerns the index of the groups $J_{d,k} [\tfrac{1}{2}] \subset \pi_{d+k} (\ko)[\tfrac{1}{2}] $, and the other is a splitting result for the index difference. We begin with the homotopy group statements. One derives from Theorem \ref{madsenschlichtkrull}:

\begin{cor}\label{image-in-htpy-away-from2}
The subgroup $J_{2,4m-2}[\tfrac{1}{2}] \subset \pi_{4m}(\ko)[\tfrac{1}{2}]$ has finite index which divides $(2^{2m-1}-1) \cdot \mathrm{Num}(\tfrac{B_m}{2m})$.
\end{cor}

Next, we adopt the convention that $(2^{2m-1}-1)\cdot \mathrm{Num}(\tfrac{B_{m}}{2m})=1$ for $m=0$ and set 
\begin{equation}\label{eq:arm}
A(m, n):=\gcd\left\{ \prod_{i=1}^n(2^{2m_i-1}-1) \cdot \mathrm{Num}\left(\tfrac{B_{m_i}}{2m_i}\right)  \,\Bigg\vert\, m_i \geq 0,  \sum_{i=1}^n m_i = m\right\}.
\end{equation}
Using products and Corollary \ref{image-in-htpy-away-from2}, we find that $J_{2n,4m-2n} [\tfrac{1}{2}] \subset \pi_{4m}(\ko)[\tfrac{1}{2}]$ has finite index dividing $A(m,n)$. Using the maps \eqref{suspension-mt-spectra} we arrive at the following conclusion, which proves Theorem \ref{thm:intro:awayfrom2}.

\begin{cor}\label{index-estimate}
For each $n,m,q \geq 0$, $J_{2n+q,4m-2n-q} [\tfrac{1}{2}] \subset \pi_{4m}(\ko)[\tfrac{1}{2}]$ has finite index dividing $A(m,n)$.
\end{cor}

The strength of this result can be demonstrated by some concrete calculations.
Recall that $\mathrm{Num}(\tfrac{B_{m}}{2m})=\pm 1$ for $m \in \{1,2,3,4,5,7\}$ and $\mathrm{Num}(\tfrac{B_{6}}{12})=-691$, so any prime $p$ dividing $A(m,2)$ in particular divides each of
\begin{align*}
1 \cdot (2^{2m-1}-1) &\cdot \mathrm{Num}(\tfrac{B_{m}}{2m})\\
1 \cdot (2^{2m-3}-1) &\cdot \mathrm{Num}(\tfrac{B_{m-1}}{2(m-1)})\\
7 \cdot (2^{2m-5}-1) &\cdot \mathrm{Num}(\tfrac{B_{m-2}}{2(m-2)})\\
31 \cdot (2^{2m-7}-1) &\cdot \mathrm{Num}(\tfrac{B_{m-3}}{2(m-3)})\\
127 \cdot (2^{2m-9}-1) &\cdot \mathrm{Num}(\tfrac{B_{m-4}}{2(m-4)})\\
7 \cdot 73 \cdot (2^{2m-11}-1) &\cdot \mathrm{Num}(\tfrac{B_{m-5}}{2(m-5)})\\
691 \cdot 23 \cdot 89 \cdot (2^{2m-13}-1) &\cdot \mathrm{Num}(\tfrac{B_{m-6}}{2(m-6)})\\
8191 \cdot (2^{2m-15}-1) &\cdot \mathrm{Num}(\tfrac{B_{m-7}}{2(m-7)})
\end{align*}
for $m \geq 7$, and the appropriately truncated list for $m \leq 6$. Computer calculation (for which we thank Benjamin Young) shows that for $m \leq 45401$ the first four conditions (or the first $m$ for $m < 4$) already imply that $A(m,2)=1$. We therefore have that $J_{4,4m-4}[\tfrac{1}{2}] = \pi_{4m}(\ko)[\tfrac{1}{2}]$ for $m \leq 45401$. Therefore, further taking products, we see that 
$$J_{4\ell+q,4k-q}[\tfrac{1}{2}] = \pi_{4\ell+4k}(\ko) [\tfrac{1}{2}] $$
for $q \geq 0$ and $k \leq 45400 \cdot \ell$.

\subsubsection{A homotopy splitting}

The final of our computational results is the following splitting theorem. To state it, we introduce a condition on prime numbers.

\begin{defn}
Recall that an odd prime number is called \emph{regular} if it does not divide any of the numbers $\mathrm{Num}(\tfrac{B_m}{2m})$. We say that an odd prime is \emph{very regular} if in addition it does not divide any of the numbers $(2^{2m-1}-1)$.
\end{defn}

\begin{remark}
The usual definition of a regular prime $p$ is one which does not divide $\mathrm{Num}(B_m)$ for any $2m \leq p-3$. This is equivalent to not dividing $\mathrm{Num}(\tfrac{B_m}{2m})$ for any $2m \leq p-3$, and is also equivalent to not dividing $\mathrm{Num}(\tfrac{B_m}{2m})$ for any $m$: If such a $p$ divided $\mathrm{Num}(\tfrac{B_n}{2n})$ for some $2n > p-3$ then we cannot have $p-1 \mid 2n$ (as then $p \mid \mathrm{Den}(\tfrac{B_n}{2n})$ by von Staudt's theorem) so we must have $p-1 \nmid 2n$. Thus we may find $0 \neq 2n \equiv 2m \, \mathrm{mod} \, (p-1)$ with $2m \leq p-3$. Kummer's congruence (of $p$-integers) $\tfrac{B_n}{2n} \equiv \tfrac{B_m}{2m} \, \mathrm{mod} \, p$ then contradicts $p$ being regular in the usual sense.
\end{remark}

\begin{remark}
By Fermat's little theorem, if a prime $p$ divides $(2^{2m-1}-1)$ for some $m$, then it also divides $(2^{2m'-1}-1)$ for some $m' \leq \tfrac{p-1}{2}$. Hence it is easy to check the condition of not dividing any $2^{2m-1}-1$, and for example of the regular primes less than 100 the primes 7, 23, 31, 47, 71, 73, 79 and 89 are not very regular, and the remaining primes 3, 5, 11, 13, 17, 19, 29, 41, 43, 53, 61, 83, and 97 are very regular.
\end{remark}

The space $\Riem^+ (S^d)$ is an $H$-space by \cite{WalshH}, and the component $\Riem^{+}_{\round} (S^d)$ of the round metric is a grouplike $H$-space. Therefore for each prime number $p$ there is a $p$-localisation $\Riem^{+}_{\round} (S^d) \to \Riem^{+}_{\round} (S^d)_{(p)}$, which induces the corresponding algebraic localisation on homotopy and homology.

\begin{thm}\label{splitting-theorem}
For $d \geq 6$ and each odd very regular prime, there is a weak homotopy equivalence
\[
\Riem^{+}_{\round} (S^d)_{(p)} \simeq \loopinf{+d+1}_{0} \ko_{(p)} \times F_{(p)}
\]
where $F$ is the homotopy fibre of the index difference map 
\[
\inddiff_{g_\round^d}:\Riem^{+}_{\round} (S^d) \lra \loopinf{+d+1}_{0} \ko. \]
\end{thm}

\begin{proof}
We consider the composition
\begin{equation}\label{splitting-theorem-sequence}
\begin{aligned}
\loopinf{+d+1}_{0} \ko_{(p)} &\stackrel{\eqref{madsenschlichtkrull}}{\lra} \loopinf{+d-1}_{0} \MTSpin(2)_{(p)} \stackrel{}{\lra} \loopinf{+d-5}_{0} \MTSpin(6)_{(p)}\\
&\stackrel{}{\lra} \Omega^{d-6} \Riem^{+}_{\round} (S^6)_{(p)} \overset{L}\lra \Riem^{+}_{\round} (S^{d})_{(p)} \stackrel{\inddiff_{g_\round^d}}{\lra} \loopinf{+d+1}_{0} \ko_{(p)}.
\end{aligned}
\end{equation}
The first map is from Theorem \ref{madsenschlichtkrull} and the second is the map \eqref{suspension-mt-spectra}. The third map is the map from Theorem \ref{factorization-theorem:even-case}, looped $(d-6)$ times. To understand the fourth map $L$, apply Theorem \ref{dimension-increasing}, which gives the left triangle of the (weakly) homotopy commutative diagram
\begin{equation*}
\xymatrix{
\Omega_{g_\round^{m-1}}\Riem^+ (S^{m-1}) \ar[r]^{T}  \ar[dr]_*!/l15pt/-{\labelstyle \Omega \inddiff_{g_\round^{m-1}}}   & \Riem^+ (D^m)_{h_\round^{m-1}}  \ar[d]_{\inddiff_{g_\tor^m}} \ar[r]^{\mu_{g_{\tor}^m}} &  \Riem^+ (S^m) \ar[dl]^*!/r30pt/-{\labelstyle \inddiff_{g_{\dtor}^m} \simeq \inddiff_{g_{\round}^m}}\\
 &  \loopinf{+m+1}\KO. & 
}
\end{equation*}

The (homotopy) commutativity of the right triangle follows from the additivity theorem, more precisely Theorem \ref{additivity-index-difference}. The $(d-6)$-fold iteration of maps along the top of this diagram (and $p$-localisation) yields the map $L$ in \eqref{splitting-theorem-sequence}; the commutativity up to homotopy of this diagram shows that $\inddiff_{g_\round^d} \circ L$ is homotopic to $\Omega^{d-6} \inddiff_{g_\round^6}$. By Theorem \ref{factorization-theorem:even-case}, the composition 
$$\loopinf{+d-5}_{0} \MTSpin(6)_{(p)} \stackrel{}{\lra} \Omega^{d-6} \Riem^{+}_{\round} (S^6)_{(p)} \overset{L}\lra \Riem^{+}_{\round} (S^{d})_{(p)} \stackrel{\inddiff_{g_\round^d}}{\lra} \loopinf{+d+1}_{0} \ko_{(p)}$$
is the same as $\loopinf{+d-5} \kothom{6}$, and so precomposing with the map 
$$\loopinf{+d-1}_{0} \MTSpin(2)_{(p)} \lra \loopinf{+d-5}_{0} \MTSpin(6)_{(p)}$$
gives $\loopinf{+d+1}\kothom{2}$. Therefore, by Theorem \ref{madsenschlichtkrull} the composition \eqref{splitting-theorem-sequence} induces multiplication by a $p$-local unit times $(2^{2m-1}-1) \cdot \mathrm{Num}(\tfrac{B_m}{2m})$, which proves Theorem \ref{thm:intro:split}. If $p$ is in addition very regular then the composition \eqref{splitting-theorem-sequence} is a weak homotopy equivalence, so composing all but the last map gives a composition
$$\loopinf{+d+1}_{0} \ko_{(p)} \stackrel{\theta}{\lra} \Riem^{+}_{\round} (S^{d})_{(p)} \stackrel{\inddiff_{g_\round^d}}{\lra} \loopinf{+d+1}_{0} \ko_{(p)}$$
which is a weak homotopy equivalence. 

Let $j:F \to \Riem^{+}_{\round} (S^{d})$ be the inclusion of the homotopy fibre of the map $\inddiff_{g_\round}$. Implicitly localising all spaces at $p$, it follows that the map 
$$\pi_* (\loopinf{+d+1}_{0} \ko) \oplus \pi_* (F) \stackrel{\theta_*\oplus j_*}{\to} \pi_* ( \Riem^{+}_{\round} (S^{d})) \oplus \pi_* ( \Riem^{+}_{\round} (S^{d}))\stackrel{\Sigma}{\to} \pi_* ( \Riem^{+}_{\round} (S^{d}))$$
is an isomorphism. But $\Riem^{+}_{\round} (S^{d})_{(p)}$ is an $H$-space, and therefore, by the classical Eckmann--Hilton lemma, addition on $ \pi_* ( \Riem^{+}_{\round} (S^{d})_{(p)})$ is induced by the $H$-space multiplication $\mu$ on $ \Riem^{+}_{\round} (S^{d})_{(p)}$. Therefore, the map 
$$\mu \circ (\theta \times j):\loopinf{+d+1}_{0} \ko_{(p)} \times F_{(p)} \lra \Riem^{+}_{\round} (S^{d})_{(p)} $$
is a weak homotopy equivalence.
\end{proof}

\subsubsection{Sharpness}\label{sec:sharp}

Recall that we write $c: \ko \to \ku$ for the complexification map, and that the Chern character map $\ch_{2m} : \pi_{4m}(\ku) \to \bQ$ is an isomorphism onto $\bZ \subset \bQ$. The identity (cf.\ \cite[Proposition 12.5]{SpinGeometry})
$$\ch(c(\lambda_{-2n})) = (-1)^n \hat{A}(\gamma_{2n}) \cdot u_{-2n} \in H^*(\MTSpin(2n);\bQ),$$
where $u_{-2n} \in H^{-2n}(\MTSpin(2n);\bQ)$ is the cohomological Thom class, yields the commutative diagram
\begin{equation*}
\xymatrix{
\pi_{4m-2n}(\MTSpin(2n)) \ar[d]^-h \ar[r]^-{\lambda_{-2n}}& \pi_{4m-2n}(\Sigma^{-2n} \ko) \ar[r]^-{c} & \pi_{4m-2n}(\Sigma^{-2n} \ku) \ar[d]^-{\ch_{2m}}\\
H_{4m-2n}(\MTSpin(2n);\bZ) & H_{4m}(B\Spin(2n);\bZ) \ar[l]^-{\simeq}_-{-\cdot u_{-2n}} \ar[r]^-{(-1)^n\hat{A}_{4m}} & \bQ,
}
\end{equation*}
and the composition along the top of the diagram has image in $\bZ \subset \bQ$. Furthermore, the complexification map $c$ is an isomorphism with $\bZ[\tfrac{1}{2}]$-module coefficients.

Let $p$ be an odd prime number such that $4m < 2p-3$. Then the map $h_{(p)}$ is an isomorphism by the Atiyah--Hirzebruch spectral sequence
$$E^2_{s,t} = H_s(\MTSpin(2n) ; \pi_t(\mathrm{S}^0)_{(p)}) \Longrightarrow \pi_{s+t}(\MTSpin(2n))_{(p)},$$
as 
the first $p$-torsion in the stable homotopy groups of spheres is in degree $2p-3$. This shows that the image of the map
\begin{equation}\label{eq:Ahat}
\hat{A}_{4m} : H_{4m}(B\Spin(2n);\bZ_{(p)}) \lra \bQ
\end{equation}
lies in $\bZ_{(p)} \subset \bQ$. In addition, if we write $j_{d,k}$ for the index of $J_{d,k}$ in $KO_{d+k}$ then its $p$-adic valuation $\nu_p(j_{d,k})$ is equal to that of the index of the image of \eqref{eq:Ahat} inside $\bZ_{(p)}$ (because the map $c$ is an isomorphism with $\bZ_{(p)}$-coefficients). We will proceed to analyse this index.

The map $\prod_{i=1}^{n} B\Spin(2) \to B\Spin(2n)$ induces a surjection on homology with $\bZ[\tfrac{1}{2}]$-module coefficients (as with these coefficients $B\Spin(k) \to BSO(k)$ is an isomorphism), so the image of \eqref{eq:Ahat}
is the same as that of
$$\hat{A}_{4m} : \left (\bigotimes_{i=1}^{n}H_{*}(B\Spin(2);\bZ_{(p)}) \right )_{4m} \lra \bQ.$$
Identify $B\Spin(2)$ with $\bC\bP^\infty$, so that $L^{\otimes 2}$ is the universal spin rank 2 bundle, and let $x = c_1(L)$. Write $x_i$ for the pullback of $x$ to the $i$th factor of $\prod_{i=1}^{n} B\Spin(2)$. Then the $\hat{A}$-class of the direct sum of the pullbacks of the $n$ universal bundles to $\prod_{i=1}^{n} B\Spin(2)$ is, by multiplicativity,
$$\prod_{i=1}^{n}\frac{x_i}{\sinh(x_i)},$$
and if we let $A_\ell := (-1)^\ell \frac{(2^{2\ell}-2)B_\ell}{(2\ell)!}$ be the coefficient of $x^{2\ell}$ in $\frac{x}{\sinh(x)}$ then we find that the image of \eqref{eq:Ahat} is precisely the $\bZ_{(p)}$-linear span
$$\left\langle \prod_{i=1}^{n} A_{m_i} \, \Bigg\vert \, \sum m_i = m\right\rangle_{\bZ_{(p)}} \subset \bQ.$$

Each $A_\ell$ with $\ell \leq m$ is a $p$-integer: our assumption $2\ell \leq 2m < p$ shows that $(2\ell)!$ is a $p$-local unit, and von Staudt's theorem shows that $\mathrm{Den}(\tfrac{B_\ell}{2\ell})$ is a $p$-local unit. (Alternatively this follows from the commutative diagram in the case $n=1$ and localised at $p$, using that composition along the top is a $p$-integer, and that composition along the bottom has image the $\bZ_{(p)}$-linear span of $A_\ell$.) Thus
$$\nu_p(j_{2n,4m-2n}) = \min_{\sum m_i = m}\nu_p\left(\prod_{i=1}^{n} (2^{2m_i-1}-1) \cdot \mathrm{Num}(\tfrac{B_{m_i}}{2m_i})\right).$$
Thus, in terms of the constants $A(m,n)$ defined in \eqref{eq:arm}, there is an identity 
$$j_{2n,4m-2n} = A(m, n) \cdot \frac{A}{B}$$
where (by Corollary \ref{index-estimate}) $A$ is a power of 2, and $B$ is an integer all of whose prime factors $p$ satisfy $p \leq 2m+2$.

\bibliographystyle{plain}
\bibliography{psc}

\end{document}